%% file: main.tex
\documentclass[11 pt]{article}
\usepackage[margin=1in]{geometry}
\input{Compilation_Files/preamble}

\makeatletter
\newcommand{\threevee}{\mathbin{\mathpalette\threevee@{}}}
\newcommand{\threevee@}[2]{
  \sbox0{$\m@th#1\vee$}
  \ooalign{
    \hfil$\m@th#1\vee$\hfil\cr
    \hfil
      \rule[-\dp0]{\dimexpr\wd0/16\relax}
                   {\dimexpr\ht0+\dp0\relax}
    \hfil\cr
  }
}
\makeatother

\title{New bounds for off-diagonal hypergraph Ramsey numbers}

\author{
David Conlon\thanks{School of Mathematics, California Institute of Technology, Pasadena, CA. Email: dconlon@caltech.edu. Research supported by NSF Award DMS-2348859.} \and
Xiaoyu He\thanks{School of Mathematics, Georgia Institute of Technology, Atlanta, GA 30332. Email: xhe399@gatech.edu.}
\and Jiaxi Nie\thanks{School of Mathematics, Georgia Institute of Technology, Atlanta, GA 30332. Email: jnie47@gatech.edu.} \and Logan Post\thanks{School of Mathematics, Georgia Institute of Technology, Atlanta, GA 30332. Email: lpost3@gatech.edu.}}

\begin{document}

\maketitle
\begin{abstract}
We study a conjecture of Conlon, Fox, Gunby, He, Mubayi, Suk, Verstra\"ete, and Yu, which characterizes the $3$-graphs $H$ for which the off-diagonal Ramsey number $r(H,n)$ is polynomial. We prove the conjecture for $3$-graphs with at most five tight components. Additionally, we show that the full conjecture is equivalent to its specialization to linear $3$-graphs. We prove quasipolynomial upper bounds for a class of `forward-doubleable' $3$-graphs, which includes the Fano plane. Finally, we give an exponential lower bound for `non-forward-laminar' $3$-graphs and use this to completely classify the Ramsey growth of a natural family of 3-graphs called chains.

\end{abstract}
\section{Introduction}

Given a 3-uniform hypergraph (3-graph) $H$, the off-diagonal Ramsey number $r(H,n)$ is the minimum $N$ such that any $N$-vertex $3$-graph contains a copy of $H$ or an independent set of size $n$. In this paper, we study the possible growth rates of $r(H,n)$. 

Unlike for ordinary graphs, where $r(G,n)$ is polynomial for every graph $G$, the growth rate of $r(H,n)$ can vary dramatically depending on $H$. It is therefore a natural problem to determine the type of growth rates for $r(H,n)$ for various fixed $3$-graphs $H$. So far we already know examples with Ramsey growth rate polynomial, quasipolynomial, or exponential in $n$~\cite{erdoshajnal,erdos1952combinatorial,quasipoly}. Here as $n\rightarrow\infty$, polynomial growth means $n^{\Theta(1)}$, quasipolynomial means $n^{(\log n)^{\Theta(1)}}$, and exponential means $2^{n^{\Theta(1)}}$. Whether there exists a 3-graph $H$ whose Ramsey growth rate lies between the three categories is an open problem raised in~\cite{quasipoly}. This problem is also closely related to the polynomial-to-exponential transition problem of Erd\H{o}s and Hajnal \cite{erdoshajnal} in 1972, resolved in the affirmative by Ascoli, He, and Yu \cite{polyexp}. 

Our main goal is to make progress towards the central \cref{conj:iteratedtripartite} below of Conlon, Fox, Gunby, He, Mubayi, Suk, Verstra\"ete, and Yu. To our knowledge, it was first posed as a question by Conlon \cite{Conlonquestion} in 2015.
\begin{conjecture}[{\cite[2025]{CFGHMSVY}}]\label{conj:iteratedtripartite}
For a 3-graph $H$, there exists a constant $c$ depending only on $H$ such that $r(H, n) \leq n^c$ for all $n$ if and only if $H$ is a subgraph of an iterated blow-up of an edge.
\end{conjecture}

To make sense of this, we first define iterated blow-ups informally. A \textit{blow-up of an edge} is a complete tripartite 3-graph, whereas an \textit{iterated blow-up of an edge} is formed recursively by forming a blow-up of an edge, then placing another iterated blow-up of an edge inside one or more of the parts. We say that a 3-graph $H$ is \textit{iterated tripartite} if it is a subgraph of an iterated blow-up of an edge.

The easier direction of \cref{conj:iteratedtripartite} is a classical result of Erd\H{o}s and Hajnal~\cite{erdoshajnal} (also see \cref{cor:setramseyblow-up}): any iterated blow-up of an edge has polynomial off-diagonal Ramsey number. Thus, the substance of \cref{conj:iteratedtripartite} is that every graph which is \textit{not} iterated tripartite (NIT) has superpolynomial Ramsey number. 

Recent results provide evidence for this conjecture. First, Conlon et al.~\cite{linearlowerbound} disproved the folklore conjecture that every linear $3$-graph has polynomial off-diagonal Ramsey number, which is necessary for \cref{conj:iteratedtripartite} since there exist NIT linear $3$-graphs, for example the Fano plane (Equation \ref{eq:fano}, Figure \ref{fig:fano}). More directly, Conlon et al. \cite{CFGHMSVY} provide evidence for the conjecture by proving it for any graph $H$ with at most two tight components.

\begin{definition}A $3$-graph $H$ is \textit{tightly connected} if for any two edges $e,f\in E(H)$, there is a sequence of edge $e=e_0,e_1,\ldots, e_\ell=f$ with $|e_{i-1}\cap e_i|=2$ for all $i\in [\ell]$. A \textit{tight component} of a 3-graph $H$ is an inclusion-maximal tightly connected subgraph, viewed as a subset of $E(H)$.
\end{definition}

\begin{theorem}[{\cite[Thm.\ 1.2]{CFGHMSVY}}]\label{thm:tightlyconnected}
If a 3-graph $H$ is tightly connected and is not iterated tripartite, then $r(H,K_n^{(3)})\geq 2^{\Omega(n^{2/3})}$.
\end{theorem}

\begin{theorem}[{\cite[Thm.\ 1.3]{CFGHMSVY}}]\label{thm:2tightcomponents}
If a 3-graph $H$ has at most two tight components and is not iterated tripartite, then $r(H,K_n^{(3)})\geq 2^{\Omega(\log^2 n)}$.
\end{theorem}

Note that the lower bounds for tightly connected 3-graphs are exponential, while those for 3-graphs with two tight components are quasipolynomial. The existing proof of Theorem~\ref{thm:2tightcomponents} is fairly technical, and our methods yield a substantially simpler proof. 

\subsection{Results and Outline}

The results stated in the introduction are the most interesting special cases of significantly more general results. The main technical results of this paper are Lemmas~\ref{lem:quotient} and \ref{lem:quotientonegraph} and Theorems \ref{thm:fwddoubling}, \ref{thm:fwddoublingvsBn}, and \ref{thm:forwardlaminar}. These rely on some technical definitions, so for the remainder of the introduction we present primarily special cases which exemplify the results. The three Sections \ref{sec:quasilower}, \ref{sec:upperbounds}, and \ref{sec:explower} are roughly disjoint in scope, so we partition the outline accordingly.

\subsubsection{Quasipolynomial lower bounds and reduction to linear}\label{sec:intro_quasilower}

The results of this section are proved in \cref{sec:quasilower}. Our main contribution to \cref{conj:iteratedtripartite} is the following reduction. 

\begin{theorem}\label{thm:reductiontolinear}If there is a nonempty family $\cH$ of NIT $3$-graphs with at most $t$ tight components satisfying $r(\cH,n)=n^{O(1)}$, then there is a family $\cL$ of linear NIT $3$-graphs with at most $t$ edges satisfying $r(\cL,n)=n^{O(1)}$.
\end{theorem}

This is restated and proven in \cref{sec:reductiontolinear} using the quotient blow-up machinery of \cref{sec:quasilower}, and we define the set Ramsey numbers formally in Equation \ref{eq:familyramsey}. Given a 3-graph $R$ with relatively small independence number avoiding $\cL$, we produce a ``sparsified iterated blow-up'' $\cG_d(R)$ of $R$ that also has relatively small independence number and avoids $\cH$. We use this construction to obtain quantitative superpolynomial lower bounds for 3-graphs $H$ with at most five tight components.

The authors of \cite{linearlowerbound} mentioned (and disproved) a folklore conjecture that linear hypergraphs always have polynomial off-diagonal Ramsey number. In fact \cref{thm:reductiontolinear} shows that linear hypergraphs are the hardest case and that \cref{conj:iteratedtripartite} reduces to a partial opposite of this folklore belief.

In \cref{prop:5edgetripartite} we check that the Fano plane minus an edge is the smallest NIT linear hypergraph, given by $V(F^-)=[7]$ and $E(F^-)=\{123,145,167,247,346,357\}$. Thus, from \cref{thm:reductiontolinear} it follows that \cref{conj:iteratedtripartite} is true for $3$-graphs with at most five tight components. Below, we catalogue the state-of-the-art lower bounds for NIT $3$-graphs with a small number of tight components, adding new bounds for three, four and five.

\begin{theorem}\label{thm:4tightcomponents}Let $H_t$ be a non-iterated tripartite $3$-graph with at most $t$ tight components. Then 
\begin{align}
r(H_1,n)&\geq \exp(\Omega(n^{2/3}))\tag{{\cite[Thm.\ 2.1]{CFGHMSVY}}}\\
r(H_2,n)&\geq \exp(\Omega(\log^2n))\tag{{\cite[Thm.\ 3.1]{CFGHMSVY}}}\\
r(H_4,n)&\geq \exp(\Omega(\log^2n/\log\log n))\\
r(H_5,n)&\geq \exp(\Omega(\log^{3/2}n/(\log\log n)^{1/2})).
\end{align}
\end{theorem}
In particular, we give a new proof for the bound on $r(H_2,n)$ which is substantially simpler than the proof given in \cite{CFGHMSVY}. This is deduced in \cref{sec:fivetightcomponents}. Finally, we prove the following general reduction.

\begin{theorem}\label{thm:singlelinear}If there is an NIT 3-graph $H$ satisfying $r(H,n)=n^{O(1)}$, then there is a linear NIT $3$-graph $L$ with $|E(L)|\leq 7|E(H)|$ satisfying $r(L,n)=n^{O(1)}$. 
\end{theorem}
Equivalently, this means that if \cref{conj:iteratedtripartite} is true for linear $3$-graphs with at most $7t$ edges, then it is true for all $3$-graphs with at most $t$ edges. This is proven in \cref{sec:reductiontolinear}.

\subsubsection{Quasipolynomial upper bounds} \label{sec:intro_quasiupper}

Previously, it was shown in \cite{quasipoly} that there exists a $6$-vertex, $5$-edge 3-graph with strictly quasipolynomial Ramsey number. This lies in contrast with the result of Ascoli, He, and Yu \cite{polyexp} that there is a polynomial-to-exponential transition for off-diagonal hypergraph Ramsey numbers of order-size pairs. In \cref{sec:upperbounds}, we prove quasipolynomial upper bounds for a new class of hypergraphs including the Fano plane. Let $F$ be the \textit{Fano plane} with vertex set $\{1,\ldots,7\}$ and edge set
\begin{equation}\label{eq:fano}
E(F)=\{123,145,167,247,346,357,256\}
\end{equation}
Note that $E(F)$ is edge-transitive, so the \textit{Fano plane minus an edge} $F^-$ is uniquely determined up to isomorphism. We prove the following quasipolynomial upper bound on this hypergraph as a special case of \cref{cor:tridoubling}.
\begin{theorem}\label{thm:fanominus} The Fano plane minus an edge satisfies $r(F^-,n)=\exp(O(\log^2n))$.
\end{theorem}

In addition, we introduce a new embedding tool called \textit{quasi-dense pairs}, roughly analogous to lower-regular pairs in the graph regularity method, which can be used to embed a broad class of $3$-graphs into $3$-graphs with small independence number, including the Fano plane itself. The following is a special case of \cref{thm:fwddoubling}.

\begin{theorem}\label{thm:fanoplane}
The Fano plane satisfies $r(F,n)=\exp(\log^{\Theta(1)}n)$.
\end{theorem}

Define the complete bipartite graph $B_{n,n}$ as the 3-graph whose vertex set consists of two parts of size $n$, and which contains all edges intersecting both parts. Our techniques for \cref{thm:fanominus} and \cref{thm:fanoplane} apply more directly to Ramsey numbers against $B_{n,n}$, for which we can prove polynomial bounds such as the following special case of \cref{thm:fwddoublingvsBn}. 

\begin{theorem}\label{thm:fanovsBn}The Fano plane satisfies $r(F,B_{n,n})=n^{O(1)}$.
\end{theorem}

This result was announced without proof by Samotij and Mubayi \cite{Conlonquestion} in 2015, but to our knowledge never published.

It is remarkable that using the same methods we obtained different upper bounds for $r(F,n)$ and $r(F,B_{n,n})$. A similar dichotomy was proved to exist in \cite{quasipoly}, wherein it was shown that the bicone $H_2$ satisfies $r(H_2,n)=\exp(\Theta(\log^2n))$ yet $r(H_2,B_{n,n})=n^{O(1)}$. In fact, as another corollary of \cref{thm:fwddoublingvsBn}, we improve Theorem 1.3 in \cite{quasipoly}.

\begin{theorem}\label{thm:iteratedbicone}
For all $t\geq 1$, the iterated bicone $H_t$ defined in \cite[Defn.\ 2.2]{quasipoly} satisfies $r(H_t,B_{n,n})=n^{O(1)}$.
\end{theorem}

\subsubsection{Exponential lower bounds and the chain} \label{sec:intro_explower} 

Towards understanding which hypergraphs have exponential Ramsey number, we provide a construction which gives an exponential lower bound to what we call \textit{non-forward-laminar} 3-graphs, defined precisely in \cref{sec:explower}. As one illustrative example, we show that off-diagonal Ramsey numbers are exponential for ``chains,'' see \cref{fig:chain}. 

\begin{theorem}\label{thm:girth} For all $t\geq 3$, there is a $3$-graph $H=\rm{CH}_{t}$ such that the union of any $t-1$ tight components of $H$ is iterated tripartite, but $r(H,n)= \exp(\Omega(n^{1/5}))$.
\end{theorem}

Our witness for \cref{thm:girth} was suggested as an open problem in \cite{CFGHMSVY}. For $t\geq 2$, we denote the \textbf{chain} of length $t$ by $\rm{CH}_t$, a $3$-graph with $V(\rm{CH}_t)=\{a_i,b_i\mid 0\leq i\leq t-1\}$ and
\[E(\rm{CH}_t)=\{a_ib_ia_{i+1},a_ib_ib_{i+1}\mid 0\leq i\leq t-1\},\]
where $i+1$ is taken modulo $t$. Note that $\rm{CH}_2=K_4^{(3)}$, for which better bounds are known (\cite{ConlonFoxSudakov2010}). For $t\ge 3$, $\rm{CH}_t$ has $t$ tight components and is not iterated tripartite, but removing any one component yields an iterated tripartite 3-graph. Furthermore, combined with techniques from other sections, we can prove the following characterization theorem for subgraphs of chains.

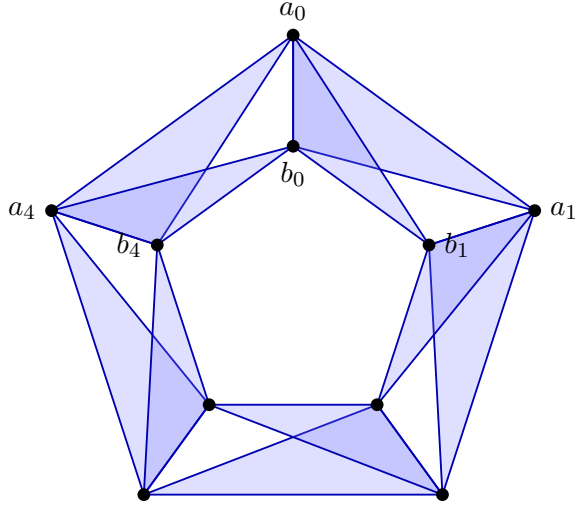
\begin{figure}
\centering
\input{tikz_diagrams/temp}
\caption{The chain $\rm{CH}_5$. There are five tight components, each with two edges.}
\label{fig:chain}
\end{figure}

\begin{theorem}\label{thm:chainclassification}
For $t\geq 3$, let $\rm{CH}_t$ be the chain of length $t$. Let $\rm{CH}_t^-$ be obtained by deleting one edge, and let $\rm{CH}_t^{--}$ be obtained by deleting at least two edges. Then
\begin{align}
r(\rm{CH}_t,n)&=\exp(\Omega(n^{1/5})),\label{eq:chainthm}\\
r(\rm{CH}_t^-,n)&=\exp(\Theta(\log^2n)),\label{eq:chain-thm}\\
r(\rm{CH}_t^{--},n)
&=n^{O(1)}.\label{eq:chain--thm}
\end{align}
\end{theorem}

As another example of the forward-laminar family, we can prove exponential lower bounds for the following type of hypergraphs.

\begin{theorem}\label{thm:spanningcomponents} If $H$ is an NIT $3$-graph such that every tight component spans $V(H)$, then $r(H,n)=\exp(\Omega(n^{1/5})$.
\end{theorem}

Finally, we have a somewhat surprising dichotomy by combining our results. Define the half-bipartite 3-graph $S_{n,n}$ as a $3$-graph with two parts $V_0,V_1$ of order $n$, including every edge with two ends in $V_0$ and one end in $V_1$. We prove the following in \cref{sec:twobicone}.
\begin{theorem}\label{thm:SnvsKn}There is an $8$-vertex, $8$-edge $3$-graph $H$ with both 
\[r(H,S_{n,n})=n^{O(1)}\quad \text{and} \quad r(H,n)=\exp(\Omega(n^{1/5})).\]
\end{theorem}

\section{Quasipolynomial lower bounds and reduction to linear}\label{sec:quasilower}

The main results in the section use \textit{set Ramsey numbers}. For a family $\cF$ of $3$-graphs, we define
\begin{equation}\label{eq:familyramsey}
r(\cF,n) \coloneqq \min\bigg(N\ \bigg|\  \begin{tabular}{l}
Any $N$-vertex $3$-graph $G$ with no independent\\
set of size $n$ contains a member of $\cF$
\end{tabular}\bigg).
\end{equation}
By definition, if $\cF=\emptyset$ then $r(\cF,n)=+\infty$ for all $n\geq 3$, since the complete graph $K_r^{(3)}$ contains no member of $\cF$ for all $r\in \N$. Also by definition, $r(\cF,n)=r(\cF^\minn,n)$ where $\cF^\minn$ is the set of subgraph-minimal elements of $F$. This is because whenever $G$ contains $F\in \cF$ as a subgraph, it also contains a minimal $F'\in \cF^\minn$.

In this section, we prove the following result and use it to deduce Theorems \ref{thm:2tightcomponents}, \ref{thm:4tightcomponents}, and \ref{thm:reductiontolinear}. 

\begin{lemma}\label{lem:quotient}
Let $\cH$ be a family of nontrivial 3-graphs and $\cF$ be the family of nontrivial tight quotients of members of $\cH$. If $r\geq 0$ is sufficiently large and $r(\cF,n)> r$, then
\[
r(\cH,r)\geq \exp\bigg(\frac{1}{36}\frac{\log^2r}{\log n}\bigg).
\]
\end{lemma}
In particular, if $r(\cF,n)$ is superpolynomial then $\log r\gg \log n$ so $\frac{\log^2r}{\log n}\gg \log r$, and thus $r(\cH,n)$ is also superpolynomial. We also have the following

\begin{corollary}\label{cor:notquotients} If $\cH$ is a family of $3$-graphs such that none admits a chain of $t\geq 1$ nontrivial tight quotients, then
\[
r(\cH,n)\geq \exp(\Omega(\log^{1+1/t}n))
\]
\end{corollary}

More simply, this gives the following. 

\begin{corollary}\label{cor:noquotient} If $\cH$ admits no nontrivial tight quotient, then $r(\cH,n)\geq \exp(\Omega(\log^2 n))$.
\end{corollary}

To obtain the reduction to linear hypergraphs (\cref{thm:reductiontolinear}) we require further lemmas which we defer to \cref{sec:reductiontolinear}.

\subsubsection{Section organization.}

First, in \cref{sec:quotientdefinitions} we define \textit{tight quotients} and all necessary terms to understand the results above. Next, in \cref{sec:proofofquotientlemma} we prove the key quotient lemma (\ref{lem:quotient}). In \cref{sec:fivetightcomponents} we use the quotient lemma and further properties of the quotient operation to prove \cref{thm:fivetightcomponents}, which strengthens \ref{thm:4tightcomponents}. Next, in \cref{sec:single3-graphreduction} we prove the standalone \cref{lem:quotientonegraph}, which gives an analogous result to \cref{lem:quotient} for individual $3$-graph Ramsey numbers instead of sets. This lemma requires an additional `robustness' condition so it is not used to deduce any headline results. Finally, in \cref{sec:reductiontolinear}, we prove \cref{thm:reductiontolinear} using the quotient machinery and we prove \cref{thm:singlelinear} using a straightforward construction.

\subsection{Quotient definitions}\label{sec:quotientdefinitions}

We shall define a quotient operation which takes an input $3$-graph $H$ and identifies sets of vertices such that some structure of $H$ is preserved and also such that all tight components are contracted to a single edge (or vertex). Then, we'll show that superpolynomiality is preserved under this reduction. To that end, consider the following.

\begin{definition}\label{defn:quotient} Let $H$ be a $3$-graph and $\cP=\{P_1,\ldots, P_m\}$ be a partition of $V(H)$. Viewing $\cP$ as a set of equivalence classes, we define the \textbf{inclusion map} $[\cdot]_\cP$ where $[v]_\cP= P_i$ for the unique $P_i\ni v$. When $\cP$ is clear we simply write $[v]$. Define the \textbf{quotient hypergraph} $H/\cP$ as the hypergraph with vertex set $\cP$ and edge set
\[
E(H/\cP)=\big\{[u][v][w]\ \big|\ uvw\in E(H),\ |\{[u],[v],[w]\}|\geq 2\}
\]
Note that $H/\cP$ may have edges of size both $2$ and $3$.
\begin{enumerate}[(I)]
\item\label{itm:properquotient} A quotient $H/\cP$ is \textbf{proper} if it is a $3$-graph, i.e.\ for all edges $uvw\in E(H)$, either $[u]=[v]=[w]$ or $[u],[v],[w]$ are all distinct. 
\item \label{itm:tightquotient} A proper quotient $H/\cP$ is \textbf{tight} if for each tight component $H'\subseteq E(H)$ and edges $uvw, u'v'w' \in H'$, we have $\{[u],[v],[w]\}=\{[u'],[v'],[w']\}$.
\end{enumerate}
As $\cP=\{V(H)\}$ is always a tight quotient of $H$, we say $H/\cP$ is \textbf{nontrivial} if $|\cP|\geq 2$. More generally, a 3-graph $H$ is nontrivial if $|H|\geq 2$. A \textbf{chain} of tight/proper quotients of $H$ of length $t$ is a sequence of $3$-graphs $H=H_0,H_1,\ldots, H_t$ such that for all $i\in [t]$, $H_i$ is a tight/proper quotient of $H_{i-1}$, and all are nontrivial.
\begin{observation}
Note that \ref{itm:tightquotient} is equivalent to the condition that for any edges $uvw,uvw'\in E(H)$ with a common pair, we have $[w]=[w']$. If this simpler condition holds, then the set equality propagates through tight paths in any tight component.
\end{observation}
\end{definition}

Observe that if $F$ is a tight quotient of a 3-graph $H$ with $k$ tight components, then by \ref{itm:tightquotient} $F$ has at most $k$ edges. In particular, if $H$ is not linear then $|E(F)|<|E(H)|$. \cref{lem:quotientpreservesNIT} will exemplify the idea that a quotient of $H$ is like an inverse of an iterated blow-up, and preserves some structure.

We observe that there are many 3-graphs $H$ which do not admit any nontrivial tight quotient. For example, if $H$ is tightly connected and spanning and $H/\cP$ is a tight quotient of $H$, then $\cP$ is either a proper tripartition of $H$, or else $\cP$ is trivial. Thus, if $H$ is not tripartite, then $H$ has only the trivial tight quotient.

\subsection{Proof of Lemma \ref{lem:quotient}}\label{sec:proofofquotientlemma}

The lower bound construction uses iterated blow-ups. For a 3-graph $R$ on vertex set $[r]$, denote by $R^d$ the \textit{$d$-times iterated blow-up} of $R$. We put $R^1=R$ and $R^d$ is obtained from $R^{d-1}$ by replacing vertex with a copy of $R$. Equivalently, $R^d$ is obtained from $R$ by replacing each vertex by $r^{d-1}$ copies of itself, then inducing a copy of $R^{d-1}$ on each blown-up set. Formally, we let $V(R^d)=[r]^d$ and define $E(R^d)$ in the following way:

For distinct tuples $u=(u_1,\ldots, u_d)$ and $v=(v_1,\ldots, v_d)$, define $\delta(u,v)$ as the smallest index $i\in [d]$ such that $u_i\neq v_i$. We say that a triple $\{u,v,w\}$ is $\ell$-level if $\delta(u,v)= \delta(v,w)=\delta(u,w)=\ell$. Now for all $u,v,w\in V(R^d)$ we put $uvw\in E(R^d)$ if and only if there exists $\ell\in [d]$ such that $uvw$ is $\ell$-level and $u_\ell v_\ell w_\ell \in E(R)$. 

Now we see that a graph $H$ is \textit{iterated tripartite} if there is some $d$ such that $H$ is isomorphic to a subgraph of $(K_3^{(3)})^d$. We call the isomorphism $\chi:V(H)\to \{1,2,3\}^d$ a \textit{proper iterated coloring}.

\begin{construction}
For a $3$-graph $R$ on vertex set $[r]$, define $\cG_d(R)$ as a \textit{random sparsification} of the iterated blow-up $R^d$ in the following way. Let $V=V(\cG_d(R))=[r]^d$ and let $\phi:\binom V2\to [r]$ be uniformly random, chosen independently for each pair. For all $u,v,w\in V$, we include $uvw\in E(\cG_d(R))$ if and only if
\begin{enumerate}[(i)]
\item $uvw\in E(R^d)$, that is, there exists $\ell\in[d]$ such $uvw$ is $\ell$-level and $u_\ell v_\ell w_\ell \in E(R)$.
\item\label{itm:phivalue} With $\ell$ as above, we have $\phi(uv)=w_\ell$, $\phi(uw)=v_\ell$, and $\phi(vw)=u_\ell$.
\end{enumerate}
\end{construction}

First, we describe how this construction reduces the problem from $H$ to its tight quotients.
\smallskip

\begin{lemma}\label{lem:G(R)tightminor} Let $H$ and $R$ be $3$-graphs. If $H\subseteq \cG_d(R)$, then $R$ contains subgraph isomorphic to a nontrivial tight quotient of $H$.
\end{lemma}
\begin{proof}
Let $\ell$ be the minimum level of any edge $e^*\in E(H)$, and define the partition $\cP=\{P_x\mid x\in V(R)\}$ by letting $P_x=\{v\in V(H)\mid v_\ell=x\}$, so $[u]=[v]$ if and only if $u_\ell=v_\ell$. By minimality of $\ell$, every edge $uvw\in E(H)$ is either $\ell$-level in which case $u_\ell v_\ell w_\ell\in E(R)$, or else $uvw$ is $\ell'>\ell$ level, in which case $u_\ell=v_\ell=w_\ell$; thus $H/\cP$ is a proper quotient of $H$. 

If $uvw,uvw'\in E(H)$ are two edges with a common pair $uv$, then since each is level, they are at the same level $\ell'=\delta(u,v)$. If $\ell'>\ell$ then $w_\ell=u_\ell=v_\ell=w_{\ell}'$, and if $\ell'=\ell$ then by \ref{itm:phivalue} we have $w_\ell=\phi(uv)=w_\ell'$;  
thus $F=H/\cP$ is a tight quotient of $H$. Since $e^*$ is $\ell$-level, we conclude that $F$ is nontrivial and the embedding $v\mapsto v_\ell$ yields an isomorphic copy of $F$ in $R$.
\end{proof}

Next, we check that $\cG_d(R)$ avoids large independent sets and thus is a good Ramsey construction.

\smallskip

\begin{lemma}\label{lem:G(r)indset} If $|R|=r$ and $\alpha(R)<n$, then with positive probability $\alpha(\cG_d(R))< 24n^{2d}r^3d\log r$.
\end{lemma}
\begin{proof}
First, observe that $\alpha(R^d)=\alpha(R)^d<n^d$. Let $S\subseteq V(R^d)$ have size $m>2n^d$. By Tur\'an bounds for complete hypergraphs (or sampling with alterations), $S$ must contain at least $m^3/8n^{2d}$ edges of $R^d$. Next, since codegrees are at most $m$ we greedily pick out a set $\cS\subseteq E(R^d)$ of $m^2/24n^{2d}$ pair-disjoint triples in $S$; i.e. $\cS$ forms a linear hypergraph. Each triple in $\cS$ is included as an edge in $\cG_d(R)$ with probability $r^{-3}$ depending only on $\phi$, so they are mutually independent. Thus, taking a union bound over all choices of $S$, we have
\[
\PP[\alpha(\cG_d(R))\geq m]\leq \binom{r^d}m(1-r^{-3})^{m^2/24a^{2d}}< \exp(m d\log r-m^2/24a^{2d}r^3)\leq 1
\]
as long as $m\geq  24a^{2d}r^3d\log r$
\end{proof}

Together, these prove \cref{lem:quotient}.

\begin{proof}[Proof of \cref{lem:quotient}]
Let $R$ be an $r$-vertex $3$-graph with no member of $\cF$ and no independent set of size $n$. Let $d=\ceil{\frac{\log r}{\log n}}$, and let $G=\cG_d(R)$. Observe that
\[
|V(G)|=r^d \geq \exp\Big(\frac{\log r}{\log n}\log r\Big).
\]
By the contrapositive of \cref{lem:G(R)tightminor}, $G$ contains no member of $\cH$. By \cref{lem:G(r)indset}, $G$ contains no independent set of size
\[
24n^{2d}r^3d\log r = \exp\Big((1+o(1))\big(2d\log n+3\log r\big)\Big)\leq \exp\Big((1+o(1))\big(2\log r+3\log r\big)\Big)< r^6,
\]
when $r$ is sufficiently large. Thus, $G$ is a witness for $r(\cH,r^6)\geq \exp\big(\frac{\log^2r}{\log n}\big)$. Plugging in $r^{1/6}$ gives the result.
\end{proof}

From this, we can quickly deduce Corollaries \ref{cor:noquotient} and \ref{cor:notquotients} by plugging in the empty family. Note that \ref{cor:noquotient} is the $t=1$ case of \ref{cor:notquotients}.

\begin{proof}[Proof of \cref{cor:noquotient}]
Let $\cH$ be a family admitting no tight quotient. Then letting $\cF=\emptyset$, we have $r(\cF,3)>r$ for all $r$ (that is, our machinery blows up the graph $R=K_r$). Then for $r$ sufficiently large, \cref{lem:quotient} implies
\[
r(\cH,r)\geq \exp\Big(\frac{1}{36}\frac{\log^2 r}{\log 3}\Big) = \exp(\Omega(\log^2 r)).\qedhere
\]
\end{proof}

\begin{proof}[Proof of \cref{cor:notquotients}]
We proceed by induction on $t$, where $t=1$ is handled above. For the inductive case, let $\cH$ be a family admitting no chain of length $t$, and let $\cF$ be the set of nontrivial tight quotients of $\cH$. Then by definition $\cF$ admits no chain of length $t-1$. By induction, $r(\cF,n)\geq r= \exp(\Theta(\log^{1+1/(t-1)}n))$, and inverting this, $\log n=\Theta(\log^{1-1/t}r)$. Then by \cref{lem:quotient},
\[
r(\cH,r)\geq \exp\bigg(\Omega\bigg(\frac{\log^2r}{\log n}\bigg)\bigg) = \exp(\Omega(\log^{1+1/t}r)),
\]
as desired.
\end{proof}

\subsection{At most five tight components}\label{sec:fivetightcomponents}

The following lemma exemplifies how the quotient operation preserves the NIT structure of a minimal $3$-graph. 

\begin{lemma}\label{lem:quotientpreservesNIT}
If $H$ is minimal NIT and $F$ is a nontrivial proper quotient of $H$, then $H$ is contained in an iterated blow-up of $F$. In particular, $F$ is NIT.
\end{lemma}
\begin{proof}
Let $xyz$ be an edge of $F=H/\cP$. By minimality of $H$, for each $P\in \cP$ the induced $3$-graph $H[P]$ is iterated tripartite; let $\chi:V(H)\to \{x,y,z\}^{\ell}$ induce a proper iterated coloring on each part $P\in \cP$. Then let $\chi':V(H)\to V(F)^{\ell+1}$ be defined $\chi'(v)=([v],\chi(v))$. Then every edge $uvw\in E(H)$ either has $[u][v][w]\in E(F)$ in which case $uvw$ is $1$-level with respect to $\chi'$, or else $[u]=[v]=[w]$ in which case there is some $d\in [\ell]$ such that $uvw$ is $d$-level with respect to $\chi$ and $d+1$-level with respect to $\chi'$, and $\chi'_{d+1}(u)\chi'_{d+1}(v)\chi'_{d+1}(w)=xyz\in E(F)$.

The corollary follows from the fact that iterated-iterated-blow-ups are still iterated blow-ups. More explicitly, if $F$ has proper iterated coloring $\chi_1:\cP\to \{x,y,z\}^{\ell_1}$, then the coloring $\chi_2:V(H)\to \{x,y,z\}^{\ell_1+\ell}$ given by $\chi_2(v)=(\chi_1([v]),\chi(v))$ is proper iterated coloring of $H$.
\end{proof}

Using \cref{lem:quotient}, we can prove lower bounds for all 3-graphs with at most five tight components. In particular, we provide a new proof of \cref{thm:2tightcomponents} and complete the following strengthening of \cref{thm:4tightcomponents}. 

\begin{theorem}\label{thm:fivetightcomponents} Let $\cH_t$ be the family of NIT $3$-graphs with at most $t$ tight components. We have
\begin{align}
r(\cH_1,n)&\geq \exp(\Omega(n^{2/3}))\tag{{\cite[Thm.\ 2.1]{CFGHMSVY}}}\\
r(\cH_2,n)&\geq \exp(\Omega(\log^2n))\tag{{\cite[Thm.\ 3.1]{CFGHMSVY}}}\\
r(\cH_4,n)&\geq \exp(\Omega(\log^2n/\log\log n))\\
r(\cH_5,n)&\geq \exp(\Omega(\log^{3/2}n/(\log\log n)^{1/2})).
\end{align}
\end{theorem}

We do not re-prove the bound for $\cH_1$; the proof of this is relatively simple in \cite{CFGHMSVY}. Note that actually \cref{thm:forwardlaminar} implies $r(\cH_1,n)\geq \exp(\Omega(n^{1/5}))$ which suffices for our recursion, making the arguments self-contained. In each of the following, we use that $r(\cH,n)=r(\cH^\minn,n)$ where $\cH^\minn$ is the subfamily of subgraph-minimal elements of $\cH$. This follows from the definition of the set Ramsey numbers in Equation \eqref{eq:familyramsey}. For the next part, we use the following

\begin{observation}If $F$ is a $3$-graph with at most $2$ edges then $F$ is tripartite.
\end{observation}

\begin{proof}[Proof of \cref{thm:fivetightcomponents} for $\cH_2$]
We apply \cref{cor:noquotient} to the family $\cH_2^\minn$. Any nontrivial tight quotient $F$ of some 3-graph $H\in \cH_2^\minn$ will have at most two edges by definition, and by \cref{lem:quotientpreservesNIT} also $F$ is NIT. This contradicts the observation above, so $\cH_2^\minn$ admits no such $F$. Then \cref{cor:noquotient} implies
\[
r(\cH_2,n)= r(\cH_2^\minn,n)\geq \exp(\Omega(\log^2n)).\qedhere
\]
\end{proof}

\begin{proposition}\label{prop:4edgetripartite}
If $F$ is an NIT $3$-graph with at most four edges, then some tight component of $F$ is NIT.
\end{proposition}
\begin{proof}[Proof of \cref{thm:fivetightcomponents} for $\cH_4$]
We apply \cref{lem:quotient} to the family $\cH_4^\minn$. Let $\cF$ be the family of nontrivial tight quotients of members of $\cH_4^\minn$. For each $F\in \cF$, we know that $F$ has at most four edges. By \cref{lem:quotientpreservesNIT} also $F$ is NIT. By \cref{prop:4edgetripartite}, $F$ must contain a member of $\cH_1$ as a subgraph. Thus, 
\[r(\cF,n)\geq r(\cH_1,n)\geq r= \exp(\Omega(n^{2/3})).
\]
Inverting this, $\log n=O(\log\log r)$, so by \cref{lem:quotient},
\[
r(\cH_4,r)=r(\cH_4^\minn,r)\geq \exp\bigg(\Omega\bigg(\frac{\log^2r}{\log n}\bigg)\bigg)=\exp(\log^2r/\log\log r),
\]
as desired.
\end{proof}

\begin{proposition}\label{prop:5edgetripartite}
If $F$ is a linear $3$-graph with at most five edges then $F$ is iterated tripartite
\end{proposition}

\begin{proof}[Proof of \cref{thm:fivetightcomponents} for $\cH_5$.] We apply \cref{lem:quotient} to the family $\cH_5^\minn$. Let $\cF$ be the family of nontrivial tight quotients of members of $\cH_5^\minn$. For each $F\in \cF$, we know that $F$ has at most 5 edges. By \cref{lem:quotientpreservesNIT} also $F$ is NIT, and by the contrapositive of \cref{prop:5edgetripartite}, $F$ is not linear. Thus $F\in \cH_4$, so we have
\[
r(\cF,n)\geq r(\cH_4,n)\geq r\coloneqq \exp(\Omega(\log^2n/\log\log n)).
\]
Inverting this, $\log n=O((\log r\log\log r)^{1/2})$, so by \cref{lem:quotient},
\[
r(\cH_5,r^6)=r(\cH_5^\minn,r^6)\geq \exp\big(\Omega\big(\log^{3/2}r/(\log\log r)^{1/2}\big)\big),
\]
as desired.
\end{proof}

It remains to prove the propositions about small graphs. We begin with the five-edge case.

\begin{proof}[Proof of \cref{prop:5edgetripartite}]
A minimal counterexample $F$ has minimum degree at least $2$, or else we may greedily color any leaves or isolated vertices last. Thus if $|V(F)|=n$ then we have the following:
\begin{itemize}
\item[(i)] By double counting edges using the minimum degree, $|E(F)|\cdot 3\geq 2n$ so $n\leq \floor{3|E(F)|/2}$.
\item[(ii)] Since $F$ is linear every vertex has $|N(v)|=2\mathrm{deg}(v)\leq n-1$, and $F$ has maximum degree at least $ \ceil{3|E(F)|/n}$. Thus $n\geq 2\ceil{3|E(F)|/n}+1$.
\end{itemize}
Suppose $F$ has four edges. Then $n\leq \floor{12/2}=6$, and $n\geq 2\ceil{12/n}+1$, which implies $n\geq 6$. If $V(F)=[6]$ then $F$ is 2-regular so up to relabeling  we have edges $123,345\in E(F)$. Then since vertex $6$ has degree 2, up to permuting $\{1,2\}$ and $\{4,5\}$ we have $561,246\in E(F)$. This graph $F$ is tripartite by coloring $1,4\mapsto 1$,$2,5\mapsto 2$, $3,6\mapsto 3$.

Suppose that $E(F)=5$. By (i) $n\leq \floor{15/2}=7$ and by (ii) $n\geq 2\ceil{15/7}+1=7$. Thus $n=7$ and $F$ has maximum degree $3$. Then up to relabeling we have $123,145,167\in E(F)$. Because $F$ is linear, up to permuting $\{2,3\}$ and the other pairs we have final two edges $246,357\in E(F)$. This is iterated tripartite by the coloring $1\mapsto (1,1),2\mapsto (2,1),4\mapsto (2,2),6\mapsto (2,3),3\mapsto (3,1),5\mapsto (3,2),7\mapsto (3,3)$.
\end{proof}

\begin{proof}[Proof of \cref{prop:4edgetripartite}] Suppose $F$ is a minimal counterexample. As above, we may assume $F$ has minimum degree at least $2$ so $F$ has at most $4\cdot 3\cdot \frac 12=6$ vertices. From \cref{prop:5edgetripartite} we may assume $F$ is not linear, so there are two edges $123,124$. If $F$ has six vertices then the two remaining must be contained in both edges $56x,56y$ for $x,y\in [4]$. This is iterated tripartite by the coloring $1\mapsto (1,1),2\mapsto (1,2),3,4\mapsto (1,3),5\mapsto(2,1), 6\mapsto (3,1)$.

If $F$ has five vertices then both remaining edges must contain vertex $5$. Suppose $5xy\in E(F)$ for $x\in \{1,2\}$ and $y\in \{3,4\}$; in this case any fourth edge containing $5$ will cause $F$ to be tightly connected, and thus tripartite. Thus we must have $E(F)=\{123,124,125,345\}$ and then $F$ is iterated tripartite by coloring $1\mapsto (1,1),2\mapsto (2,1),3\mapsto(3,1),4\mapsto (3,2),5\mapsto (3,3)$.

Finally if $F$ has at most four vertices then $F=K_4^{(3)}$, which is tightly connected but not tripartite, a contradiction.
\end{proof}

\subsection{A single 3-graph reduction}\label{sec:single3-graphreduction}

In general, we are not able to prove a result analogous to \cref{lem:quotient} for an individual 3-graph, bounding $r(H,n)$ using one of its tight quotients. However, we are able to do this with an additional mild assumption. We have the following:

\begin{lemma}\label{lem:maxtightquotient} For every $3$-graph $H$ there exists a ``maximum'' tight quotient $F^*$ such that every tight quotient $F$ of $H$ is isomorphic to a proper quotient of $F^*$, and vice versa.
\end{lemma}

For a 3-graph $F^*$, say that a subgraph $F\subseteq F^*$ is \textit{quotient-robust} if no nontrivial quotient $F^*/\cP$ reduces $F$ to a single point. For example, any spanning subgraph is quotient-robust. In this section, we prove the following lemma. 

\begin{lemma}\label{lem:quotientonegraph}
Let $H$ be a minimal NIT $3$-graph and $F^*$ be the maximal tight quotient of $H$. Let $F$ be a quotient-robust minimal NIT subgraph of $F^*$. There exists $c>0$ such that for sufficiently large $r$, if $r(F,n)\geq r$, then
\[
r(H,r)\geq \exp\bigg(c\frac{\log^2r}{\log n}\bigg).
\]
\end{lemma}

We note that in many natural cases, if $H$ is min NIT then $F^*$ is min NIT as well. Moreover, we do not require $F^*$ to be min NIT, just that it has a ``large'' min NIT subgraph. First, we show that ``maximum'' tight quotients exist.

\begin{proof}[Proof of \cref{lem:maxtightquotient}]
Let $\cP^*$ be the common refinement of all partitions of tight quotients of $\cH$, and denote $[v]_{\cP^*}$ by $[v]_*$. Then for all $u,v\in V(H)$, we have $[u]_*=[v]_*$ if and only if $[u]_\cP=[v]_\cP$ for every tight quotient $H/\cP$. Now we check that $F^*=H/\cP^*$ is a tight quotient. Suppose we have $uvw\in E(H)$ with $[u]_*=[v]_*$. Then for every tight quotient $H/\cP$ we have $[u]=[v]$ and thus $[u]=[v]=[w]$ so we conclude $[v]_*=[w]_*$. Thus $F^*$ is proper. Let $uvw$ and $uvw'$ be two edges with common pair $uv$. For every tight quotient $H/\cP$ we have $[w]=[w']$ so we have $[w]_*=[w']_*$ and therefore $F^*$ is tight. Finally, we observe that for every tight quotient $F=H/\cP$, the partition $\cP$ is a coarsening of $\cP^*$ so $F$ is (isomorphic to) a proper quotient of $F^*$. Conversely, for any proper quotient $F=(H/\cP^*)/\cQ$, $F$ is also a tight quotient of $H$ by tightness of $F^*$. 
\end{proof}

Suppose $H$ is minimal NIT, $F^*$ is the maximum tight quotient of $H$, and $F$ is a minimal NIT subgraph of $F^*$. If $\cF$ is the family of nontrivial tight quotients of $H$ (equivalently, the proper quotients of $F^*$), then by applying \cref{lem:quotient} as a black box it will suffice to relate $r(\cF,n)$ to $r(F^*,n)$ by a polynomial factor. 

To relate the Ramsey number of a minimal NIT 3-graph to its set of proper quotients, we use a generalization of a blow-up result of Fox and He \cite[Prop.\ 6.2]{FoxHe2021} for set Ramsey numbers.

\begin{proposition}Let $\cF$ be a set of $3$-graphs, let $F$ be a $3$-graph on $t$ vertices and let $v\in V(F)$. Let $F'$ be the iterated blow-up of $F$ obtained by replacing $v$ with $t$ copies of itself which induce a copy of $F$,
\[
r(\cF\cup \{F'\},n)\leq 2^t\cdot r(\cF\cup \{F\},n)^t.
\]
\end{proposition}

\noindent The proof is identical to the one in \cite{FoxHe2021}. By iterating this for each 3-graph in a finite set, we obtain the following.
\begin{corollary}\label{cor:setramseyblow-up}
If $\cF=\{F_1,\ldots, F_m\}$ and $\cF'=\{F_1',\ldots, F_m'\}$ where $F_i'$ is an iterated blow-up of $F_i$ for each $i$, then there exists a constant $C$ such that for $n\geq 2$,
\[
r(\cF',n)\leq r(\cF,n)^C.
\]
\end{corollary}

Now that we know set Ramsey numbers are preserved under iterated blow-ups, we can reduce the entire family of nontrivial proper quotients to a single graph. 

\begin{lemma}\label{lem:quotientspreserveramsey}
Let $F$ be a minimal NIT 3-graph, and let $\cF$ be the collection of its nontrivial proper quotients. Then there exists a constant $c>0$ such that
$$
r(F,n)^c\le r(\cF,n).
$$
\end{lemma}
\begin{proof}
Let $F$ be a minimal NIT 3-graph and $\cF=\{F_1,\ldots, F_m\}$. By \cref{lem:quotientpreservesNIT}, we have a set of iterated blow-ups $\cF'=\{F_1',\ldots, F_m'\}$ such that $F_i'$ contains $F$ as a subgraph for all $i$. Then by \cref{cor:setramseyblow-up} and monotonicity of Ramsey numbers,
\[
r(F,n)^c\leq r(\cF',n)^c\leq r(\cF,n)
\]
for some constant $c>0$.
\end{proof}

With this, the proof of \cref{lem:quotientonegraph} is a calculation.
\begin{proof}[Proof of \cref{lem:quotientonegraph}]
Let $H$ be minimal NIT, $F^*$ be its maximum tight quotient, and $F$ be a quotient-robust minimal NIT subgraph of $F^*$. Let $\cF^*$ be the family of tight quotients of $H$, which by \cref{lem:maxtightquotient} is the family of nontrivial proper quotients of $F^*$. Let $\cF$ be the family of nontrivial proper quotients of $F$. Suppose that $r(F,n)\geq r$.

Note that property \ref{itm:properquotient} is preserved under taking subgraphs, so every proper quotient $F^*/\cP\in\cF^*$ of $F^*$ contains, as a subgraph, the proper quotient $F/\cP$ of $F$. Moreover, by robustness each nontrivial quotient in $\cF^*$ passes to a nontrivial quotient in $\cF$. Thus by \cref{lem:quotientspreserveramsey} we have 
\[r(\cF^*,n)\geq r(\cF,n)\geq r(F,n)^c\geq r^c,
\]
for some $c\in (0,1]$. Applying \cref{lem:quotient} to $H$ we have
\[
r(H,r)\geq r(H,r^c)\geq \exp\Big(\frac{c^2}{36}\frac{\log^2r}{\log n}\Big),
\]
when $r$ is sufficiently large.
\end{proof}

\subsection{Reduction to linear 3-graphs}\label{sec:reductiontolinear}

In this section, we prove Theorems~\ref{thm:reductiontolinear} and \ref{thm:singlelinear}. We restate \cref{thm:reductiontolinear} in an equivalent way.

\begin{numtheorem}{\ref{thm:reductiontolinear}} Let $\cH_t$ be the family of NIT $3$-graphs with at most $t$ tight components and $\cL_t$ be the family of linear NIT $3$-graphs with at most $t$ edges. If $r(\cL_t,n)\geq n^{\omega(1)}$ for infinitely many $n$ then $r(\cH_t,n)\geq n^{\omega(1)}$ for infinitely many $n$.
\end{numtheorem}

\begin{proof} Let $\cF_0=\cH_t$. Then for $i\geq 0$ let $\cF_i'$ be the set of min NIT subgraphs of elements of $\cF_i$, and let $\cF_{i+1}$ be the set of nontrivial tight quotients of $\cF_i'$. By applying \cref{lem:quotientpreservesNIT} recursively, every $F\in \cF_i$ is NIT for all $i$. Moreover, for all $F\in \cF_{i+1}$, we have $F=F'/\cP$ for some $F'\in \cF_i'$. Thus either $F=F'$ is linear or else $F$ is strictly smaller than $F'$

Since every member of $\cF_1$ has at most $3t$ vertices, it holds recursively that every member of $\cF_{1+i}$ is linear or has at most $3t-i$ vertices. Thus every member of $\cF_{3t}$ is linear, so $\cF_{3t}\subseteq \cL_t$. In particular, $r(\cF_{3t},n)\geq r(\cL_t,n)$.

Finally, by \cref{lem:quotient}, if $r(\cF_{i+1},n)=n^{\omega(1)}$, then $r(\cF_i,n)\geq r(\cF_i',n)=n^{\omega(1)}$. Applying this bound recursively, we conclude that $r(\cH_t,n)=n^{\omega(1)}$.
\end{proof}

Next, we prove \cref{thm:singlelinear}. We require the following lemma, which reduces from a single 3-graph $H$ to a linear 3-graph which is not much larger.

\begin{lemma}\label{lem:gadgetreplacement} For every NIT 3-graph $H$, there is a linear NIT $3$-graph $L$ such that $|E(L)|\leq 7|E(H)|$ and $L$ is contained in an iterated blow-up of $H$. 
\end{lemma}
Before the proof, we introduce a new term. Let $F$ be a connected $3$-graph. A \textit{top-level} 3-coloring of $F$ is a coloring $\chi:V(F)\to \{1,2,3\}$ such that each edge of $F$ is rainbow or monochromatic. Observe that if $F$ is iterated tripartite with minimal depth coloring $(\chi,\ldots, \chi_t):V(F)\to \{1,2,3\}^t$, then $\chi_1$ is a non-constant top-level coloring. A top-level coloring $\chi$ can be identified with a tripartition $\cP=\{\chi^{-1}(1),\chi^{-1}(2),\chi^{-1}(3)\}$ of $F$ such that $F/\cP$ is trivial if $\chi$ is constant, and otherwise $F/\cP\cong K_3^{(3)}$.

\begin{figure}
    \centering
\input{tikz_diagrams/splitfano}
\caption{The split Fano plane $F^{\protect\threevee}$. This is iterated tripartite, where up to permutation, the unique non-constant top-level 3-coloring is determined by making the vertical edges $x23,y45,z67$ monochromatic. In the context of iterated $3$-colorings, this gadget behaves as if there were an edge $xyz$.}
\label{fig:splitfano}
\end{figure}
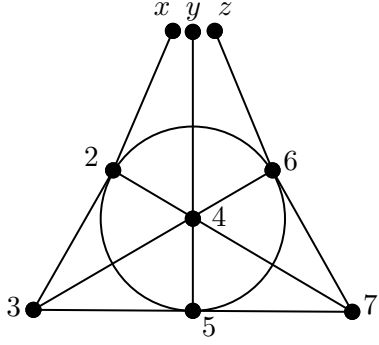

\begin{proof}
It suffices to prove this when $H$ is min NIT. By \cref{lem:quotientpreservesNIT}, the only tripartite proper quotient of $H$ is trivial, so any top-level 3-coloring of $H$ is constant. We construct $L$ by using $H$ as a template, then replacing each edge $e$ by a 9-vertex, 7-edge gadget which has the same color-forcing properties as $e$. Define the \textit{split Fano plane} $F^{\threevee}$ (see Figure \ref{fig:splitfano}) as the 3-graph obtained from the Fano plane (Equation \ref{eq:fano}) by splitting vertex 1 into three pieces, as follows:
\begin{align*}
V(F^{\threevee})&=\{x,y,z,2,3,4,5,6,7\}\\
E(F^{\threevee})&=\{x23,y45,z67,247,346,256,357\}.
\end{align*}
Now, we obtain $L$ by replacing each edge $xyz\in H$ with a copy of $F^{\threevee}$. Formally,
\begin{align*}
V(L)&=V(H)\cup \{e_2,\ldots, e_7\mid e\in E(H)\}\\
E(L)&=\bigcup_{e\in E}F^{\threevee}_e,
\end{align*}
where for each edge $e=xyz\in E(H)$, $F^{\threevee}_e$ is a copy of $F^{\threevee}$ on vertex set $\{x,y,z,e_2,\ldots, e_7\}$. The exact order of labeling $x,y,z$ is irrelevant by symmetry; we denote $N_{F^{\threevee}_e}(x)=\{x'_e,x''_e\}$, and similarly for $y,z$. The top-level coloring given by $\chi(x,x'_e,x''_e)=1$ and $\chi(y,y'_e,y''_e)=2$ and $\chi(z,z'_e,z''_e)=3$ shows that $F^{\threevee}_e$ is iterated tripartite. Obtain the graph $H^+$ from $H$ by blowing up each $x\in H$ into the set $X=\{x\}\cup \{x'_e,x''_e\mid e\ni x\}$, then placing the edges $\{xx'_ex''_e\mid e\ni x\}$ in each part. Now $H^+$ contains $L$ and since each part is tripartite, $H^+$ is contained in an iterated blow-up of $H$. Also, $L$ is linear because $F^{\threevee}$ is, and any two gadgets only intersect in the vertices $x,y,z$. Now we show the following, which is the key gadget property.
\begin{claim}\label{claim:splitfano} The 3-graph $F^{\threevee}$ is iterated tripartite and every non-constant top-level coloring of $F^{\threevee}$ is rainbow on $\{x,y,z\}$.
\end{claim}
\begin{subproof} The first part is shown above. By symmetry we suppose $\chi$ is a top-level coloring $\chi$ where $\chi(x)=\chi(y)$. By gluing $x,y$ together we see that $F^{\threevee}/\{x,y\}-z\cong F^-$, the Fano plane minus an edge. Since $\chi$ induces a top-level coloring of $F^-$ and $F^-$ is min NIT, $\chi$ must be constant apart from $\chi(z)$. However, the edge $z67$ implies that $\chi$ is constant.
\end{subproof}
Finally, we show that $L$ is NIT. If $\chi$ is a top-level $3$-coloring of $L$, then for every $e=xyz\in E(H)$, by \cref{claim:splitfano} either $\chi$ is constant on the gadget $F^{\threevee}_e$ or else $\chi$ properly colors $e$. Thus $\chi$ restricts to a top-level 3-coloring of $H$, which must be constant. Thus $\chi$ is also constant. We conclude that $L$ has no non-constant top-level $3$-coloring, so $L$ is NIT.
\end{proof}

Using this Lemma, we restate and prove the contrapositive of \cref{thm:singlelinear}.
\begin{numtheorem}{\ref{thm:singlelinear}} If $H$ has at most $k$ edges, $H$ is NIT, and $r(H,n)=n^{O(1)}$, then there exists an NIT linear $3$-graph $L$ with at most $7k$ edges such that $r(L,n)=n^{O(1))}$.
\end{numtheorem}
\begin{proof}Let $L$ be obtained from \cref{lem:gadgetreplacement}, i.e. $L$ is obtained by replacing each edge of $H$ with a copy of $F^{\threevee}$. Since $L$ is contained in an iterated blow-up of $H$, \cref{cor:setramseyblow-up} implies $r(L,n)\leq r(H,n)^C=n^{O(1)}$ for some constant $C$, as desired.
\end{proof}

\section{Quasipolynomial upper bounds}\label{sec:upperbounds}

In this section, we will prove the following and use them to deduce the results of Section \ref{sec:intro_quasiupper}. Recall that $B_{n,n}$ is the two-sided complete bipartite $3$-graph and $S_{n,n}$ is the one-sided version. We note that \cref{thm:doublingvsSnn} is essentially the same as \cite[Prop. 6.3]{gridramsey}, though slightly stronger.

\begin{theorem}\label{thm:fwddoubling}
If $H$ is forward-doubleable, then $r(H,K_n)=\exp(\log^{\Theta(1)}n)$.
\end{theorem}

\begin{theorem}\label{thm:fwddoublingvsBn}
If $H$ is forward-doubleable, then $r(H,B_{n,n})=n^{O(1)}$. 
\end{theorem}

\begin{theorem}\label{thm:doublingvsSnn} If $H$ is doubleable, then $r(H,S_{n,n})=n^{O(1)}$.
\end{theorem}

\begin{theorem}\label{thm:similardoubling}
Suppose $H$ is a graph with two vertices $u,w\in V(H)$, such that some automorphism on $H$ maps $u$ to $w$. Then
\[
r(H[u,w],n)\leq r(H,n)^{O(\log n)}.
\]
\end{theorem}

\begin{corollary}\label{cor:tridoubling}
If $H$ is iterated tripartite and $H^+$ is obtained by one doubling in $H$, then $r(H^+,n)=\exp(O(\log^2n))$.
\end{corollary}

Let us define terms.

\begin{definition}Let $H$ be a $3$-graph and $u,v\in V(H)$ be distinct vertices with no common edge. Then $H[u,v]$ is obtained by duplicating the vertex $u$ into $u'$, and adding the edge $uu'v$. We call this \textbf{doubling} the pair $(u,v)$. Formally
\begin{align*}
V(H[u,v])&=V(H)\cup \{u'\},\\
E(H[u,v])&=E(H)\cup\{u'xy\mid uxy\in E(H)\}\cup\{uu'v\}.
\end{align*}
If the vertices $u,v\in V(H)$ are contained in some common edges $uvw$, we define the doubling $H[u,v]$ by first deleting all common edges, then doubling the pair $(u,v)$. We say that $H$ is \textbf{doubleable} if $H$ can be obtained from an independent graph by a sequence of doubling and taking subgraphs. 
\end{definition}

We show that this definition generalizes the class of iterated tripartite graphs. Recall that the growth type of off-diagonal Ramsey numbers is closed under iterated blow-ups (see \cite[Prop.\ 6.2]{FoxHe2021} or \cref{cor:setramseyblow-up}), so the following is a natural property.

\begin{proposition}\label{prop:doubleableiteratedblow-ups}
The class of doubleable $3$-graphs is closed under iterated blow-ups.
\end{proposition}
\begin{proof}First note that the class is closed under vertex duplication; to duplicate $u$, begin with an auxiliary vertex $x$ and double $(u,x)$. To build an iterated blow-up of $H$, first construct $H$ via doubling; then duplicate each vertex $v$ into an independent set $V_v$, then construct another copy of $H$ within $V_v$ via doubling for each $v$.
\end{proof}

The definition of forward doubling is more restrictive, and is somewhat analogous to the definition of forward laminar in \cref{sec:explower}. 
\begin{definition}\label{defn:fwddoubling} Define a \textbf{blocked} 3-graph $H$ as a $3$-graph $H$ together with some ordered partition $A=(A_1,\ldots, A_t)$ of $V(H)$. A \textbf{forward doubling} is a doubling of a pair $(u,v)$ where $u\in A_i, v\in A_j$ for $i<j$. In the resulting partition, we put $u'\in A_i$.

A 3-graph $H$ is \textbf{forward-doubleable} if there is some partition $(A_i)_{i=1}^t$ such that $H$ can be obtained from an independent graph (with singleton blocks) by a sequence of forward doubling and taking subgraphs.
\end{definition}

\begin{example}See \cref{lem:fanofwddoubleable} and the corresponding Figure \ref{fig:buildingfano}, which shows how to obtain the Fano plane by forward doubling.
\end{example}

\begin{remark}A stronger version of forward-doubleable can be defined on a partially ordered graph $(H,\prec)$ by allowing the additional operation of duplicating a vertex into ordered twins in $\prec$. 

An even stronger version can be defined on a linearly ordered graph $H$ by the combined ``ordered-forward-doubling'' operation $H[u,v]$ for vertices $u<v$ where in the resulting ordering we let $u<u'$. Both of these stronger versions are closed under iterated blow-ups, and it is possible that all of our theorems are true for these versions. This would require building a quasi-dense embedding structure in tandem with $H$ such that all pairs $u\prec v$ are associated with a forward quasi-dense pair $(U_u,U_v)$.
\end{remark}

\subsubsection{Section organization}

This section is organized in forward order of complexity. We begin with the proof of \cref{thm:doublingvsSnn} which uses a familiar supersaturation argument to embed a graph $H[u,v]$ using a blow-up of $H$. We also prove \cref{thm:similardoubling} using a clean iterative bipartitioning argument.

Next, we prove \cref{thm:fwddoublingvsBn} by counting copies of $H[u,v]$ using a similar blow-up embedding strategy, but where $H$ is mapped into a particularly nice ``forward-hitting'' substructure of $G$. Finally, we prove \cref{thm:fwddoubling} by combining the iterative bipartitioning argument with a more complex ``quasi-dense'' embedding structure with more careful quantitative bounds.

\subsection{Proofs of doubling theorems}\label{sec:doublingproofs}
In this section we prove Theorems \ref{thm:doublingvsSnn} and \ref{thm:similardoubling}. These techniques are then refined to prove the forward-doubling theorems.

\begin{definition}Let $H,G$ be $3$-graphs. A (labeled) \textbf{copy} of $H$ in $G$ is an injective function $\sigma:V(H)\to V(G)$, also informally denoted
\[\sigma=\big\{(v\mapsto \sigma(v))\mid v\in V(H)\big\},\] such that for all $xyz\in E(H)$, we have $\sigma(x)\sigma(y)\sigma(z)\in E(G)$. This gives a labeled subgraph of $G$ isomorphic to $H$.

An \textbf{unlabeled} copy of $H$ in $G$ is a set $S\subseteq V(G)$ which is the image of some copy $\sigma$. An \textbf{independent} copy of $H$ in $G$ is a copy of $H$ in $\overline G$. Recall that $r(H,K)$ is the minimum $N$ such that any $N$-vertex $3$-graph contains a copy of $H$ or an independent copy of $K$.
\end{definition}

\begin{proposition}[Supersaturation]\label{prop:supersaturation} Let $H$ be a $t$-vertex $3$-graph, $t\geq 1$, and let $K$ be an arbitrary $3$-graph. Then for all $N\geq r(H,K)$, any $N$-vertex $3$-graph $G$ with no independent copy of $K$ contains at least $N^t/r(H,K)^{t}$ copies of $H$.
\end{proposition}
\begin{proof}
Let $\cG$ be the auxiliary $t$-uniform hypergraph on $V(G)$ where $E(\cG)$ is the set of unlabeled copies of $H$ in $G$. Then $\cG$ contains no independent set of size $r\coloneqq r(H,K)$ by definition of Ramsey numbers. Let $S$ be a uniformly random $r$-set of $V(\cG)$, so
\[
1\leq \Ex[|E(\cG[S])|]\leq |E(\cG)|(r/N)^t.
\]
Rearranging gives the result. 
\end{proof}
\begin{remark}Actually, standard Tur\'an bounds (e.g. sampling with alterations) give lower bounds of $\Omega(N^t/r^{t-1})$. We choose the weaker bound because it is cleaner and does not change the result.
\end{remark}

\begin{definition}Let $H,G$ be $3$-graphs. We say that $\cH=\{U_v\mid v\in V(H)\}$ forms a \textbf{blow-up} of $H$ in $G$ if $U_v\subseteq V(G)$ for all $v$ and the sets $\{U_v\}_v$ are pairwise disjoint, and for all $uvw\in E(H)$, $x\in U_u$, $y\in U_v,$ and $z\in U_w$, we have $xyz\in E(G)$.
\end{definition}

\begin{lemma}\label{lem:doublingblow-up} Let $H$ be a $t$-vertex $3$-graph with distinct vertices $u,w\in V(H)$ contained in no common edge, and let $K$ be an arbitrary $3$-graph. For all $N\geq 2r(H,K)$, any $N$-vertex $3$-graph with no independent copy of $K$ contains a blow-up $\cH=\{U_v\mid v\in V(H)\}$ of $H$ in $G$ such that $|U_u|,|U_w|\geq N/2r(H,K)^t$.
\end{lemma}
\begin{proof}
Let $r=r(H,K)$. For a copy $\sigma$ of $H-u-w$, let $U(\sigma)$ be the set of extensions to a copy of $H-w$; that is, $x\in U(\sigma)$ if and only if $\sigma\cup \{(u\mapsto x)\}$ is a copy of $H-w$. Similarly, let $W(\sigma)$ be the set of extensions to a copy of $H-u$. Since $u,w$ are not contained in any common edge, a copy of $H$ is formed from $\sigma$ by choosing any distinct $x\in U(\sigma)$ and $y\in W(\sigma)$, so by \cref{prop:supersaturation}
\[
N^t/r^{t}\leq \sum_{\sigma\in S_H}|U(\sigma)||W(\sigma)|-|U(\sigma)\cap W(\sigma)| 
\]
By averaging over $|S_H|\leq N^{t-2}$ choices of $\sigma$, there is some copy $\sigma$ with $|U(\sigma)||W(\sigma)|\geq N^2/r^{t}$. In particular, we have $|U(\sigma)|,|W(\sigma)|\geq N/r^{t}$. Choose disjoint $U_u\subseteq U(\sigma)$ and $U_w\subseteq W(\sigma)$ of size at least $N/2r^t$, and put $U_v=\{\sigma(v)\}$ for each other $v\in V(H)$. Then $\cH=\{U_v\}_v$ forms the desired blow-up of $H$.
\end{proof}

\begin{proof}[Proof of \cref{thm:doublingvsSnn}]
We proceed by induction on $|V(H)|$. If $H$ is the independent $3$-graph on $t$ vertices then $r(H,n)\leq t=O(1)$. Otherwise, suppose the claim holds for $H'$ and consider $H'[u,w]$ for distinct vertices $u,w\in V(H)$. Let $H\subseteq H'$ be obtained by deleting the edges of $H'$ containing $\{u,w\}$; by definition $H'[u,w]=H[u,w]$ and by assumption $r(H,S_{n,n})\leq r(H',S_{n,n})\leq n^{C}$ for some $C\geq 1$. Let $|H|=t$. 

Let $G$ be an $N$-vertex graph with $N= 2n^{Ct+1}$ and no copy of $K=S_{n,n}$. By \cref{lem:doublingblow-up}, there is a blow-up $\cH=\{U_v\mid v\in V(H)\}$ with $|U_u|,|U_w|\geq N/2n^{Ct}\geq n$. Let $A=U_u$ and $B=U_w$. Let $\sigma$ be a copy of $H-u-w$ formed by $\cH$.

If $G$ contains an $AAB$ edge $a_1a_2b$, then the map $\sigma\cup\{(u\mapsto a_1),(u'\mapsto a_2),(w \mapsto b)\}$ forms a copy of $H[u,w]$ in $G$. Otherwise, the sets $A,B$ induce a copy of $S_{n,n}$ in $\overline G$. Thus $r(H,S_{n,n})\leq N$.
\end{proof}

Next, we prove \cref{thm:similardoubling} using the method of iterative bipartitioning. Say that vertices $u,w\in V(H)$ are \textit{similar} if there is an automorphism of $H$ mapping $u$ to $w$. The following inductive lemma finds a large empty $B_{M,M}$ in any $(H[u,w],n)$-Ramsey graph. We'll iterate this within each side to find an independent set.

\begin{lemma}\label{lem:similarpartitioning} Suppose $H$ is a $t$-vertex 3-graph with distinct similar vertices $u,w\in V(H)$. Let $G$ be an $N$-vertex graph with $N\geq 2r(H,n)$, with no copy of $H[u,w]$ and no independent set of size $n$. Then there exist disjoint $A,B\subseteq V(G)$ of size at least $N/2r(H,n)^{t}$ such that $e(A,B)=\emptyset$.
\end{lemma}

\begin{proof} By monotonicity of $r(H,n)$, we may assume $u,w$ are contained in no common edge. By \cref{lem:doublingblow-up}, we have a blow-up $\cH$ of $H$ with $|U_u|,|U_w|\geq N/2r(H,n)^t$. Let $A=U_u$ and $B=U_w$. If $G$ contains an $AAB$ edge $a_1a_2b$, then this completes a copy of $H[u,w]$ in $G$, a contradiction. If $G$ contains an $ABB$ edge $ab_1b_2$, then this completes a copy of $H[w,u]$ in $G$, which by similarity is isomorphic to $H[u,w]$. Thus $e(A,B)=\emptyset$, as desired.
\end{proof}

\begin{proof}[Proof of \cref{thm:similardoubling}]
Let $n\leq 2^\ell$, let $\eps^{-1}=2r(H,n)^t$, and let $N=\eps^{-\ell}$. Let $G$ be an $N$-vertex graph with no copy of $H[u,w]$ and no independent set of size $n$. We recursively construct sets $V_b$ for each binary word $b$ of length at most $\ell$, maintaining that $|V_b|\geq \eps^{|b|}N$. Let $V_\emptyset=V(G)$, so $|V_\emptyset|=N$.

For a binary word $b$ of length less than $\ell$, assuming $|V_b|\geq \eps^{|b|}N \geq 2r(H,n)^t$, \cref{lem:similarpartitioning} implies that there exist $V_{b0},V_{b1}\subseteq V_b$ of size $|V_{b0}|,|V_{b1}|\geq \eps |V_b|\geq \eps^{|b|+1}N$ such that $e(V_{b0},V_{b1})=\emptyset$.

Finally, we have a family of at least $n$ nonempty sets $\{V_b\mid b\in \{0,1\}^\ell\}$ with no edges crossing between them. Choose one vertex from each set to obtain an independent set of size $n$ in $G$, a contradiction. Thus $r(H[u,v],n)\leq N\leq r(H,n)^{O(\log n)}$, as desired.
\end{proof}

From this, we obtain the corollary for iterated tripartite graphs.

\begin{proof}[Proof of \cref{cor:tridoubling}]
By monotonicity (and because a doubling of a subgraph is a subgraph of a doubling), it suffices to prove this when $H=(K_3^{(3)})^d$ is the $d$-times iterated blow-up of an edge. By symmetry, any two vertices of $H$ are similar, so $r(H^+,n)\leq r(H,n)^{O(\log n)}=n^{O(\log n)}$, as desired.
\end{proof}

\subsection{Proof of Theorem \ref{thm:fwddoublingvsBn}}\label{sec:fwddoublingvsBn}

In this section, we show that if a graph $H$ is forward-doubleable, then $r(H,B_{n,n})=n^{O(1)}$ by proving that any graph $G$ with no independent $B_{n,n}$ contains many copies of $H$. Here, we refine the blow-up argument by building a more rigid framework of sets which contains many copies of $H$.

\begin{definition}
For a graph $G$ and disjoint subsets $A,B\subseteq V(G)$, we say that the pair $(A,B)$ is \textit{forward-$n$-hitting} (or just $n$-hitting) if for all $X\subseteq A$ and $Y\subseteq B$ having size $n$, $G$ contains some $XXY$ edge.
\end{definition}

Note that if $G$ had no independent copy of $S_{n,n}$, then every pair of disjoint sets would be $n$-hitting. In the case of $B_{n,n}$, we can still construct such pairs.

\begin{proposition}\label{prop:hitting}Let $n\geq 2$ and $C\geq 1$. If $G$ is an $N$-vertex graph with $N\geq 2n^{C^2}$ and no independent copy of $B_{n,n}$, then there exists $m\in \{n,n^{C}\}$ such that $G$ contains disjoint sets $A,B$ of size at least $m^{C}$ such that $(A,B)$ is $m$-hitting. 
\end{proposition}
\begin{proof}Let $A,B\subseteq V(G)$ be subsets of size $n^{C^2}$. If $(A,B)$ is $n^{C}$-hitting then we are done. Otherwise, let $X\subseteq A$ and $Y\subseteq B$ have size $n^C$ such that $e(X,X,Y)=0$. Then $(Y,X)$ must be $n$-hitting, or else there exist subsets $U\subseteq X$, $W\subseteq Y$ of size $n$ such that $e(U,W)=e(U,U,W)+e(W,W,U)=0$, giving an independent copy of $B_{n,n}$.
\end{proof}

Repeating this, we build a sequence of sets which will ultimately house many copies of our forward-doubleable graph $H$.

\begin{lemma}\label{lem:forward-hitting-seq} For $k,t\geq 1$, there exists $C\geq 1$ such that the following holds. If $G$ is an $N$-vertex graph with $N\geq n^{C}$ and no independent copy of $B_{n,n}$, then for some $m\geq n$, there exist disjoint sets $U_1,\ldots, U_t\subseteq V(G)$ of size at least $m^k$ such that for all $i<j$, the pair $(U_i,U_j)$ is $m$-hitting.
\end{lemma}
\begin{proof}For the base case $t=1$ we take $U_1=V(G)$ and $m=n$. For the inductive case, suppose we have some $m$ and sets $U_1,\ldots, U_{t-1},U_t\subseteq V(G)$ with the desired properties for $K=k^2+1$, and by \cref{prop:hitting} there exists $U_t',U_{t+1}'\subseteq U_t$ and some $\ell\in \{m,m^k\}$ such that $|U_t'|,|U_{t+1}'|\geq \ell^k$ and the pair $(U_t',U_{t+1}')$ is $\ell$-hitting. Since $\ell\geq m$ and $\ell^k\leq m^K$, the sets $U_1,\ldots, U_t',U_{t+1}'$ have the desired properties.
\end{proof}

To find many copies of $H$, we use the following fact, which ensures that an $n$-hitting pair is dense everywhere.

\begin{proposition}\label{prop:hittingdensity} If $(A,B)$ is $n$-hitting, then for all $X\subseteq A$ and $Y\subseteq B$ with $|X|,|Y|\geq n$, we have $e(X,X,Y)> |X|^2|Y|/n^3$.
\end{proposition}
\begin{proof}
Let $X'\in \binom Xn$ and $Y'\in \binom Yn$ be uniformly random. Then since $(X,Y)$ is also $n$-hitting,
\[
1\leq \Ex[e(X',X',Y')] = \frac{n(n-1)}{|X|(|X|-1)}\frac{n}{|Y|} e(X,X,Y).
\]
Rearranging gives the result.
\end{proof}

Finally, we prove that $n$-hitting sequences actually contain forward-doubleable graphs. For a forward-doubleable blocked $3$-graph $H$ with blocks $A_1,\ldots, A_t$, and another 3-graph $G$ with disjoint sets $U_1,\ldots, U_t\subseteq V(G)$, say that a \textit{$(U_i)_i$-copy} of $H$ in $G$ is a copy $\sigma:V(H)\to V(G)$ such that $\sigma(A_i)\subseteq U_i$ for all $i$. 

\begin{lemma}
\label{lem:countingagainstBnn} Let $H$ be a forward-doubleable blocked 3-graph with blocks $A_1,\ldots, A_t$. There is a constant $C$ depending on $H$ such that the following holds for all $k\geq C$.

Let $G$ be a 3-graph, let $U_1,\ldots, U_t\subseteq V(G)$ have size $n^k$, and suppose for all $i<j$, the pair $(U_i,U_j)$ is $n$-hitting. Then there are at least $n^{k|V(H)|-C}$ many $(U_i)_i$ copies of $H$ in $G$.
\end{lemma}
\begin{proof}We proceed by induction on $|H|=|V(H)|$. If $H$ is an independent set with $|A_i|=1$ for all $i$, then there are exactly $n^{k|H|}$ many $(U_i)_i$ copies of $H$, so $C=0$ suffices.

For the inductive case, consider $H[u,w]$ for vertices $u\in A_i$ and $w\in A_j$ with $i<j$. By monotonicity, we may assume $u,w$ are contained in no common edge. By induction we have some constant $C=C_H$. Let $C'=2C+5$ and suppose $k\geq C'$. By induction there are at least $n^{k|H|-C}$ many $(U_i)_i$ copies of $H$. For each $(U_i)_i$ copy $\sigma$ of $H-u-w$ define $U(\sigma)\subseteq U_i$ and $W(\sigma)\subseteq U_j$ as in \cref{lem:doublingblow-up}, where for all $x\in U(\sigma)$ and $y\in W(\sigma)$, the map $\sigma\cup \{(u\mapsto x),(w\mapsto y)\}$ forms a copy of $H$. Then we have
\[
n^{k|H|-C}\leq \sum_{\sigma\in \Sigma}|U(\sigma)||W(\sigma)|,
\]
where $\Sigma$ is the set of all $(U_i)_i$ copies of $H-u-w$. Observe that if $M$ is the number of $(U_i)_i$ copies of $H[u,w]$ then 
\[
M=\sum_{\sigma\in \Sigma}e(U(\sigma),U(\sigma),W(\sigma)).
\]
We lower bound $M$ by splitting $\Sigma$ into cases. Let $\Sigma=\Sigma_1\sqcup \Sigma_2$, where $\sigma\in \Sigma_1$ if $|U(\sigma)|,|W(\sigma)|\geq n$, and otherwise $\sigma\in \Sigma_2$. Since $|\Sigma|\leq n^{k(|H|-2)}$, we have
\begin{align*}
\sum_{\sigma\in \Sigma_1}|U(\sigma)||W(\sigma)|
&\geq n^{k|H|-C}-\sum_{\sigma\in \Sigma_2}|U(\sigma)||W(\sigma)|\\
&\geq n^{k|H|-C}-n^{k|V(H)|-2k}\cdot n\cdot n^k\\
&\geq n^{k|H|-C}(1-n^{C-k+1})\\
&\geq n^{k|H|-C-1}.
\end{align*}
Then by \cref{prop:hittingdensity} and convexity,
\begin{align*}
M\geq \sum_{\sigma\in \Sigma_1}e(U(\sigma),U(\sigma),W(\sigma)) 
&\geq n^{-3}\sum_{\sigma\in \Sigma_1}|U(\sigma)|^2|W(\sigma))|\\
&\geq n^{-3-k}\sum_{\sigma\in \Sigma_1}|U(\sigma)|^2|W(\sigma))|^2\\
&\geq n^{-3-k}\frac{(n^{k|H|-C-1})^2}{|\Sigma_1|}\\
&\geq n^{-3-k}n^{k|H|+2k-2C-2}\\
&\geq n^{k|H[u,v]|-C'},
\end{align*}
where $C'=2C+5$, as desired.
\end{proof}

From this, \cref{thm:fwddoublingvsBn} follows immediately.

\begin{proof}[Proof of \cref{thm:fwddoublingvsBn}]
Let $H$ be a forward-doubleable blocked $3$-graph with $t$ blocks and let $k$ be the constant from \cref{lem:countingagainstBnn}. By \cref{lem:forward-hitting-seq}, there exists some $C=C_{k,t}$ such that if $G$ is an $N$-vertex graph with $N\geq n^{C}$ and no independent copy of $B_{n,n}$, there is some $m\geq n$ and disjoint sets $(U_i)_i$ with the given properties (large and forward $m$-hitting). Then by \cref{lem:countingagainstBnn}, $G$ contains many copies of $H$, and thus $r(H,B_{n,n})\leq N$.
\end{proof}

\subsection{Quasi-denseness and Theorem \ref{thm:fwddoubling}}\label{sec:quasidense}

In this section, we prove our main theorem of Section \ref{sec:upperbounds}. The idea is to generalize the construction of \cref{lem:forward-hitting-seq}.\footnote{In fact, this argument came first, and we noticed that the techniques could be simplified to prove \cref{thm:fwddoublingvsBn}.} Instead of requiring that the pairs of sets $(U_i)_i$ are $m$-hitting, we ensure a weaker density property which still allows us to embed many copies of some $H$. The main useful property of $n$-hitting sets is given by \cref{prop:hittingdensity}, which loosely translates to saying that if $(A,B)$ is $n$-hitting then the pair is ``not too sparse at any level''. The following definition captures a weaker version of the same idea.

\begin{definition}
Let $G$ be a $3$-uniform hypergraph ($3$-graph) and let $U\subseteq V(G)$. For $1\leq \alpha_1<\alpha_2$ and $d\geq 1$, we say that $U$ is \textbf{$(\alpha_1,\alpha_2,d)$-quasi-dense} if for all $\beta\in [\alpha_1,\alpha_2]$ and all disjoint $A,B\subseteq U$ with $|A|,|B|\geq 2^{-\beta}|U|$ we have
\[
e(A,B)\geq 2^{-\beta d}|U|^3.
\]
For disjoint subsets $U,W\subseteq V(G)$ we say that the pair $(U,W)$ is \textbf{$(\alpha_1,\alpha_2,d)$-forward-quasi-dense} if for all $\beta\in [\alpha_1,\alpha_2]$ and all $A\subseteq U$ and $B\subseteq W$ with $|A|\geq 2^{-\beta}|U|$ and $|B|\geq 2^{-\beta}|W|$ we have
\[
e(A,A,B)\geq 2^{-\beta d}|U|^2|W|.
\]
The key difference is that we can force edges to be 2-1 across the partition, which can build doublings $H[u,w]$ in the desired direction. 
\end{definition}

\begin{lemma}[Quasi-dense sets exist]\label{lem:quasidenseexists}Let $G$ be a $3$-graph on $N$ vertices with $\alpha(G)<n$. Let $d=100\log^2_2n$ and suppose $2^{-d^{k+1/2}}N\geq 1$. Then there exists $U\subseteq V(G)$ of size $|U|\geq 2^{-d^{k+1/2}}N$ such that $U$ is $(1,d^k,d)$-quasi-dense.
\end{lemma}

\begin{lemma}[Discrete construction of quasi-denseness]\label{lem:discretequasidense} Let $G$ be a $3$-graph on $N$ vertices with $\alpha(G)<n$. For $i\geq 0$ let $\eps_i=2^{-(10\log_2 n)^i}$. If $\eps_k N\geq 1$ then there exists $U\subseteq V(G)$ of size $|U|\geq \eps_kN$ such that for all $i<k$ and disjoint $A,B\subseteq U$ with $|A|,|B|\geq \eps_i|U|$, we have $e(A,B)\geq \eps_{i+1}|U|^3$.
\end{lemma}
\begin{proof}[Proof (\ref{lem:quasidenseexists} from \ref{lem:discretequasidense})] Choose $U$ as above with $\eps_1,\ldots, \eps_{2k+1}$. Then for any $\beta\in [1,d^k]$, we have $\eps=2^{-\beta}\in [\eps_{2k},\eps_0]$,  there is some index $i$ with $\eps\in [\eps_{i},\eps_{i-1}]$, so for all $|A|,|B|\geq \eps|U|$ we have $e(A,B)\geq \eps_{i+1}|U|^3\geq \eps^{100\log^2n}|U|^3$, as desired.
\end{proof}

Somehow, the quasi-dense property is almost useless if we don't know which direction the edges are going ($AAB$ versus $ABB$). Fortunately, using a surprisingly simple magic trick similar to \cref{prop:hitting} we can black box the quasi-dense lemma to obtain a powerful embedding structure.

\begin{lemma}\label{lem:fwdqdpairexists} For $k,d\geq 3$, let $K=4k^2$. Let $G$ be a 3-graph and $U\subseteq V(G)$ be $(\alpha,\alpha d^K,d)$ quasi-dense. There exist disjoint subsets $A,B\subseteq U$ such that for some $\beta\in [\alpha, \alpha d^K]$ and $D\in \{d^2,d^{3k}\}$ we have $|A|=|B|\geq 2^{-\beta}|U|$ and $(A,B)$ is $(\beta,\beta D^k,D)$ forward-quasi-dense.
\end{lemma}

Repeating this as in \cref{lem:forward-hitting-seq}, we have a sequence of forward quasi-dense sets which will house copies of $H$.

\begin{lemma}\label{lem:forwarddensesequence} For $t,k\geq 1$, there exists $C>1$ such that the following holds. For $n\geq 3$, if $G$ is a $3$-graph on $N$ vertices with $\alpha(G)<n$, and $N\geq 2^{\log^Cn}$, then there exists equal-sized disjoint subsets $U_1,\ldots, U_t$ such that for some $d,\alpha\geq 3$,
\begin{itemize}
\item Each of $U_1,\ldots, U_t$ is $(\alpha,\alpha d^k,d)$ quasi-dense
\item For all $i<j$, the pair $(U_i,U_j)$ is $(\alpha,\alpha d^k,d)$ forward-quasi-dense.
\end{itemize}
\end{lemma}

\subsubsection{Constructing quasi-dense sets.}

Here, we prove that quasi-dense sets, and in turn forward-quasi-dense sequences, exist in graphs with small independence number. The pivotal machine is a density-tuned version of the iterative bipartitioning argument used in \cref{lem:similarpartitioning} and the proof of \cref{thm:similardoubling}.

\begin{proof}[Proof of \cref{lem:discretequasidense}] 
For $i\geq 0$, define the property
\begin{align}
&\textit{For all $A,B\subseteq U$ with $|A|,|B|\geq \eps_i|U|$, we have $e(A,B)> \eps_{i+1}|U|^3$}\tag{$\rm P_i$}
\end{align}
say $U$ is \textit{$k$-dense} if $P_i$ holds for $i\in \{0,\ldots, k-1\}$. We proceed by induction on $k$. For the base case $k=0$, we take $U=V(G)$ which is vacuously $0$-dense, and also $|U|>\eps_0N$.

Suppose towards a contradiction that the claim holds for $k$ and fails for $k+1$. Let $G$ be an $N$-vertex $3$-graph with $\alpha(G)<n$ and $\eps_{k+1}N\geq 1$. Then for all $V'\subseteq V(G)$ of size $N'\geq (\eps_{k+1}/\eps_k)N$, applying the inductive hypothesis to $G[V']$ means there is some $k$-dense set $U\subseteq V'$ with $|U|\geq \eps_kN'\geq  \eps_{k+1}N$. Since we suppose failure for $k+1$, this $U$ must not be $k+1$-dense. That is, 
there exist $A,B\subseteq U$ with $|A|,|B|\geq \eps_k|U|> \eps_k^2|V'|$ and $e(A,B)\leq \eps_{k+1}|U|^3$

We now proceed using iterative bipartitioning to construct a large sparse $n$-partite graph in $G$, from which we can sample an independent set of size $n$. We begin with $V_\emptyset=V(G)$. Then by $(\neg\rm P_k)$, there exist $A,B$ which we re-denote by $V_0,V_1$ of size at least $\eps_k^2N$ with fewer than $\eps_{k+1}N^3$ crossing edges.

Supposing $n=2^\ell$, we iterate this procedure $\ell$ times. For each binary word $w$ of length $t<\ell$, we have a set $V_w$ of size $|V_w|\geq \eps_k^{2t}N$. By choice of $\eps_{k+1}=\eps_k^{10\ell}$, we have $|V_w|\geq (\eps_{k+1}/\eps_k)N$, so there exist disjoint $V_{w0},V_{w1}\subseteq V_w$ of size $|V_{w0}|,|V_{w1}|>\eps_k^2|V_w|\geq \eps_k^{2t+2}N$ with fewer than $\eps_{k+1}N^3$ crossing edges.

After $\ell$ steps, we have a family of $n$ sets $\cA=\{V_w\mid w\in \{0,1\}^\ell\}$, each of size at least $\eps_k^{2\ell}N$ (so they are nonempty). Consider the induced $n$-partite graph $G[\cA]$. Every edge of $G[\cA]$ is contained one of the cuts $V_{w0},V_{w1}$, and there are exactly $\sum_{t<\ell}2^t=n-1$ such cuts. Thus, we have
\[|E(G[\cA])|< n\eps_{k+1}N^3.
\]
Note that $n=\eps_0^{-\ell}$ so the $n$ factor is small. Form a set $S$ by choosing one vertex uniformly at random from each $A_w'\in \cA$. By linearity of expectation, we have
\[
\Ex[|E(G[S])|]< \frac{n\eps_{k+1}N^3}{(\eps_k^{2\ell}N)^3} <  \eps_{k+1}/\eps_k^{7\ell} <1,
\]
using that $\eps_{k+1}<\eps_k^{7\ell}$. Then there is some choice of $S$ which is an independent set of size $n$ in $G$, a contradiction.
\end{proof}

Next, we prove \cref{lem:fwdqdpairexists} by a two-sided restriction argument similar in spirit to \cref{prop:hitting}. For this, we need to show that quasi-denseness is preserved for large subgraphs.

\begin{proposition}[Scope Restriction]\label{prop:quasidensescope} Let  $\alpha_1<\alpha_2$ and $d\geq 3$ and let $\beta\in [\alpha_1,\alpha_2-\alpha_1]$. Let $U,W\subseteq V(G)$ for some $3$-graph $G$.
\begin{itemize}
    \item If $U$ is $(\alpha_1,\alpha_2,d)$-quasi-dense and $U'\subseteq U$ has size $|U'|\geq 2^{-\beta}|U|$, then $U'$ is $(\beta,\alpha_2-\beta,2d)$ quasi-dense.
    \item If $(U,W)$ is $(\alpha_1,\alpha_2,d)$ forward-quasi-dense and $U'\subseteq U$, $W'\subseteq W$ have sizes $|U'|\geq 2^{-\beta}|U|$ and $|W'|\geq 2^{-\beta}|W|$ then  $(U',W')$ is $(\beta, \alpha_2-\beta, 2d)$ forward-quasi-dense.
\end{itemize}
\end{proposition}

\begin{proof}For the first part, let $\gamma\in [\beta,\alpha_2-\beta]$ and let $A,B\subseteq U'$, have size $|A|,|B|\geq 2^{-\gamma} |U'|> 2^{-(\gamma+\beta)}|U|$. Since $\gamma+\beta\leq \alpha_2$, we have  $e(A,B)\geq 2^{-d(\gamma+\beta)}|U|^3\geq 2^{-2d\gamma}|U'|^3$, as desired. The second version is identical.
\end{proof}

Now, we prove that forward pairs exist.

\begin{proof}[Proof of \cref{lem:fwdqdpairexists}]
Suppose $U$ is $(\alpha, \alpha d^K,d)$ quasi-dense, where $K=4k^2$ and $k\geq 3$. By definition of quasi-dense-ness, we certainly have $|U|\geq 2$. Partition $U$ arbitrarily into balanced parts $A,B$. If $(A,B)$ is $(\alpha,\alpha d^{3k^2},d^{3k})$ forward-quasi-dense, then we are done with $D=d^{3k}$. Otherwise, there exist $X\subseteq A$ and $Y\subseteq B$ and $\beta \in [\alpha, \alpha d^{3k^2}]$ such that $|X|=|Y|= 2^{-\beta}|A|\geq 2^{-\beta-2}|U|$ and 
\[
e(X,X,Y)\leq  2^{-\beta d^{3k}}|U|^3= 2^{-\beta(d^{3k}-3)}|X|^2|Y|\leq 2^{-\beta d^{3k-1}}|X|^2|Y|
\]
By scope restriction (\cref{prop:quasidensescope}), with $\beta'=\beta+2$, each of $X,Y,X\cup Y$ are $(\beta', \alpha d^K-\beta', 2d)$ quasi-dense. We claim that $(Y,X)$ is $(\beta', \beta' d^{2k},d^2)$ forward-quasi-dense. Note that the upper limit of this range satisfies $\beta' d^{2k}\leq \alpha d^{3k^2+2k}< \alpha d^K-\beta'$. Now, let $\gamma\in [\beta',\beta' d^{2k}]$ and let $W_1\subseteq X$ and $W_2\subseteq Y$ with $|W_1|,|W_2|\geq 2^{-\gamma}|X|$. We have 
\begin{align*}
e(W_2,W_2,W_1)&\geq 2^{-2d\gamma}|Y|^2|X|-e(X,X,Y)\\
&\geq (2^{-2d\gamma}-2^{-\beta d^{3k-1}})|Y|^2|X|\\
&\geq 2^{-d^2\gamma}|Y|^2|X|,
\end{align*}
because $-2d\gamma> -\beta d^{3k-1}$ and $-d^2\gamma \leq -2d\gamma-1$. Thus $(Y,X)$ is forward-quasi-dense, as desired.
\end{proof}

Finally, since we are able to form a single forward-quasi-dense pair, we repeat this to construct a sequence. 

\begin{proof}[Proof of \cref{lem:forwarddensesequence}]
We proceed by induction on $t$. The case $t=1$ is proven with $\alpha=1$ by \cref{lem:quasidenseexists} (it then also holds for $\alpha=3$). For the inductive case, choose $K=8k^2$ and $C$ large enough that we have $W_1,\ldots, W_t$ satisfying the desired properties for $(\alpha, \alpha d^K,d)$. By \cref{lem:fwdqdpairexists} there exists $\beta\in [\alpha, \alpha d^{4k^2}]$ and $D$ and sets $U_t,U_{t+1}\subseteq W_t$ such that $|U_t|=|U_{t+1}|\geq 2^{-\beta}|W_t|$ and $(U_t,U_{t+1})$ is $(\beta, \beta D^k,D)$ forward-quasi-dense. For $i<t$ choose $U_i\subseteq W_i$ of size $|U_t|$ arbitrarily. By \cref{prop:quasidensescope}, all sets $U_i$ and pairs $(U_i,U_j)$ are $(\beta, \beta D^k,D)$ (forward-)quasi-dense with room to spare, so the sequence $U_1,\ldots,U_{t+1}$ suffices.
\end{proof}

\subsubsection{Counting via quasi-dense-ness.}

The following lemma essentially completes the proof of \cref{thm:fwddoubling}, by showing that once we construct a forward-quasi-dense sequence via \cref{lem:forwarddensesequence}, we have many copies of every forward-doubleable $3$-graph.

Recall that if $H$ is forward-doubleable with blocks $(A_i)_{i=1}^t$ and we have some disjoint $U_1,\ldots, U_t\subseteq V(G)$, we say that a $(U_i)_i$ copy of $H$ is a copy $\sigma$ with $\sigma(A_i)\subseteq U_i$ for all $i\in [t]$. Say that the sequence $(U_i)_i$ is $(\alpha_1,\alpha_2,d)$ forward-quasi-dense if every pair $(U_i,U_j)$ is for $i<j$.

\begin{lemma}\label{lem:qdcounting} Let $H$ be forward doubleable with $t$ blocks. There exists $k$ such that if $G$ is a 3-graph, $U_1,\ldots, U_t\subseteq G_j$ have size $N$, and $(U_i)_i$ is $(\alpha, \alpha d^k,d-2)$ forward-quasi-dense for $\alpha,d\geq 3$, then $G$ contains at least $2^{-\alpha d^k}N^{|H|}$ many $(U_i)_i$ copies of $H$.
\end{lemma}
\begin{proof}By induction on $|H|$. If $H$ is independent with singleton blocks, then there are $N^{|H|}$ many copies. Now, suppose $H$ satisfies the lemma with parameter $k\geq 1$. Suppose $(U_i)_i$ is $(\alpha, \alpha d^{k+1},d-2)$ forward-quasi-dense, so $G$ contains at least $2^{-\alpha d^k}N^{|H|}$ many $(U_i)_i$ copies of $H$. Let $u\in A_i$ and $w\in A_j$ with $i<j$; we may assume $u,w$ are contained in no common edge. For each $(U_i)_i$ copy $\sigma$ of $H-u-w$ define $U(\sigma)\subseteq U_i$ and $W(\sigma)\subseteq U_j$ as in Lemmas~\ref{lem:doublingblow-up} and \ref{lem:countingagainstBnn} where for all $x\in U(\sigma),y\in W(\sigma)$ the map $\sigma\cup \{(u\mapsto x),(w\mapsto y)\}$ forms a copy of $H$, and for every edge $xx'y\in e(U(\sigma),U(\sigma),W(\sigma))$, the map $\sigma\cup \{(u\mapsto x),(u'\mapsto x'),(w\mapsto y)\}$ forms a copy of $H[u,w]$. We have
\[
2^{-\alpha d^k}N^{|H|}\leq \sum_{\sigma\in \Sigma}|U(\sigma)||W(\sigma)|
\]
where $\Sigma$ is the set of $(U_i)_i$ copies of $H-u-w$ and if $M$ is the number of $(U_i)_i$ copies of $H[u,w]$ then
\[
M=\sum_{\sigma\in \Sigma}e(U(\sigma),U(\sigma),W(\sigma))
\]
Split $\Sigma$ into two cases, where $\sigma\in \Sigma_1$ if $|U(\sigma)|,|W(\sigma)|\geq 2^{-\alpha d^{k}-1}N$ and otherwise $\sigma\in \Sigma_2$. Then because $|\Sigma|\leq N^{|H|-2}$ we have
\[
\sum_{\sigma\in \Sigma_1}|U(\sigma)||W(\sigma)|\geq 2^{-\alpha d^k}N^{|H|}-|\Sigma_2|\cdot 2^{-\alpha d^k-1}N^{2}\geq 2^{-\alpha d^k-1}N^{|H|}.
\]
Since $|U(\sigma)||W(\sigma)|\leq N^2$ for all $\sigma\in \Sigma$, we have $|\Sigma_1|\geq 2^{-\alpha d^k-1}N^{|H|-2}$. Finally, because $(U_i,U_j)$ is forward-quasi-dense,
\[
M\geq \sum_{\sigma\in \Sigma_1}e(U(\sigma),U(\sigma),W(\sigma)) \geq (2^{-\alpha d^k-1}N^{|H|-2})\cdot (2^{(-\alpha d^k-1)\cdot (d-2)}N^3)\geq 2^{-\alpha d^{k+1}}N^{|H|+1},
\]
where we use that $\alpha\geq 1$ and $(-d^k-1)+(-d^k-1)(d-2)\geq -d^{k+1}$ for $d\geq 3$ and $k\geq 1$. This completes the claim with parameter $k+1$.
\end{proof}

Putting together these ingredients, we conclude the main theorem of this section.

\begin{proof}[Proof of \cref{thm:fwddoubling}]
Fix forward-doubleable $H$ with $t$ blocks. Let $k$ be given by \cref{lem:qdcounting}, and $C=C(t,2k)$ be given by \cref{lem:forwarddensesequence}. Now, let $G$ be an $N=2^{\log^Cn}$-vertex $3$-graph with $\alpha(G)<n$. Then \cref{lem:forwarddensesequence} gives disjoint subsets $U_i,\ldots, U_t$ of size $M$ which are $(\alpha,\alpha d^{2k},d)$ forward-quasi-dense. Note that $(d+2)^k<d^{2k}$, so then \cref{lem:qdcounting} gives $2^{-\alpha (d+2)^k}M^{|H|}$ copies of $H$ in $G$, and thus $r(H,n)\leq N$.
\end{proof}

\subsection{The Fano plane and other special cases}\label{sec:fanoandspecial}

In this section, we apply the technical theorems of Section \ref{sec:upperbounds} to deduce Theorems \ref{thm:fanominus}, \ref{thm:fanoplane}, \ref{thm:fanovsBn}, and \ref{thm:iteratedbicone}. It suffices to prove that the corresponding 3-graphs are doubleable and forward-doubleable.

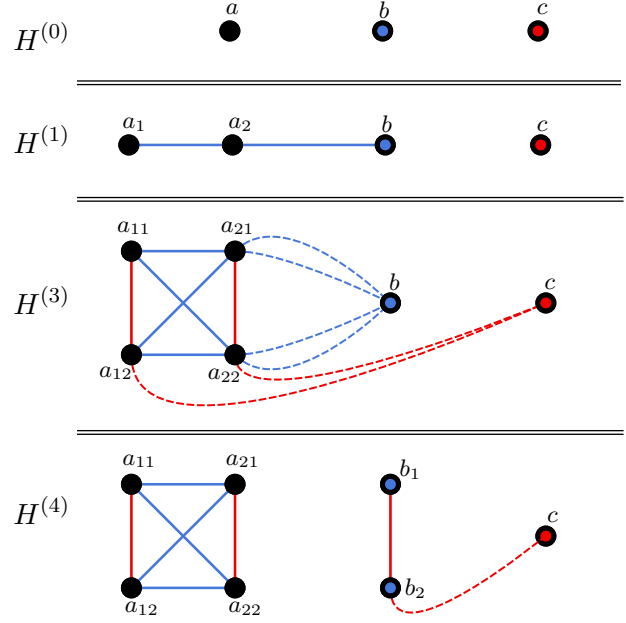
\begin{figure}
\centering
\input{tikz_diagrams/fanodoubling2}
\vspace{-15pt}
\caption{ How to obtain the Fano plane via forward-doubling (\cref{lem:fanofwddoubleable}) We have $H^{(1)}=H^{(0)}[a,b]$ and $H^{(3)}=H^{(1)}[a_1,c][a_2,c]$ and $H^{(4)}=H^{(3)}[b,c]$. We interpret vertices $b$ and $c$ as the `sources' of all blue and red edges respectively. Thus we interpret the last figure by extending each blue line to include each of $b_1,b_2$ and extending each red line to include $c$.
}
\label{fig:buildingfano}
\end{figure}

\begin{lemma}\label{lem:fanofwddoubleable} The Fano plane is forward-doubleable.
\end{lemma}
Using this lemma, \cref{thm:fanoplane} follows from \cref{thm:fwddoubling} and \cref{thm:fanovsBn} follows from \cref{thm:fwddoublingvsBn}.

\begin{proof}We construct a 3-graph $H^{(n)}$ with blocks $A<B<C$. Let $H^{(0)}$ be the independent graph with blocks $\{a\}<\{b\}<\{c\}$. Let $H^{(1)}=H^{(0)}[a,b]$, where we double $a$ into $a_1,a_2$. Then we have
\begin{align*}
V(H^{(1)}) &= \{a_1,a_2\}\cup \{b\}\cup \{c\}\\
E(H^{(1)}) &= \{a_1a_2b\}.
\end{align*}
Then let $H^{(2)}=H^{(1)}[a_1,c]$, where we double $a_1$ into $a_{11},a_{12}$, and let $H^{(3)}=H^{(2)}[a_2,c]$, where we double $a_2$ into $a_{21},a_{22}$. This blows up the link of $b$ into a $K_{2,2}$ while adding two new edges incident to $c$, so we have
\begin{align*}
V(H^{(3)}) &= \{a_{11},a_{12},a_{21},a_{22}\}\cup \{b\}\cup \{c\}\\
E(H^{(3)}) &= \{a_{1i}a_{2j}b\}_{i,j} \cup\{a_{11}a_{12}c,a_{21}a_{22}c\}.
\end{align*}
Finally, let $H^{(4)}=H^{(3)}[b,c]$, where we double $b$ into $b_1,b_2$. This duplicates all edges incident to $b$ and adds the edge $b_1b_2c$, so we have
\begin{align*}
V(H^{(4)})&=\{a_{11},a_{12},a_{21},a_{22}\}\cup \{b_1,b_2\}\cup \{c\}\\
E(H^{(4)})&=\{a_{1i}a_{2j}b_k\}_{i,j,k} \cup\{a_{11}a_{12}c,a_{21}a_{22}c\}\cup \{b_1b_2c\}.
\end{align*}
This forms a copy of the Fano plane when we delete the edges $a_{12}a_{21}b_2,a_{11}a_{22}b_2,a_{12}a_{22}b_1,a_{11}a_{21}b_1$; see Figures~\ref{fig:buildingfano} and \ref{fig:fano}.
\end{proof}

\begin{figure}
    \centering
\input{tikz_diagrams/fano2}
\caption{The Fano plane, with labels and colored vertices and edges according to Figure~\ref{fig:buildingfano} and \cref{lem:fanofwddoubleable}.}
\label{fig:fano}
\end{figure}
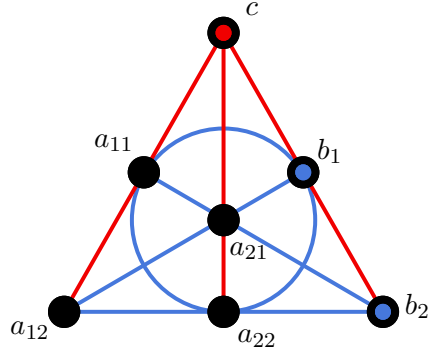

\begin{lemma}\label{lem:fanominusdoubleable} The Fano plane minus an edge can be obtained from an iterated tripartite graph by one doubling.
\end{lemma}
Using this lemma, \cref{thm:fanominus} follows from \cref{cor:tridoubling}.
\begin{proof} Let $H^{(3-)}=H^{(3)}-\{a_{11}a_{12}c\}$, where $H^{(3)}$ is as in \cref{lem:fanofwddoubleable}, illustrated in Figure~\ref{fig:buildingfano}. This $3$-graph is iterated tripartite, by the coloring 
\begin{align*}
\chi(a_{11})=\chi(a_{12})&=1\\
\chi(a_{21})=\chi(a_{22})=\chi(c)&=2\\
\chi(b)&=3.
\end{align*}
Then every edge of $H^{(3-)}$ is $\chi$-rainbow except for $a_{21}a_{22}c$, which is color $2$. Now let $H^{(4-)}=H^{(3-)}[b,c]$, so also $H^{(4-)}=H^{(4)}-\{a_{11}a_{12}c\}$. Since $H^{(4)}$ contains the Fano plane, surely $H^{(4-)}$ contains the Fano plane minus one edge.
\end{proof}
It remains to prove \cref{thm:iteratedbicone}, for which we define the bicone of a 3-graph as in \cite[Defn.\ 2.1]{quasipoly}.
\begin{definition}\label{defn:bicone} Let $H$ be a $3$-graph and $X,Y$ be a bipartition of $V(H)$. The bicone, denoted $\rm{bicone}(H;X,Y)$, is a $3$-graph with vertex set $\{x,y,z\}\sqcup V(H)$ and edge set
\[
E(\rm{bicone}(H;X,Y))=\{xyz\}\cup \{zxv\mid v\in Y\}\cup\{zyw \mid w\in X\} \cup E(H).
\]
When the bipartition $X,Y$ is clear, we simply write $\rm{bicone}(H)$, having the implicit bipartition $X'=X\cup \{x,z\}$ and $Y'=Y\cup \{y\}$.
\end{definition}
\begin{definition}\label{defn:iteratedbicone} For $t\geq 0$, define the \emph{iterated bicone} $H_t$ by letting $H_0=\emptyset$ and $H_{t+1}=\rm{bicone}(H_t)$.
\end{definition}
For example, $H_1$ is a single edge $xyz$ with implicit bipartition $\{x,z\},\{y\}$, and $H_2$ can be defined with vertex set $\{x,y,z,a,b,c\}$ with edge set
\[
E(H_2)=\{xyz\}\cup \{zxb\}\cup \{zya,zyc\}\cup \{abc\},
\]
and bipartition $\{x,z,a,c\},\{y,b\}$.

\begin{observation}
By \cref{defn:fwddoubling}, the class of forward-doubleable graphs is closed under vertex duplication, similar to \cref{prop:doubleableiteratedblow-ups}. To duplicate vertex $u$, begin with an auxiliary vertex $x$ in its own block which is greater than all others, then double the pair $(u,x)$
\end{observation}
\begin{lemma}\label{lem:iterbiconefwd} For $t\geq 0$, the iterated bicone $H_t$ is forward-doubleable.
\end{lemma}

Using this lemma, \cref{thm:iteratedbicone} follows from \cref{thm:fwddoublingvsBn}. To prove \cref{lem:iterbiconefwd}, the base case $H_0$ is trivial so it suffices to show the following.

\begin{claim}\label{claimbiconepreservesfwd} If a $3$-graph $H$ is forward-doubleable with two blocks $X<Y$, then $\rm{bicone}(H;X,Y)$ is forward-doubleable with blocks $X\cup \{x,z\}<Y\cup \{y\}$.
\end{claim}
\begin{proof}
Begin with the independent graph $H^{(0)}$ with blocks $\{x\}<\{y\}$. Let $H^{(1)}=H^{(0)}[x,y]$, so 
\[
V(H^{(1)})=\{x,z\}\cup \{y\},
\]
and $E(H^{(1)})= \{xyz\}$. Next, duplicate vertex $x$ into $a$ and vertex $y$ into $b$, and delete the edge $zab$ to obtain graph $H^{(2)}$, so
\begin{align*}
V(H^{(2)})&=\{x,z,a\}\cup \{y,b\}\\
E(H^{(2)})&=\{zxy,zya,zxb\}.
\end{align*}
Since we have $\{a\}<\{b\}$ and there is no pair incident to $a,b$, by assumption there is some sequence of forward doubling, vertex duplication, and taking subgraphs which replaces $H^{(2)}[a,b]$ with a copy of $H$ with blocks $X,Y$. Then this graph, together with the vertices $x,y,z$ and edges $\{xyz\}\cup zxY\cup zyX$ forms a copy of $\rm{bicone}(H)$.
\end{proof}

\section{Exponential lower bounds and the chain}\label{sec:explower}

In this section, we will prove the following and use it to deduce the results of Section \ref{sec:intro_explower}.

\begin{theorem}\label{thm:forwardlaminar} If a $3$-graph $H$ is not forward laminar, then $r(H,n)=\exp(\Omega(n^{1/5}))$.
\end{theorem}

Let us define terms. For a set $V$, recall that a \textit{laminar family} on $V$ is a set family $\cF\subseteq 2^V$ such that any two sets of $\cF$ are disjoint or nested. That is, for all non-disjoint $A,B\in \cF$, either $A\subseteq B$ or $B\subseteq A$. Then we have the following.

\begin{definition}
A $3$-graph $H$ with vertex set $V$ is \textbf{laminar} if there exists a laminar family $\cS\subseteq 2^V$ such that for each tight component $H'\subseteq H$, there are disjoint sets $S_1,S_2,S_3\in \cS$ with $E(H')\subseteq K_{S_1,S_2,S_3}^{(3)}$.
\end{definition}

\noindent This is slightly more general than the property that tight component is tripartite and additionally the family of `parts of tight components of $H$' is laminar. Next, we define the stronger property which \cref{thm:forwardlaminar} concerns. The following definition is also fairly standard.

\begin{definition}For a laminar family $\cF\subseteq 2^V$, the \textbf{containment tree} $T_\cF$ is the rooted tree obtained by the Hasse diagram of $\cF\cup\{V\}$ ordered by containment. Specifically, $T_{\cF}$ has vertex set $\cF\cup \{V\}$ with root $V$ and for each $A\in \cF\setminus\{V\}$, the parent of $A$ is the smallest set in $V(T_{\cF})$ properly containing $A$.
\end{definition}

We define a \textbf{binary} laminar family $\cF$ as one where the containment tree $T_{\cF}$ is binary. Note that by adding internal vertices, any laminar $3$-graph has some corresponding laminar family which is binary. Similarly, the following definitions can be extended to an equivalent non-binary version. Let $\cF$ be a binary laminar family. For sets $S_1,S_2\in \cF$, let $\Delta(S_1,S_2)$ be the smallest set in $\cF$ containing both $S_1$ and $S_2$. For elements $x,y\in V$ we define $\Delta_\cF(x,y)$ similarly.

\begin{definition}
Suppose that the tree $T_{\cF}$ is an ordered tree with ordering $<$. We say that disjoint sets $S_1<S_2<S_3\in \cF$ is a \textbf{forward triple} if $\Delta(S_1,S_2)\subseteq \Delta(S_2,S_3)$. Finally, we say that a $3$-graph $H$ with vertex set $V$ is \textbf{forward laminar} if there exists a binary laminar family $\cS\subseteq 2^V$ and some ordered containment tree $T_{\cS}$ such that for every tight component $H'\subseteq H$, there is a forward triple $S_1,S_2,S_3\in \cS$ with $E(H')\subseteq K_{S_1,S_2,S_3}^{(3)}$.
\end{definition}

Since forward laminar is a stronger condition than just laminar, \cref{thm:forwardlaminar} also applies to the simpler class of non-laminar $3$-graphs. It is logically necessary by Theorems~\ref{thm:fwddoubling},\ref{thm:forwardlaminar} that any iterated tripartite or forward-doubleable 3-graph $H$ must be forward laminar, (e.g.\ Figure~\ref{fig:H2laminar}). For intuition, observe that $H$ is iterated tripartite if $H$ is forward laminar with some family $\cF$ such that any two tight components $H_1,H_2$ with corresponding triples $C_1,C_2,C_3\in \cF$ and $S_1,S_2,S_3\in \cF$
respectively, we have either $\{S_1,S_2,S_3\}=\{C_1,C_2,C_3\}$ or else for some $i$ we have $C_1\cup C_2\cup C_3\subseteq S_i$ or $S_1\cup S_2\cup S_3\subseteq C_i$.  A similar description can be given for forward-doubleable 3-graphs, where the triples are partially nested in some order.

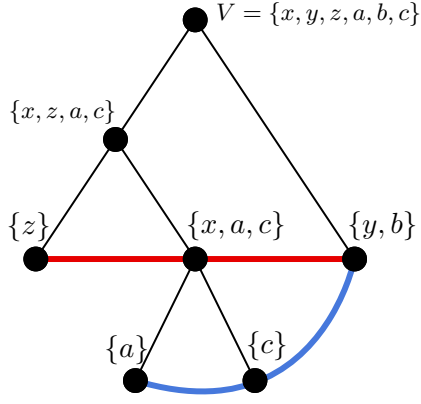
\begin{figure}
    \centering
\input{tikz_diagrams/H2laminar}
\caption{A binary laminar tree $T_{\cF}$ such that the triples $\{z\},\{x,a,c\},\{y,b\}$ and $\{a\},\{c\},\{b\}$ are forward. This tree demonstrates that the iterated bicone $H_2$ (\cref{defn:iteratedbicone}) with components $\{zxy,zxb,zay,zcy\}$ and $\{abc\}$ is forward laminar. See also \cite[Fig. 1,3]{quasipoly}
}
\label{fig:H2laminar}
\end{figure}

\subsection{Proof of Theorem \ref{thm:forwardlaminar}}

We begin with the construction. Let $N=2^d$ where $d\leq 0.1n^{1/5}$, and let $V=[N]$, we identify $x\in V$ with its binary representation in $\{0,1\}^d$. For distinct $x,y\in V$ let $\Delta(xy)$ be the longest common prefix of the binary representations of $x,y$ and $\delta(xy)=|\Delta(xy)|$. Sample a function $\phi:\binom N2\to \{0,\ldots, d-1\}$ uniformly at random. For vertices $x<y<z$ we put an edge $xyz\in E(G)$ if and only if:
\begin{itemize}
\item We have $\delta(xy)>\delta(yz)$; we say $xyz$ is \textit{forward}. Note that for any forward triple we have $\delta(xz)=\delta(yz)$.
\item We have $\phi(xy)=\delta(yz)$ and $\phi(xz)=\phi(yz)=\delta(xy)$.
\end{itemize}
Equivalently, we put $xyz\in E(G)$ if and only if $xyz$ is forward and
\begin{equation}\label{eq:temp2}
(\delta(xy),\delta(yz))
=(\delta(xy),\phi(xy))
=(\phi(xz),\delta(xz))
=(\phi(yz),\delta(yz)).
\end{equation}

\begin{lemma}\label{lem:explowerbound_indset} As defined above, $G$ has no independent set of size $n$ with positive probability.
\end{lemma}
To prove this, we first need to show that the forward triples do not admit large sparse sets. It is easy to see that the largest set with no forward triples is of the form $\{0000,1000,1100,1110,1111\}$, so the auxiliary 3-graph $\cI_d$ of forward triples in $\{0,1\}^d$ satisfies $\alpha(\cI_d)=d+1$. Then by standard Tur\'an bounds, every subset of order $n> 2d$ contains at least $\Omega(n^3/d^2)$ forward triples. We improve this bound below, which improves our final hyperexponent.

\begin{claim}\label{claim:forwardtriples}
If $V=\{0,1\}^d$, and $n>5d$ then every $n$-element subset of $V$ contains at least $0.01n^3/d$ forward triples.
\end{claim}
\begin{proof}
For $d\geq 1$, $k\in \Z$ let $g(k,d)$ be the minimum number of forward triples in a subset $S$ of $\{0,1\}^d$ with size $|S|=k+d+1$. If $k+d+1> 2^d$ we set $g(k,d)=\infty$, as no subset can exist. Set $c=0.02$ and we claim that $g(k,d)\geq ck^3/d$. The base case $d=1$ is vacuous and $d=2$ satisfies $g(1,2)=2$. For $d\geq 3$, by choosing $x$ elements of $S$ to have first bit $0$, we have exactly
\begin{align*}
g(k,d)&\geq \min_{1\leq x\leq (k+d+1)/2}\binom x2(k+d+1-x)+g(k-x+1,d-1)\\
&\geq \min_{y\geq 0}\frac{ky^2}{4}+g(k-y,d-1),
\end{align*}

where we change variables from $x\mapsto y+1$, and we use that 
\[\binom{y+1}2(k+d+1-x)> \frac 12y^2\cdot \max(y,k-y)\geq \begin{cases}
\frac 12y(k/2)\cdot (y)&\text{ if }y\geq k/2\\
\frac 12y^2\cdot (k/2)&\text{ if }y\leq k/2.
\end{cases}
\]
Now we use this recurrence relation to prove that for some $c>0$, we have $g(k,d)\geq ck^3/d$ for all $k,d$. When $k\leq 0$ we have $g(k,d)=0$ and when $k=1$ it is easy to verify $g(1,d)=2$ for all $d\geq 2$. Finally for all $k,d\geq 2$ and $y\in \R$, we have
\[
\begin{aligned}
g(k,d)\ge \frac{k y^2}{4} + \frac{c (k-y)^{3}}{d-1} 
&\geq \frac{ky^2}{4}+\frac{c k^3 -3c k^2 y}{d-1}\\
&\ge \frac{c k^{3}}{d}
  + \left(
      \frac{k y^2}{4}
      + \frac{c k^3}{d^{2}}
      - \frac{3 c k^2 y}{d-1}
    \right) \\
&\geq \frac{c k^3}{d}+k\Big(\Big(\frac y2\Big)^2+\Big(\frac{\sqrt c k}{d}\Big)^2-12\sqrt c \Big(\frac y2\Big)\Big(\frac{\sqrt c ky}{d}\Big)\bigg)\\
&\ge \frac{c k^3}{d},
\end{aligned}
\]
as long as $c\leq  1/36$. Then as long as $|S|=n>5d$, we have proven that $S$ contains at least $c(n-d-1)^3/d\geq cn^3/2d$ forward triples, as desired.
\end{proof}

\begin{proof}[Proof of \cref{lem:explowerbound_indset}]
Let $G$ be given by the construction above. We have $n\geq 100d^5$; let $S\subseteq V$ have size $n$. By Claim \ref{claim:forwardtriples}, $S$ (deterministically) has at least $cn^3/d$ forward triples, where $c=0.01$. By greedily picking out edges, we obtain a set $T$ of $cn^2/3d$ forward triples which are pair-disjoint; i.e. $T$ is a linear hypergraph. For each triple $x<y<z\in T$, recall that $xyz\in E(G)$ if and only if $\phi(xz)=\phi(yz)=\delta(xy)$ and $\phi(xy)=\delta(yz)$, occurring with probability $d^{-3}$. Moreover, each triangle in $T$ is independent, so
\[
\PP[S\text{ is independent in $G$}]\leq (1-d^{-3})^{cn^2/3d}\leq \exp(-cn^2/3d^4).
\]
Then taking a union bound and recalling that $\log_2N=d$, we have
\[
\PP[\alpha(G)\geq n]\leq \binom Nne^{-cn^2/3d^4}\leq \exp(nd -cn^2/3d^4)<1,
\]
using that $n\geq (10d)^5$. 
\end{proof}

Next, we prove that $G$ avoids any graph which is not forward laminar.

\begin{lemma}\label{lem:Glaminar}As defined above, every subgraph of $G$ is forward laminar.
\end{lemma}
\begin{proof}It suffices to show that $G$ is forward laminar, as this property is inherited by subgraphs. Define the laminar family $\cF=\{V_w\mid w\in \{0,1\}^{\leq d}\}$ where
\[V_w=\{x\in V\mid w\text{ is a prefix of the binary representation of $x$}\}\]
In particular, for a word $w$ with $|w|<d$ the set $V_w$ partitions into $V_{w0},V_{w1}$. We show that this is a valid family for $G$, using the ordering on $[N]$ where $V_{w_1}<V_{w_2}$ if neither $w_1,w_2$ is a subword of the other and $w_1<w_2$ in lexicographical order. Note that $\cF$ is a binary laminar family with $\Delta_\cF(x,y)=V_{\Delta(xy)}$.

\begin{claim}\label{claim:Glaminar} If $H\subseteq G$ is tightly connected, then there exist binary words $v,w\in \{0,1\}^*$ such that $H\subseteq K[V_{v0w0},V_{v0w1},V_{v1}]$. In particular, $H$ is tripartite, and the triple $V_{v0w0}<V_{v0w1}<V_{v1}$ is forward.
\end{claim}
\begin{subproof}Let $xyz \in E(H)$ with $x<y<z$, and let $v=\Delta(xz)$ be the longest common prefix of $x,y,z$, so $|v|=\delta(xz)=\delta(yz)$. Note that because $x,y<z$, we have $x,y\in V_{v0}$ and $z\in V_{v1}$. Then because $\delta(xy)>\delta(yz)$, we define $w\in \{0,1\}^*$ such that $v0w=\Delta(xy)$ is the longest common prefix of $x,y$, having $|v0w|=\delta(xy)$. Because $x<y$, we have $x\in V_{v0w0}$ and $y\in V_{v0w1}$. Now, we wish to show that every other edge in $H$ has this property. Since $H$ is tightly connected, it suffices to show this holds for some other edge $abc\in E(H)$ with $|xyz\cap abc|=2$. Intuitively, this is guaranteed by Equation \ref{eq:temp2} since for a single pair of vertices among $xyz$, we have enough information using $\phi$ to recover the branch points $\Delta(xy)$ and $\Delta(xz)$. Formally, we split into cases.

Suppose $a<b<c$ and $xyz\cap abc = \{x,y\}$. Then because $\phi(xy) = \delta(yz)<\delta(xy)$ and also $abc$ is forward so $\delta(ab)>\delta(bc)$, we must have $x=a$ and $y=b$, so $a\in V_{v0w0}$ and $b\in V_{v0w1}$. Then since $\phi(ab)=\delta(bc)$ and $a,b\in V_{v0}$, we must have $c\in V_{v1}$, as desired.

Suppose that instead $xyz\cap abc=\{y,z\}$. Then because $\phi(yz)=\delta(xy)>\delta(yz)$ and also $abc$ is forward we must have $c=z$. Moreover observe that $\delta(ab)=\phi(yz)=\delta(xy)$, so using that $a<b$, the $\delta(xy)+1\nth$ bit of $a$ must be $0$, and the $\delta(xy)+1\nth$ bit of $b$ is $1$. Since this bit of $y$ is $1$, we must have $y=b$. Then similarly since $b\in V_{v0w1}$ we have $a\in V_{v0w0}$, as desired.

Finally, suppose $xyz\cap abc=\{x,z\}$. Once again we have $c=z$, and since $\delta(xy)+1\nth$ bit of $x$ is $0$ we have $x=a$. Once again since $a\in V_{v0w0}$ we have $b\in V_{v0w1}$, as desired.

Since any two edges sharing two vertices must embed into the same forward triple, by propagating along tight paths we conclude that the entire tight component embeds into the same forward triple, as desired.
\end{subproof}
By \cref{claim:Glaminar}, every tight component $G$ embeds into a forward triple of $\cF$, so $G$ is forward laminar, and indeed every subgraph of $G$ is forward laminar.
\end{proof}

For completeness, we conclude our main theorem.
\begin{proof}[Proof of \cref{thm:forwardlaminar}]
Let $H$ be a 3-graph which is not forward laminar and $n$ sufficiently large. Then put $d=\floor{0.1n^{1/5}}$ and construct $G$ as above with size $N=\exp(\Omega(n^{1/5}))$. By \cref{lem:explowerbound_indset}, $G$ has no independent set of size $n$, and by \cref{lem:Glaminar}, $G$ contains no copy of $H$. Thus $r(H,n)>N$.
\end{proof}

\subsection{Properties of forward laminar graphs.}\label{sec:propertiesofforwardlaminar}

In this section, we deduce \cref{thm:spanningcomponents} and a similar result (\cref{cor:chromaticnumber}) from \cref{thm:forwardlaminar}. To prove \cref{thm:spanningcomponents}, it suffices to show the following. 

\begin{lemma}\label{lem:forwardlaminarspanning} If $H$ is forward laminar and every tight component is spanning, then $H$ is tripartite.
\end{lemma}
The contrapositive of this lemma together with \cref{thm:forwardlaminar} gives \cref{thm:spanningcomponents}.

\begin{proof}
Let $\cF$ be a binary laminar family for $H$ and $T_{\cF}$ be an ordered containment tree. Let $H_1$, $H_2$ be two (nonempty) tight components of $H$. We have forward triples $S_1<S_2<S_3$ and $C_1<C_2<C_3$ for $H_1,H_2$ respectively. By assumption, $\bigcup_iS_i=\bigcup_iC_i=V(H)$. In order for $\{S_1,\ldots, C_3\}$ to be laminar, we must have either $S_1=C_1$ or $S_1=C_1\cup C_2$ or $S_1=C_1\cup C_2\cup C_3$. The last case implies $S_2=S_3=\emptyset$, which is impossible. In the second case, we must have $C_3=S_1\cup S_2$. But then since $C_3$ is an ancestor of $S_2$, we have $\Delta(S_1,S_2)\supseteq \Delta(S_1,C_3)=V$, and also $\Delta(S_2,S_3)\subseteq C_3$. Thus $\Delta(S_1,S_2)\supsetneq \Delta(S_2,S_3)$, a contradiction.

We conclude that $S_1=C_1$, and since $S_2\cup S_3=C_2\cup C_3$ we must also have $S_2=C_2$ and $S_3=C_3$. In conclusion, since every component has the same tripartition, $H$ is tripartite with tripartition $S_1,S_2,S_3$.
\end{proof}

Here, we prove an additional result not advertised in the introduction. For a $3$-graph $H$, define the \textit{strong chromatic number} $\chi_s(H)$ as the minimum number of colors in a \textit{strong vertex coloring}, a coloring where no edge contains two vertices of the same color. We have $\chi_s(H)=3$ if and only if $H$ is tripartite.

\begin{lemma}\label{lem:strongcoloring} If $H$ is NIT and laminar and has $t$ tight components then $\chi_s(H)\leq 2t$.
\end{lemma}
The contrapositive gives a Ramsey lower bound, also relating to NIT girth.

\begin{corollary}\label{cor:chromaticnumber} If $H$ is NIT, has $t$ tight components, and $\chi_s(H)\geq 2t+1$, then $r(H,n)\geq \exp(\Omega(n^{1/5}))$.
\end{corollary}
Before we prove \cref{lem:strongcoloring}, we need a structural claim about laminar $3$-graphs.

\begin{lemma}\label{lem:laminarcharacterization} If $H$ is a laminar $3$-graph with $t$ tight components $\{H_i\mid i\in [t]\}$, then there exists a laminar family $\cS=(S_{ij}\mid i\in [t],\ j\in [3])$ such that:
\begin{enumerate}[(1)]
\item For each $i\in [t]$, we have $E(H_i)\subseteq E(K_{S_{i1},S_{i2},S_{i3}}^{(3)})$; this requires $S_{i1},S_{i2},S_{i3}$ are pairwise disjoint.
\item For each non-leaf $S\in V(T_{\cS})$, $S$ has at most $3$ children, and $S$ is partitioned by its children.
\item If $S_{ij}$ and $S_{i'j'}$ are distinct leaves with the same parent, then $i=i'$.
\end{enumerate}
\end{lemma}
Note that $\cS$ may have repeated sets. By convention, if $S_{i_1j_1}=S_{i_2j_2}=\cdots = S_{i_kj_k}$ with $i_1<i_2<\cdots<i_k$, we define $T_{\cS}$ to have $k$ distinct vertices, where $S_{i_2j_2}$ is the only child of $S_{i_1j_1}$, and so on.

\begin{proof}
Choose $\cS$ satisfying (1) such that each $S_{ij}$ is inclusion maximal. By definition, there is some $\cS$ satisfying (1). Now we check that this satisfies properties (2) and (3). Suppose that $S\in \cS$ is a non-leaf with children $S_{i_1j_1},\ldots,S_{i_mj_m}$; without loss of generality let $i_1=j_1=1$. First, note that if $S$ is not partitioned by its children, then we could replace $S_{11}$ with $S_{11}\cup S\setminus \bigcup_{k=1}^mS_{i_kj_k}$, which maintains (1) and contradicts maximality. Next, suppose $m\geq 4$. By pigeonhole, there is some $k\in [2,m]$ such that $S_{12},S_{13}\not\subseteq S_{i_kj_k}$. Thus, we can replace $S_{11}$ with $S_{11}\cup S_{i_kj_k}$, which maintains (1) and contradicts maximality.

Finally, suppose that $S_{ij},S_{i'j'}$ are distinct leaves with the same parent, but $i\neq i'$. Then we can replace $S_{ij}$ with $S_{ij}\cup S_{i'j'}$, which maintains (1) and contradicts maximality.
\end{proof}

\begin{proof}[Proof of \cref{lem:strongcoloring}]
Let $H$ be a laminar 3-graph. Let $\cS$ be the family from \cref{lem:laminarcharacterization}. Say that $H$ is $\cS$-IT if there is an iterated $3$-coloring of $H$ which is constant on each leaf of $\cS$; otherwise $H$ is $\cS$-NIT. We prove the claim for any $H$ which is $\cS$-NIT, which is more general.

Now, $T_{\cS}$ is a rooted tree with $3t+1$ vertices such that every vertex has at most $3$ children. From this, it is easy to check that $T_\cS$ has $\ell\leq 2t+1$ leaves; label these $L_1,\ldots,L_\ell$. These leaves partition $V$; let $\chi$ be the $\ell$-coloring with classes $L_1,\ldots, L_\ell$.

We check that $\chi$ is a proper coloring of $H$. Suppose $xyz\in E(H_i)$ with $x\in S_{i1},y\in S_{i2},z\in S_{i3}$, where $\{S_{ij}\}_j$ are pairwise disjoint. Since no leaf set intersects multiple $S_{ij}$'s, the vertices $x,y,z$ receive distinct colors.

Finally, we show that $\ell\neq 2t+1$. Let $H$ be a minimal counterexample with $\ell=2t+1$, and observe that every non-leaf of $T_{\cS}$ must have $3$ children. In particular, there is some maximum depth non-leaf $S\in \cS$ which has three children, all of which are leaves. Then by condition (3), these children are $S_{i1},S_{i2},S_{i3}$ for some $i\in [t]$. Let $H'=H-H_i$, and $\cS'=\cS-\{S_{i1},S_{i2},S_{i3}\}$, so $\cS'$ has $2(t-1)+1$ leaves. If $H'$ were $\cS'$-IT, then there is an iterated coloring which is constant on $S$. Then extending this coloring to $S_{i1},S_{i2},S_{i3}$, also $H$ is $\cS$-IT, a contradiction. Then $H'$ is $\cS'$-NIT, contradicting minimality.
\end{proof}

\subsection{The chain}\label{sec:chain}

In this section, we deduce Theorems~\ref{thm:girth} and \ref{thm:chainclassification} from \cref{thm:forwardlaminar}. The following lemma reduces the full classification of chains to our previous technical theorems.

\begin{lemma}\label{lem:chainclassification}For $t\geq 3$, let $\rm{CH}_t$ be the chain of length $t$.
\begin{enumerate}[(I)]
\item\label{itm:chainnotlaminar} The chain $\rm{CH}_t$ is not forward laminar.
\item\label{itm:chainnoquotient} The chain minus an edge $\rm{CH}_t^-$ does not admit a nontrivial tight quotient.
\item\label{itm:chaindoubling} The chain minus an edge $\rm{CH}_t^-$ can be obtained from doubling a pair in some iterated tripartite 3-graph.
\item\label{itm:chainitertrip} The chain minus any two edges $\rm{CH}_t^{--}$ is iterated tripartite.
\end{enumerate}
\end{lemma}
Before proving these, we use this lemma to conclude the trichotomy by appealing to several results of this paper.

\begin{proof}[Proof of \cref{thm:chainclassification}]
Equation \eqref{eq:chainthm} follows from (I) and \cref{thm:forwardlaminar}. The lower bound for Equation \eqref{eq:chain-thm} follows from (II) and \cref{lem:quotient}. The upper bound for Equation \eqref{eq:chain-thm} follows from (III) and \cref{cor:tridoubling}. Finally, Equation \eqref{eq:chain--thm} follows from (IV) and the easy direction of \cref{conj:iteratedtripartite}.
\end{proof}

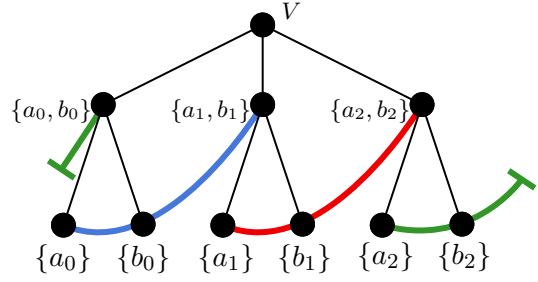
\begin{figure}
    \centering
\input{tikz_diagrams/chainlaminar}
\caption{This figure demonstrates that the chain $\rm{CH}_3$ is \textit{laminar}, even though it is not forward laminar. The components of the chain are tripartite, with partitions $\{a_0\},\{b_0\},\{a_1,b_1\}$ in blue, $\{a_1\},\{b_1\},\{a_2,b_2\}$ in red, and $\{a_2\},\{b_2\},\{a_0,b_0\}$ in green.
}
\label{fig:placeholder}
\end{figure}

Now we prove \cref{lem:chainclassification} one piece at a time.

\begin{proof}[Proof of (I)]
Suppose $\cF$ is a laminar family for $\rm{CH}_t$ with some ordered tree $T_{\cF}$. Then $\cF$ must contain sets $\{A_i,B_i,S_i\mid i\in [t]\}$ such that $A_iB_iS_{i+1}$ is a forward triple for all $i$ (in some order) and $A_i,B_i\subseteq S_i$. Up to symmetry, we can label the sets such that $A_i<B_i$ for each $i$. It cannot hold that $B_{i+1}<B_i$ for all $i$ modulo $t$, so there must be some $i$ (say $i=0$) where $A_0<B_0<S_{1}$. We claim by induction that for all $i$, $S_i$ is contained in the right (greater) subtree of $\Delta(B_0,S_1)$. For the base case we take $i=1$. Suppose that $S_i$ is contained in the right subtree of $\Delta(B_0,S_1)$. If $S_{i+1}$ is not, then we have $S_{i+1}<A_i<B_i$ and 
\[\Delta(S_{i+1},A_i)\supseteq \Delta(B_0,S_1)\supseteq \Delta(A_i,B_i),\]
contradicting that $S_{i+1}<A_i,B_i$ is forward. Finally, this means that $S_t=S_0$ is contained in the right subtree of $\Delta(B_0,S_1)$, a contradiction.
\end{proof}

\begin{proof}[Proof of (II)] Up to symmetry, let $\rm{CH}^-_t=\rm{CH}_t-a_{t-2}b_{t-2}b_{t-1}$. Let $\rm{CH}/\cP$ be a tight quotient (refer to \cref{defn:quotient}). For all $0\leq i\leq t-2$, since $\cP$ is tight $a_i,b_i$ must lie in the same part of $\cP$. Since $\cP$ is proper, for all $0\leq i\leq t-3$, the edges $a_ib_ia_{i+1},a_ib_ib_{i+1}$ imply that $a_i,b_i,a_{i+1},b_{i+1}$ must lie in the same part $P$ of $\cP$. Furthermore, the edge $a_{t-2}b_{t-2}a_{t-1}$ implies that also $a_{t-1}\in P$. Finally, the edge $a_{t-1}b_{t-1}a_0$ implies that $b_{t-1}\in P$, so $\cP=\{P\}$ is trivial.
\end{proof}

\begin{proof}[Proof of (III)]
Up to symmetry, let $\rm{CH}_t^-=\rm{CH}_t-a_{t-1}b_{t-1}b_0$. Define the linear chain 
\[\rm{LCh}_t=\rm{CH}_{t}-a_{t-1}b_{t-1}a_{0}-a_{t-1}b_{t-1}b_0-b_{t-1}.\]
This 3-graph forms a path-like shape instead of a cyclic one. We observe that the doubling $H[a_{t-1},a_0]$ duplicates the vertex $a_{t-1}$ into $a_{t-1},b_{t-1}$, and adds the edge $a_{t-1}b_{t-1}a_0$, giving exactly $\rm{CH}_t^-$. The fact that $H$ is iterated tripartite follows from the first case of (IV).
\end{proof}

\begin{proof}[Proof of (IV)]
There are two main cases. Suppose we delete a tight component, having $\rm{CH}_t^{--}=\rm{CH}_t-a_{t-1}b_{t-1}a_0-a_{t-1}b_{t-1}b_0$. In this case, we show $\rm{CH}_t^{--}$ is iterated tripartite by induction on $t$, for all $t\geq 2$. If $t=2$ then $\rm{CH}_t^{--}$ is tripartite. For the inductive case, define coloring $\chi$ by $\chi(a_0)=1$, $\chi(b_0)=2$, and $\chi(V-a_0-b_0)=3$. The induced subgraph on $S_3$ is precisely $\rm{CH}_{t-1}^{--}$, which is iterated tripartite by induction. Thus $\rm{CH}_{t}^{--}$ is iterated tripartite. Since $\rm{LCh}_t$ is a subgraph, this is also iterated tripartite. Note that the ends $a_0$,$a_1$ and $a_{t-1}$ have different colors.

In the other case, up to symmetry we have $\rm{CH}_t^{--}=\rm{CH}_t-a_{t-1}b_{t-1}b_0-a_{j-1}b_{j-1}b_j$ for some $j\in [t-1]$. Then $\rm{CH}_t^{--}$ is composed of two subgraphs:
\begin{align*}
H^{(1)} \coloneqq \rm{CH}_t^{--}[\{a_0,b_0,\ldots, a_{j-1}b_{j-1}a_j\}]\\
H^{(2)} \coloneqq \rm{CH}_t^{--}[\{a_j,b_j,\ldots, a_{t-1}b_{t-1}a_0\}].
\end{align*}
Observe that $H^{(1)}$ and $H^{(2)}$ are isomorphic to $\rm{LCh}_{j+1}$ and $\rm{LCh}_{t-j-1}$, respectively and $V(H^{(1)})\cap V(H^{(2)})=\{a_0,a_j\}$. By the first case, each of these is iterated tripartite where the ends have distinct colors. Form a proper iterated coloring of $\rm{CH}^{--}$ by coloring each of $H^{(1)},H^{(2)}$ with $\chi_1,\chi_2$ ensuring that $\chi_1(a_0)=\chi_2(a_0)$ and $\chi_1(a_j)=\chi_2(a_j)$.
\end{proof}

\subsection{The two-sided bicone}\label{sec:twobicone}

In this section we prove \cref{thm:SnvsKn}. Let $X=\{a,c\}$ and $Y=\{b,d,e\}$ and let $H'$ be a $3$-graph with $V(H')=X\cup Y$ and $E(H')=\{abc,cde\}$. Then let
\[
H=\rm{Bicone}(H';X,Y),
\]
as given by \cref{defn:bicone}. Then $H$ has $8$ vertices $\{x,y,z,a,b,c,d,e\}$ and $8$ edges. We prove the following:
\begin{lemma}\label{lem:twobicone} Let $H$ be defined as above.
\begin{enumerate}[(I)]
\item $H$ is doubleable.
\item $H$ is not forward laminar.
\end{enumerate}
\end{lemma}
Then \cref{thm:SnvsKn} is immediate, as \cref{thm:doublingvsSnn} and (I) imply the upper bound and \cref{thm:forwardlaminar} and (II) imply the lower bound.

\begin{proof} For (I), note that the iterated bicone $H_2=H-d-e$ (\cref{defn:iteratedbicone}) is forward-doubleable. By \cref{prop:doubleableiteratedblow-ups}, we may duplicate vertex $b$ into $b,d$ and delete the edge $adc$ to obtain $H-e$. Finally, double the pair $(d,c)$ into the edge $cde$, which forms a copy of $H$.

For (II), suppose $H$ is forward laminar with a family $\cS$ and some ordered binary containment tree $T_\cS$. Note that $H$ has three tight components. The largest has parts $\{z\},X\cup\{x\},$ and $Y\cup \{y\}$. Since these parts partition $V(H)$, these must appear exactly in $\cS$. Let $S\in \cS$ be the lowest common ancestor of $X'=X\cup \{x\}$ and $Y'=Y\cup \{y\}$ in $\cS$. In the order of $T_\cS$, assume that $Y'<X'$. For the tight component $\{abc\}$, there is a forward triple of disjoint sets $S_a,S_b,S_c\in \cS$ in some order such that $i\in S_i$. We must have $S_a,S_c\subseteq X'$, and thus $X'\not\subseteq S_b$. Thus $S_a,S_c$ are in the right subtree of $S$ and $S_b$ is in the left subtree of $S$. This contradicts that $S_a,S_b,S_c$ is forward.

If instead $X'<Y'$ in the order of $T_\cS$, the same argument applies to the component $\{cde\}$.
\end{proof}

\section{Concluding remarks}

Our results still leave many gaps in the landscape. In particular, we highlight the following.

\begin{problem}For the Fano plane $F$, is $r(F,n)$ polynomial or quasipolynomial?
\end{problem}

If the Fano plane minus one edge is quasipolynomial, then so are all NIT 3-graphs with six tight components. We would be interested in the following generalization of \cref{thm:fanoplane}.

\begin{problem}\label{prob:linearquasi} Is it true that if $H$ is a linear $3$-graph, then $r(H,n)$ is at most quasipolynomial?
\end{problem}

If true, the quasipolynomials can be arbitrarily large, as shown in \cite{linearlowerbound}. Observe that every linear $3$-graph is forward laminar, using the trivial construction of an increasing binary tree. That is, if $V(H)=[k]$ in order then define the family $\cF=\{\{1\},\{2\},\{3\},\ldots,\{k\}\}\cup \{[2],[3],\ldots,[k]\}$. Then every triple is forward, and every component is a single edge so is covered by some triple of singletons. Thus, our lower bound tools from \cref{sec:explower} are not able to disprove \cref{prob:linearquasi}.

A more general problem, analogous to the polynomial-to-exponential threshold proposed in \cite{erdoshajnal} and proven in \cite{polyexp}, is the following.
\begin{problem}
Which of the following two alternatives holds?
\begin{itemize}
\item There is a quasipolynomial-to-exponential transition for off diagonal $3$-graph Ramsey numbers.
\item There is a $3$-graph $H$ which \emph{is} forward laminar, \emph{is not} forward doubleable, and such that
\[
\log^{\omega(1)}n\leq \log r(H,n)\leq n^{o(1)}.
\]
\end{itemize}
\end{problem}

The gap between the forward doubleable class and forward laminar class is significant, but our remark after \cref{defn:fwddoubling} suggests some room for improvement. Both classes capture the asymmetry of the apparent threshold; \cref{thm:SnvsKn} (\cref{lem:twobicone}) shows that doubleable is not enough. 

The following proposition simplifies a generalization of \cref{conj:iteratedtripartite} which is naturally suggested by \cref{thm:fivetightcomponents}.

\begin{proposition}\label{conj:girth} Let $\cH_t$ be the set of NIT 3-graphs with at most $t$ tight components. The following conjectures are equivalent.
\begin{itemize}
\item For all finite $t$, $r(\cH_t,n)=n^{\omega(1)}$.
\item For any finite family $\cF$ of NIT $3$-graphs, $r(\cF,n)=n^{\omega(1)}$.
\end{itemize}
\end{proposition}
Note that $\cH_t$ (even just $\cH_1$) is an infinite family of $3$-graphs of unbounded size. The latter form seems plausible.
\begin{proof}
The first item trivially implies the second, since every finite family is contained in some $\cH_t$. For the opposite direction, recall that $r(\cH_t,n)=r(\cH_t^\minn,n)$. If $\cF$ is the set of nontrivial proper quotients of $\min \cH_t$, then $\cF$ is a family of $3$-graphs with at most $t$ edges and no isolated vertices. In particular, $\cF$ is finite. By \cref{lem:quotientpreservesNIT} each $F\in \cF$ is NIT. By \cref{lem:quotient} if $r(\cF,n)=n^{\omega(1)}$, then $r(\cH_t^\minn,n)=n^{\omega(1)}$, so the second item implies the first. 
\end{proof}

For higher uniformities, there is an analogous statement to \cref{conj:iteratedtripartite}, noted in \cite{CFGHMSVY}. Say a $k$-graph is iterated $k$-partite if it lies in the iterated blow-up of a $k$-edge. If $H$ is iterated $k$-partite then it is simple to show $r(H,K_n^{(k)})=n^{O(1)}$, and the converse of this is open. We present evidence for this generalization by showing it is true for graphs with fewer than $f(k)$ many 2-tight components, where $f(k)$ is the minimum number of edges in a linear $k$-graph which is not iterated $k$-partite. The following result relates to Theorems~\ref{thm:4tightcomponents} and \ref{thm:fivetightcomponents}.
\begin{theorem}\label{thm:f(k)components} For $1\leq t<f(k)$, let $\cH_t^{(k)}$ be the family of non-iterated $k$-partite $k$-graphs with at most $t$ many $2$-tight components. We have
\begin{align}
r(\cH_1^{(k)},n)&\geq \exp(\Omega(n^{2/k}))\tag{\cite[Thm.\ 4.1]{CFGHMSVY}}\\
r(\cH_t^{(k)},n)&\geq \exp(\Omega(\log^{1+\frac 1{t-1}}n))\quad\text{ for all $2\leq t<f(k)$.}\notag
\end{align}
\end{theorem}
When $t=2$ this cleanly generalizes \cref{thm:2tightcomponents}. To prove this, we define proper $2$-tight quotients as in \cref{defn:quotient}, and then construct the sparsified blow-up as in \cref{sec:proofofquotientlemma}. The random map $\phi$ is replaced with $\phi:\binom V2\to [r]^{k-2}$, where we put $e\in E(\cG_d(R))$ if $e$ is $\ell$-level and $e\in E(R^d)$ and for all pairs $u<v\in e$, the map $\phi(uv)=(w_\ell)_{w\in e-u-v}$. Using this, a single pair $uv\subseteq e$ determines the embedding of the whole edge into the blow-up. Then the analogous quotient lemma to \ref{lem:quotient} applies with a weaker constant and then \cref{cor:notquotients} can be used to conclude the $t\geq 2$ cases of \cref{thm:f(k)components}.

The Erd\H{o}s-Faber-Lov\'asz conjecture \cite{EFLconjecture}, proven for sufficiently large $k$ \cite{EFLproof}, would imply that $f(k)\geq k+1$. In the other direction, we note that $f(k)\leq k+1$ for odd $k\geq 7$. To see this, let $H'$ be the (not uniform) hypergraph obtained from $K_{k+1}^{(2)}$ by replacing one triangle $\{\{xy\},\{yz\},\{xz\}\}$ with the edges $\{\{xyz\},\{x\},\{y\},\{z\}\}$. Then letting $H$ be the hypergraph dual of $H'$, we see that $H$ has $k+1$ edges, is $k$-uniform, is linear, and is not iterated $k$-partite. 

The analogous quotient lemma to \ref{lem:quotient} for $k$-graphs also gives a reduction analogous to \cref{thm:reductiontolinear}. We are not able to obtain a result analogous to \cref{lem:quotientonegraph}, because \cref{lem:maxtightquotient} does not generalize to $k\geq 4$, since for the tight component $\{1234,1256\}$, there are multiple ways to collapse this into a single edge: we could identify $3$ with $5$ and $4$ with $6$, or vice versa. Finally, the proof of \cref{thm:singlelinear} generalizes to higher uniformities by substituting a $k$-uniform `edge simulating' gadget like $F^{\threevee}$ which satisfies the $k$-partite analogue of \cref{claim:splitfano}. To construct such a gadget, choose any $k$-partite linear $k$-graph with a unique non-constant top-level coloring up to permutation, having parts $V_1,\ldots, V_k$ of size at least $k-1$. Then add edges $e_1,\ldots, e_k$ containing leaves $x_1,\ldots, x_k$ such that $e_i\subseteq \{x_i\}\cup V_i$. In the resulting gadget, the only non-constant top-level coloring is rainbow on $\{x_1,\ldots, x_k\}$.

The methods from Sections \ref{sec:upperbounds} and \ref{sec:explower} can also be generalized to higher uniformities with caveats. For example, it is straightforward to construct quasi-dense sets in any uniformity, but the notion of forward-quasi-denseness is not easily generalizable or achievable for $k\geq 4$.

\subsection{AI Disclosure}
During the research process, ChatGPT 5.5 claimed a proof of \cref{thm:fanovsBn} using a counting argument, which motivated further study. Our proof of \cref{thm:fanovsBn} is deduced from \cref{thm:fwddoublingvsBn} which is structurally distinct but uses similar methods. The free version of ChatGPT was used to generate Figure \ref{fig:chain}.

During preparation, ChatGPT 6 Astra suggested a proof of \cref{thm:singlelinear}. We added this result to the paper, using a slightly different proof with a similar gadget-replacement idea. AI was used to typo-check the final manuscript.

\input{Compilation_Files/biblio}

\end{document}

%% file: Compilation_Files/preamble.tex
\usepackage{amsmath,amsthm,amsfonts,graphicx,float}
\usepackage{amssymb,upgreek,stmaryrd,framed,bbm}
\usepackage{wrapfig,mathrsfs,xfrac,mathtools,epsdice}
\usepackage{qtree,setspace,accents,verbatim}
\usepackage[x11names]{xcolor}
\usepackage[ruled,vlined]{algorithm2e}
\usepackage[normalem]{ulem}
\usepackage[english]{babel}
\usepackage{framed}
\usepackage[shortlabels]{enumitem}

\usepackage{thmtools,thm-restate}

\usepackage{tikz}
\usetikzlibrary{shapes}

\makeatletter\let\@@citation@@=\citation\renewcommand{\citation}[1]{\@@citation@@{#1}\@for\@tempa:=#1\do{\@ifundefined{cit@\@tempa}{\global\@namedef{cit@\@tempa}{}}{}}}\makeatother
\usepackage[linktocpage=true]{hyperref}
\usepackage[capitalize]{cleveref}
\makeatletter\def\@lbibitem[#1]#2#3\par{\@ifundefined{cit@#2}{}{\@skiphyperreftrue\H@item[\ifx\Hy@raisedlink\@empty\hyper@anchorstart{cite.#2\@extra@b@citeb}\@BIBLABEL{#1}\hyper@anchorend\else\Hy@raisedlink{\hyper@anchorstart{cite.#2\@extra@b@citeb}\hyper@anchorend}\@BIBLABEL{#1}\fi\hfill]\@skiphyperreffalse}\if@filesw\begingroup\let\protect\noexpand\immediate\write\@auxout{\string\bibcite{#2}{#1}}\endgroup\fi\ignorespaces\@ifundefined{cit@#2}{}{#3}} \def\@bibitem#1#2\par{\@ifundefined{cit@#1}{}{\@skiphyperreftrue\H@item\@skiphyperreffalse\Hy@raisedlink{\hyper@anchorstart{cite.#1\@extra@b@citeb}\relax\hyper@anchorend}}\if@filesw\begingroup\let\protect\noexpand\immediate\write\@auxout{\string\bibcite{#1}{\the\value{\@listctr}}}\endgroup\fi\ignorespaces\@ifundefined{cit@#1}{}{#2}}\makeatother
\newcommand{\Ex}{\mathbb{E}}

\newcommand{\N}{\mathbb{N}}

\newcommand{\PP}{\mathbb{P}}

\newcommand{\R}{\mathbb{R}}

\newcommand{\Z}{\mathbb{Z}}

\newcommand{\cA}{\mathcal{A}}

\newcommand{\cF}{\mathcal{F}}
\newcommand{\cG}{\mathcal{G}}\newcommand{\cH}{\mathcal{H}}
\newcommand{\cI}{\mathcal{I}}
\newcommand{\cL}{\mathcal{L}}

\newcommand{\cP}{\mathcal{P}}
\newcommand{\cQ}{\mathcal{Q}}
\newcommand{\cS}{\mathcal{S}}

\newcommand{\nth}{^{\text{th}}} 
\newcommand{\eps}{\upvarepsilon}
\renewcommand{\rm}[1]{\mathrm{#1}}

\newcommand{\ceil}[1]{\left\lceil #1 \right\rceil}

\newcommand{\floor}[1]{\left\lfloor #1 \right\rfloor}

\def\multichoose#1#2{\ensuremath{\left(\kern-.3em\left(\genfrac{}{}{0pt}{}{#1}{#2}\right)\kern-.3em\right)}}

\newtheorem{theorem}{Theorem}[section]
\newtheorem{lemma}[theorem]{Lemma}
\newtheorem{problem}[theorem]{Problem}
\newtheorem{proposition}[theorem]{Proposition}
\newtheorem{corollary}{Corollary}[theorem]
\newtheorem{claim}{Claim}[theorem]
\newtheorem{conjecture}[theorem]{Conjecture}

\newtheorem{innercustomtheorem}{Theorem}
\newenvironment{numtheorem}[1]
{\renewcommand\theinnercustomtheorem{#1}\innercustomtheorem}
  {\endinnercustomtheorem}

\theoremstyle{definition}
\newtheorem{definition}[theorem]{Definition}
\newtheorem*{definition*}{Definition}
\newtheorem*{construction}{Construction}
\newtheorem*{observation}{Observation}

\theoremstyle{remark}
\newtheorem*{remark}{Remark}
\newtheorem*{example}{Example}

\newenvironment{subproof}[1][\proofname]{
\begin{proof}[#1]}{\end{proof}}



\makeatother

\makeatother
\newcommand{\minn}{\rm{min}}

%% file: tikz_diagrams/temp.tex
\begin{tikzpicture}[
    scale=1.4,
    vertex/.style={circle,fill=black,inner sep=1.7pt},
    hedge/.style={
        draw=blue!70!black,
        fill=blue!50,
        fill opacity=.25,
        line width=.7pt
    }
]

%--------------------------------------------------
% Coordinates
%--------------------------------------------------
\def\R{2.4}
\def\r{1.35}

\foreach \i in {0,...,4}{
    \coordinate (A\i) at ({90-72*\i}:\R);
    \coordinate (B\i) at ({90-72*\i}:\r);
}

%--------------------------------------------------
% Hyperedges
%--------------------------------------------------
\foreach \i [evaluate=\i as \j using {mod(\i+1,5)}] in {0,...,4}{

    % a_i b_i a_{i+1}
    \filldraw[hedge]
        (A\i) -- (B\i) -- (A\j) -- cycle;

    % a_i b_i b_{i+1}
    \filldraw[hedge]
        (A\i) -- (B\i) -- (B\j) -- cycle;
}

%--------------------------------------------------
% Vertices
%--------------------------------------------------
\foreach \i in {0,...,4}{
    \node[vertex] (a\i) at (A\i) {};
    \node[vertex] (b\i) at (B\i) {};
}

%--------------------------------------------------
% Labels
%--------------------------------------------------
\node[above=2pt] at (A0) {$a_0$};
\node[below=2pt] at (B0) {$b_0$};

\node[right=2pt] at (A1) {$a_1$};
\node[right=2pt] at (B1) {$b_1$};

\node[left=2pt] at (A4) {$a_4$};
\node[left=2pt] at (B4) {$b_4$};
\end{tikzpicture}

%% file: tikz_diagrams/splitfano.tex
\tikzset{every picture/.style={line width=0.75pt}} %set default line width to 0.75pt        

\begin{tikzpicture}[x=0.75pt,y=0.75pt,yscale=-1,xscale=1]
%uncomment if require: \path (0,719); %set diagram left start at 0, and has height of 719

%Straight Lines [id:da5280552392997866] 
\draw [color={rgb, 255:red, 0; green, 0; blue, 0 }  ,draw opacity=1 ]   (160,200) -- (240,200.5) ;
%Straight Lines [id:da4622657961090464] 
\draw [color={rgb, 255:red, 0; green, 0; blue, 0 }  ,draw opacity=1 ]   (240,200.5) -- (320,201) ;
%Straight Lines [id:da6900225772939664] 
\draw    (230,60) -- (200,130) -- (160,200) ;
%Straight Lines [id:da9508294520901769] 
\draw    (251,60) -- (280,130) -- (320,201) ;
%Shape: Circle [id:dp6869995539239468] 
\draw  [color={rgb, 255:red, 0; green, 0; blue, 0 }  ,draw opacity=1 ] (193.75,154.25) .. controls (193.75,128.71) and (214.46,108) .. (240,108) .. controls (265.54,108) and (286.25,128.71) .. (286.25,154.25) .. controls (286.25,179.79) and (265.54,200.5) .. (240,200.5) .. controls (214.46,200.5) and (193.75,179.79) .. (193.75,154.25) -- cycle ;
%Straight Lines [id:da6429476964532436] 
\draw    (240,60.5) -- (240,200.5) ;
%Straight Lines [id:da8618504691881106] 
\draw    (280,130) -- (160,200) ;
%Straight Lines [id:da35455643942951987] 
\draw    (200,130) -- (320,201) ;
%Straight Lines [id:da4270749924471825] 
\draw [color={rgb, 255:red, 0; green, 0; blue, 0 }  ,draw opacity=1 ][line width=1.5]    (230,60) ;
\draw [shift={(230,60)}, rotate = 0] [color={rgb, 255:red, 0; green, 0; blue, 0 }  ,draw opacity=1 ][fill={rgb, 255:red, 0; green, 0; blue, 0 }  ,fill opacity=1 ][line width=1.5]      (0, 0) circle [x radius= 3.05, y radius= 3.05]   ;
\draw [shift={(230,60)}, rotate = 0] [color={rgb, 255:red, 0; green, 0; blue, 0 }  ,draw opacity=1 ][fill={rgb, 255:red, 0; green, 0; blue, 0 }  ,fill opacity=1 ][line width=1.5]      (0, 0) circle [x radius= 3.05, y radius= 3.05]   ;
%Straight Lines [id:da24721951105091378] 
\draw [color={rgb, 255:red, 0; green, 0; blue, 0 }  ,draw opacity=1 ][line width=1.5]    (240,200.5) ;
\draw [shift={(240,200.5)}, rotate = 0] [color={rgb, 255:red, 0; green, 0; blue, 0 }  ,draw opacity=1 ][fill={rgb, 255:red, 0; green, 0; blue, 0 }  ,fill opacity=1 ][line width=1.5]      (0, 0) circle [x radius= 3.05, y radius= 3.05]   ;
\draw [shift={(240,200.5)}, rotate = 0] [color={rgb, 255:red, 0; green, 0; blue, 0 }  ,draw opacity=1 ][fill={rgb, 255:red, 0; green, 0; blue, 0 }  ,fill opacity=1 ][line width=1.5]      (0, 0) circle [x radius= 3.05, y radius= 3.05]   ;
%Straight Lines [id:da7832739715812458] 
\draw [color={rgb, 255:red, 0; green, 0; blue, 0 }  ,draw opacity=1 ][line width=1.5]    (240,154.25) ;
\draw [shift={(240,154.25)}, rotate = 0] [color={rgb, 255:red, 0; green, 0; blue, 0 }  ,draw opacity=1 ][fill={rgb, 255:red, 0; green, 0; blue, 0 }  ,fill opacity=1 ][line width=1.5]      (0, 0) circle [x radius= 3.05, y radius= 3.05]   ;
\draw [shift={(240,154.25)}, rotate = 0] [color={rgb, 255:red, 0; green, 0; blue, 0 }  ,draw opacity=1 ][fill={rgb, 255:red, 0; green, 0; blue, 0 }  ,fill opacity=1 ][line width=1.5]      (0, 0) circle [x radius= 3.05, y radius= 3.05]   ;
%Straight Lines [id:da22365745375510349] 
\draw [color={rgb, 255:red, 0; green, 0; blue, 0 }  ,draw opacity=1 ][line width=1.5]    (160,200) ;
\draw [shift={(160,200)}, rotate = 0] [color={rgb, 255:red, 0; green, 0; blue, 0 }  ,draw opacity=1 ][fill={rgb, 255:red, 0; green, 0; blue, 0 }  ,fill opacity=1 ][line width=1.5]      (0, 0) circle [x radius= 3.05, y radius= 3.05]   ;
\draw [shift={(160,200)}, rotate = 0] [color={rgb, 255:red, 0; green, 0; blue, 0 }  ,draw opacity=1 ][fill={rgb, 255:red, 0; green, 0; blue, 0 }  ,fill opacity=1 ][line width=1.5]      (0, 0) circle [x radius= 3.05, y radius= 3.05]   ;
%Straight Lines [id:da5605042107102415] 
\draw [color={rgb, 255:red, 0; green, 0; blue, 0 }  ,draw opacity=1 ][line width=1.5]    (200,130) ;
\draw [shift={(200,130)}, rotate = 0] [color={rgb, 255:red, 0; green, 0; blue, 0 }  ,draw opacity=1 ][fill={rgb, 255:red, 0; green, 0; blue, 0 }  ,fill opacity=1 ][line width=1.5]      (0, 0) circle [x radius= 3.05, y radius= 3.05]   ;
\draw [shift={(200,130)}, rotate = 0] [color={rgb, 255:red, 0; green, 0; blue, 0 }  ,draw opacity=1 ][fill={rgb, 255:red, 0; green, 0; blue, 0 }  ,fill opacity=1 ][line width=1.5]      (0, 0) circle [x radius= 3.05, y radius= 3.05]   ;
%Straight Lines [id:da36567644450168135] 
\draw [color={rgb, 255:red, 0; green, 0; blue, 0 }  ,draw opacity=1 ][line width=1.5]    (251,60) ;
\draw [shift={(251,60)}, rotate = 0] [color={rgb, 255:red, 0; green, 0; blue, 0 }  ,draw opacity=1 ][fill={rgb, 255:red, 0; green, 0; blue, 0 }  ,fill opacity=1 ][line width=1.5]      (0, 0) circle [x radius= 3.05, y radius= 3.05]   ;
\draw [shift={(251,60)}, rotate = 0] [color={rgb, 255:red, 0; green, 0; blue, 0 }  ,draw opacity=1 ][fill={rgb, 255:red, 0; green, 0; blue, 0 }  ,fill opacity=1 ][line width=1.5]      (0, 0) circle [x radius= 3.05, y radius= 3.05]   ;
%Straight Lines [id:da9476636638462362] 
\draw [color={rgb, 255:red, 0; green, 0; blue, 0 }  ,draw opacity=1 ][line width=1.5]    (240,60.5) ;
\draw [shift={(240,60.5)}, rotate = 0] [color={rgb, 255:red, 0; green, 0; blue, 0 }  ,draw opacity=1 ][fill={rgb, 255:red, 0; green, 0; blue, 0 }  ,fill opacity=1 ][line width=1.5]      (0, 0) circle [x radius= 3.05, y radius= 3.05]   ;
\draw [shift={(240,60.5)}, rotate = 0] [color={rgb, 255:red, 0; green, 0; blue, 0 }  ,draw opacity=1 ][fill={rgb, 255:red, 0; green, 0; blue, 0 }  ,fill opacity=1 ][line width=1.5]      (0, 0) circle [x radius= 3.05, y radius= 3.05]   ;
%Straight Lines [id:da6664649808848905] 
\draw [color={rgb, 255:red, 0; green, 0; blue, 0 }  ,draw opacity=1 ][line width=1.5]    (320,201) ;
\draw [shift={(320,201)}, rotate = 0] [color={rgb, 255:red, 0; green, 0; blue, 0 }  ,draw opacity=1 ][fill={rgb, 255:red, 0; green, 0; blue, 0 }  ,fill opacity=1 ][line width=1.5]      (0, 0) circle [x radius= 3.05, y radius= 3.05]   ;
\draw [shift={(320,201)}, rotate = 0] [color={rgb, 255:red, 0; green, 0; blue, 0 }  ,draw opacity=1 ][fill={rgb, 255:red, 0; green, 0; blue, 0 }  ,fill opacity=1 ][line width=1.5]      (0, 0) circle [x radius= 3.05, y radius= 3.05]   ;
%Straight Lines [id:da5811332207568883] 
\draw [color={rgb, 255:red, 0; green, 0; blue, 0 }  ,draw opacity=1 ][line width=1.5]    (280,130) ;
\draw [shift={(280,130)}, rotate = 0] [color={rgb, 255:red, 0; green, 0; blue, 0 }  ,draw opacity=1 ][fill={rgb, 255:red, 0; green, 0; blue, 0 }  ,fill opacity=1 ][line width=1.5]      (0, 0) circle [x radius= 3.05, y radius= 3.05]   ;
\draw [shift={(280,130)}, rotate = 0] [color={rgb, 255:red, 0; green, 0; blue, 0 }  ,draw opacity=1 ][fill={rgb, 255:red, 0; green, 0; blue, 0 }  ,fill opacity=1 ][line width=1.5]      (0, 0) circle [x radius= 3.05, y radius= 3.05]   ;

% Text Node
\draw (219,44) node [anchor=north west][inner sep=0.75pt]  [font=\normalsize] [align=left] {$\displaystyle x$};
% Text Node
\draw (235,44) node [anchor=north west][inner sep=0.75pt]  [font=\normalsize] [align=left] {$\displaystyle y$};
% Text Node
\draw (251,44) node [anchor=north west][inner sep=0.75pt]  [font=\normalsize] [align=left] {$\displaystyle z$};
% Text Node
\draw (184,117) node [anchor=north west][inner sep=0.75pt]  [font=\normalsize] [align=left] {$\displaystyle 2$};
% Text Node
\draw (248,147) node [anchor=north west][inner sep=0.75pt]  [font=\normalsize] [align=left] {$\displaystyle 4$};
% Text Node
\draw (284,119) node [anchor=north west][inner sep=0.75pt]  [font=\normalsize] [align=left] {$\displaystyle 6$};
% Text Node
\draw (324,191) node [anchor=north west][inner sep=0.75pt]  [font=\normalsize] [align=left] {$\displaystyle 7$};
% Text Node
\draw (243,202) node [anchor=north west][inner sep=0.75pt]  [font=\normalsize] [align=left] {$\displaystyle 5$};
% Text Node
\draw (145,192) node [anchor=north west][inner sep=0.75pt]  [font=\normalsize] [align=left] {$\displaystyle 3$};

\end{tikzpicture}

%% file: tikz_diagrams/fanodoubling2.tex
\tikzset{every picture/.style={line width=0.975pt}} %set default line width to 0.975pt        

\definecolor{myblue}{RGB}{70,120,220}
\definecolor{myred}{RGB}{240,0,0}
\definecolor{myblack}{RGB}{0,0,0}

\begin{tikzpicture}[x=0.975pt,y=0.975pt,yscale=-1,xscale=1]
%uncomment if require: \path (0,433);
%set diagram left start at 0, and has height of 433

\newlength{\ssr}
\setlength{\ssr}{1pt}

%Straight Lines [id:da8805437764760149] 
\draw    (156,206.5) ;
%Straight Lines [id:da7227463204624529] 
\draw    (171,271.5) ;
%Straight Lines [id:da5607116737304921] 
\draw [color=myred  ,draw opacity=1 ]   (130,240) -- (130,280) ;
%Straight Lines [id:da2763388876719297] 
\draw [color=myblue  ,draw opacity=1 ]   (130,240) -- (170,240) ;
%Straight Lines [id:da33898041957309966] 
\draw [color=myblue  ,draw opacity=1 ]   (130,280) -- (170,280) ;
%Straight Lines [id:da7323932579304792] 
\draw [color=myred  ,draw opacity=1 ]   (170,240) -- (170,280) ;
%Straight Lines [id:da5341776501246013] 
\draw [color=myblue  ,draw opacity=1 ]   (130,240) -- (170,280) ;
%Straight Lines [id:da6158918639598254] 
\draw [color=myblue  ,draw opacity=1 ]   (170,240) -- (130,280) ;
%Straight Lines [id:da04514813256314132] 
\draw [color=myblue  ,draw opacity=1 ]   (129,199) -- (228,199) ;
%Curve Lines [id:da3841083334506835] 
\draw [line width=0.75pt][color=myblue  ,draw opacity=1 ] [dash pattern={on 2.925pt off 0.975pt}]   (170,240) .. controls (190,240) and (230,260) .. (230,260) ;
%Curve Lines [id:da7952380513154768] 
\draw [line width=0.75pt][color=myblue  ,draw opacity=1 ] [dash pattern={on 2.925pt off 0.975pt}]  (170,240) .. controls (190,220) and (230,260) .. (230,260) ;
%Curve Lines [id:da11976676716069412] 
\draw [line width=0.75pt][color=myblue  ,draw opacity=1 ] [dash pattern={on 2.925pt off 0.975pt}]  (170,280) .. controls (190,280) and (230,260) .. (230,260) ;
%Curve Lines [id:da9545070629832995] 
\draw [line width=0.75pt][color=myblue  ,draw opacity=1 ] [dash pattern={on 2.925pt off 0.975pt}]  (170,280) .. controls (190,300) and (230,260) .. (230,260) ;
%Curve Lines [id:da6442677241177867] 
\draw [line width=0.75pt][color=myred  ,draw opacity=1 ] [dash pattern={on 2.925pt off 0.975pt}]  (170,280) .. controls (170,310) and (290,260) .. (290,260) ;
%Curve Lines [id:da14989063563056648] 
\draw [line width=0.75pt][color=myred  ,draw opacity=1 ] [dash pattern={on 2.925pt off 0.975pt}]  (130,280) .. controls (130,320) and (220,291) .. (290,260) ;
%Straight Lines [id:da8870980175955763] 
\draw    (171,271.5) ;
%Straight Lines [id:da004138432147504423] 
\draw [line width=1.95]    (130,240) ;
\draw [shift={(130,240)}, rotate = 0] [color=myblack  ][fill=myblack  ][line width=1.95]      (0, 0) circle [x radius= 3.05, y radius= 3.05]   ;
\draw [shift={(130,240)}, rotate = 0] [color=myblack  ][fill=myblack  ][line width=1.95]      (0, 0) circle [x radius= 3.05, y radius= 3.05]   ;
%Straight Lines [id:da5971827705673634] 
\draw [line width=1.95]    (130,280) ;
\draw [shift={(130,280)}, rotate = 0] [color=myblack  ][fill=myblack  ][line width=1.95]      (0, 0) circle [x radius= 3.05, y radius= 3.05]   ;
\draw [shift={(130,280)}, rotate = 0] [color=myblack  ][fill=myblack  ][line width=1.95]      (0, 0) circle [x radius= 3.05, y radius= 3.05]   ;
%Straight Lines [id:da9157802466893139] 
\draw [line width=1.95]    (170,280) ;
\draw [shift={(170,280)}, rotate = 0] [color=myblack  ][fill=myblack  ][line width=1.95]      (0, 0) circle [x radius= 3.05, y radius= 3.05]   ;
\draw [shift={(170,280)}, rotate = 0] [color=myblack  ][fill=myblack  ][line width=1.95]      (0, 0) circle [x radius= 3.05, y radius= 3.05]   ;
%Straight Lines [id:da48396571891592033] 
\draw [color=myblack  ,draw opacity=1 ][line width=1.95]    (170,240) ;
\draw [shift={(170,240)}, rotate = 0] [color=myblack  ,draw opacity=1 ][fill=myblack  ,fill opacity=1 ][line width=1.95]      (0, 0) circle [x radius= 3.05, y radius= 3.05]   ;
\draw [shift={(170,240)}, rotate = 0] [color=myblack  ,draw opacity=1 ][fill=myblack  ,fill opacity=1 ][line width=1.95]      (0, 0) circle [x radius= 3.05, y radius= 3.05]   ;
%Straight Lines [id:da0654191351550234] 
\draw [line width=1.95]    (230,260) ;
\draw [shift={(230,260)}, rotate = 0] [color=myblack  ][fill=myblack  ][line width=1.95]      (0, 0) circle [x radius= 3.05, y radius= 3.05]   ;
\draw [shift={(230,260)}, rotate = 0] [color=myblack  ][fill=myblack  ][line width=1.95]      (0, 0) circle [x radius= 3.05, y radius= 3.05]   ;
%Straight Lines [id:da2472475317365055] 
\draw [line width=1.95]    (290,260) ;
\draw [shift={(290,260)}, rotate = 0] [color=myblack  ][fill=myblack  ][line width=1.95]      (0, 0) circle [x radius= 3.05, y radius= 3.05]   ;
\draw [shift={(290,260)}, rotate = 0] [color=myblack  ][fill=myblack  ][line width=1.95]      (0, 0) circle [x radius= 3.05, y radius= 3.05]   ;
%Straight Lines [id:da7014079693041428] 
\draw [color=myblue  ,draw opacity=1 ][line width=1.95]    (230,260) ;
\draw [shift={(230,260)}, rotate = 0] [color=myblue  ,draw opacity=1 ][fill=myblue  ,fill opacity=1 ][line width=1.95]      (0, 0) circle [x radius= \ssr, y radius= \ssr]   ;
\draw [shift={(230,260)}, rotate = 0] [color=myblue  ,draw opacity=1 ][fill=myblue  ,fill opacity=1 ][line width=1.95]      (0, 0) circle [x radius= \ssr, y radius= \ssr]   ;
%Straight Lines [id:da5771103918387307] 
\draw [color=myred  ,draw opacity=1 ][line width=1.95]    (290,260) ;
\draw [shift={(290,260)}, rotate = 0] [color=myred  ,draw opacity=1 ][fill=myred  ,fill opacity=1 ][line width=1.95]      (0, 0) circle [x radius= \ssr, y radius= \ssr]   ;
\draw [shift={(290,260)}, rotate = 0] [color=myred  ,draw opacity=1 ][fill=myred  ,fill opacity=1 ][line width=1.95]      (0, 0) circle [x radius= \ssr, y radius= \ssr]   ;
%Straight Lines [id:da10553764305319502] 
\draw [line width=1.95]    (129,199) ;
\draw [shift={(129,199)}, rotate = 0] [color=myblack  ][fill=myblack  ][line width=1.95]      (0, 0) circle [x radius= 3.05, y radius= 3.05]   ;
\draw [shift={(129,199)}, rotate = 0] [color=myblack  ][fill=myblack  ][line width=1.95]      (0, 0) circle [x radius= 3.05, y radius= 3.05]   ;
%Straight Lines [id:da23768904047700368] 
\draw [line width=1.95]    (169,199) ;
\draw [shift={(169,199)}, rotate = 0] [color=myblack  ][fill=myblack  ][line width=1.95]      (0, 0) circle [x radius= 3.05, y radius= 3.05]   ;
\draw [shift={(169,199)}, rotate = 0] [color=myblack  ][fill=myblack  ][line width=1.95]      (0, 0) circle [x radius= 3.05, y radius= 3.05]   ;
%Straight Lines [id:da8592378396009179] 
\draw [line width=1.95]    (228,199) ;
\draw [shift={(228,199)}, rotate = 0] [color=myblack  ][fill=myblack  ][line width=1.95]      (0, 0) circle [x radius= 3.05, y radius= 3.05]   ;
\draw [shift={(228,199)}, rotate = 0] [color=myblack  ][fill=myblack  ][line width=1.95]      (0, 0) circle [x radius= 3.05, y radius= 3.05]   ;
%Straight Lines [id:da4612889159696697] 
\draw [line width=1.95]    (288,199) ;
\draw [shift={(288,199)}, rotate = 0] [color=myblack  ][fill=myblack  ][line width=1.95]      (0, 0) circle [x radius= 3.05, y radius= 3.05]   ;
\draw [shift={(288,199)}, rotate = 0] [color=myblack  ][fill=myblack  ][line width=1.95]      (0, 0) circle [x radius= 3.05, y radius= 3.05]   ;
%Straight Lines [id:da7430094700476751] 
\draw [color=myblue  ,draw opacity=1 ][line width=1.95]    (228,199) ;
\draw [shift={(228,199)}, rotate = 0] [color=myblue  ,draw opacity=1 ][fill=myblue  ,fill opacity=1 ][line width=1.95]      (0, 0) circle [x radius= \ssr, y radius= \ssr]   ;
\draw [shift={(228,199)}, rotate = 0] [color=myblue  ,draw opacity=1 ][fill=myblue  ,fill opacity=1 ][line width=1.95]      (0, 0) circle [x radius= \ssr, y radius= \ssr]   ;
%Straight Lines [id:da9931021686526179] 
\draw [color=myred  ,draw opacity=1 ][line width=1.95]    (288,199) ;
\draw [shift={(288,199)}, rotate = 0] [color=myred  ,draw opacity=1 ][fill=myred  ,fill opacity=1 ][line width=1.95]      (0, 0) circle [x radius= \ssr, y radius= \ssr]   ;
\draw [shift={(288,199)}, rotate = 0] [color=myred  ,draw opacity=1 ][fill=myred  ,fill opacity=1 ][line width=1.95]      (0, 0) circle [x radius= \ssr, y radius= \ssr]   ;
%Straight Lines [id:da9131297789792832] 
\draw    (171,361.5) ;
%Straight Lines [id:da9478868993754114] 
\draw [color=myred  ,draw opacity=1 ]   (130,330) -- (130,370) ;
%Straight Lines [id:da30483500655576545] 
\draw [color=myblue  ,draw opacity=1 ]   (130,330) -- (170,330) ;
%Straight Lines [id:da728287865949761] 
\draw [color=myblue  ,draw opacity=1 ]   (130,370) -- (170,370) ;
%Straight Lines [id:da5488701596852875] 
\draw [color=myred  ,draw opacity=1 ]   (170,330) -- (170,370) ;
%Straight Lines [id:da2457811281156067] 
\draw [color=myblue  ,draw opacity=1 ]   (130,330) -- (170,370) ;
%Straight Lines [id:da09456208270326061] 
\draw [color=myblue  ,draw opacity=1 ]   (170,330) -- (130,370) ;
%Straight Lines [id:da7212457824402909] 
\draw    (171,361.5) ;
%Straight Lines [id:da8712895161309218] 
\draw [color=myred  ,draw opacity=1 ]   (230,330) -- (230,370) ;
%Straight Lines [id:da023778228584427796] 
\draw [line width=1.95]    (130,330) ;
\draw [shift={(130,330)}, rotate = 0] [color=myblack  ][fill=myblack  ][line width=1.95]      (0, 0) circle [x radius= 3.05, y radius= 3.05]   ;
\draw [shift={(130,330)}, rotate = 0] [color=myblack  ][fill=myblack  ][line width=1.95]      (0, 0) circle [x radius= 3.05, y radius= 3.05]   ;
%Straight Lines [id:da3293571958913827] 
\draw [line width=1.95]    (130,370) ;
\draw [shift={(130,370)}, rotate = 0] [color=myblack  ][fill=myblack  ][line width=1.95]      (0, 0) circle [x radius= 3.05, y radius= 3.05]   ;
\draw [shift={(130,370)}, rotate = 0] [color=myblack  ][fill=myblack  ][line width=1.95]      (0, 0) circle [x radius= 3.05, y radius= 3.05]   ;
%Straight Lines [id:da47379366112699484] 
\draw [line width=1.95]    (170,370) ;
\draw [shift={(170,370)}, rotate = 0] [color=myblack  ][fill=myblack  ][line width=1.95]      (0, 0) circle [x radius= 3.05, y radius= 3.05]   ;
\draw [shift={(170,370)}, rotate = 0] [color=myblack  ][fill=myblack  ][line width=1.95]      (0, 0) circle [x radius= 3.05, y radius= 3.05]   ;
%Straight Lines [id:da7988064846187624] 
\draw [color=myblack  ,draw opacity=1 ][line width=1.95]    (170,330) ;
\draw [shift={(170,330)}, rotate = 0] [color=myblack  ,draw opacity=1 ][fill=myblack  ,fill opacity=1 ][line width=1.95]      (0, 0) circle [x radius= 3.05, y radius= 3.05]   ;
\draw [shift={(170,330)}, rotate = 0] [color=myblack  ,draw opacity=1 ][fill=myblack  ,fill opacity=1 ][line width=1.95]      (0, 0) circle [x radius= 3.05, y radius= 3.05]   ;
%Straight Lines [id:da15392818457208723] 
\draw [line width=1.95]    (230,330) ;
\draw [shift={(230,330)}, rotate = 0] [color=myblack  ][fill=myblack  ][line width=1.95]      (0, 0) circle [x radius= 3.05, y radius= 3.05]   ;
\draw [shift={(230,330)}, rotate = 0] [color=myblack  ][fill=myblack  ][line width=1.95]      (0, 0) circle [x radius= 3.05, y radius= 3.05]   ;
%Straight Lines [id:da8221148512521259] 
\draw [line width=1.95]    (290,350) ;
\draw [shift={(290,350)}, rotate = 0] [color=myblack  ][fill=myblack  ][line width=1.95]      (0, 0) circle [x radius= 3.05, y radius= 3.05]   ;
\draw [shift={(290,350)}, rotate = 0] [color=myblack  ][fill=myblack  ][line width=1.95]      (0, 0) circle [x radius= 3.05, y radius= 3.05]   ;
%Straight Lines [id:da834785734187802] 
\draw [color=myblue  ,draw opacity=1 ][line width=1.95]    (230,330) ;
\draw [shift={(230,330)}, rotate = 0] [color=myblue  ,draw opacity=1 ][fill=myblue  ,fill opacity=1 ][line width=1.95]      (0, 0) circle [x radius= \ssr, y radius= \ssr]   ;
\draw [shift={(230,330)}, rotate = 0] [color=myblue  ,draw opacity=1 ][fill=myblue  ,fill opacity=1 ][line width=1.95]      (0, 0) circle [x radius= \ssr, y radius= \ssr]   ;
%Straight Lines [id:da35897296071975526] 
\draw [color=myred  ,draw opacity=1 ][line width=1.95]    (290,350) ;
\draw [shift={(290,350)}, rotate = 0] [color=myred  ,draw opacity=1 ][fill=myred  ,fill opacity=1 ][line width=1.95]      (0, 0) circle [x radius= \ssr, y radius= \ssr]   ;
\draw [shift={(290,350)}, rotate = 0] [color=myred  ,draw opacity=1 ][fill=myred  ,fill opacity=1 ][line width=1.95]      (0, 0) circle [x radius= \ssr, y radius= \ssr]   ;
%Curve Lines [id:da22074391430807283] 
\draw [line width=0.75pt][color=myred  ,draw opacity=1 ] [dash pattern={on 2.925pt off 0.975pt}]  (230,370) .. controls (230,400) and (290,349.5) .. (290,350) ;
%Straight Lines [id:da38567690161369017] 
\draw [line width=1.95]    (230,370) ;
\draw [shift={(230,370)}, rotate = 0] [color=myblack  ][fill=myblack  ][line width=1.95]      (0, 0) circle [x radius= 3.05, y radius= 3.05]   ;
\draw [shift={(230,370)}, rotate = 0] [color=myblack  ][fill=myblack  ][line width=1.95]      (0, 0) circle [x radius= 3.05, y radius= 3.05]   ;
%Straight Lines [id:da6244593558117005] 
\draw [color=myblue  ,draw opacity=1 ][line width=1.95]    (230,370) ;
\draw [shift={(230,370)}, rotate = 0] [color=myblue  ,draw opacity=1 ][fill=myblue  ,fill opacity=1 ][line width=1.95]      (0, 0) circle [x radius= \ssr, y radius= \ssr]   ;
\draw [shift={(230,370)}, rotate = 0] [color=myblue  ,draw opacity=1 ][fill=myblue  ,fill opacity=1 ][line width=1.95]      (0, 0) circle [x radius= \ssr, y radius= \ssr]   ;
%Straight Lines [id:da24369890842331887] 
\draw[line width=0.5pt]    (109,219.3) -- (320,219.3) ;
%Straight Lines [id:da7278594086731974] 
\draw[line width=0.5pt]    (109,220.7) -- (320,220.7) ;
%Straight Lines [id:da7686133682169158] 
\draw[line width=0.5pt]    (109,309.3) -- (320,309.3) ;
%Straight Lines [id:da5783882407913247] 
\draw[line width=0.5pt]    (109,310.7) -- (320,310.7) ;
%Straight Lines [id:da2555530132097408] 
\draw    (155,162.5) ;
%Straight Lines [id:da8796155011016652] 
\draw [line width=1.95]    (168,155) ;
\draw [shift={(168,155)}, rotate = 0] [color=myblack  ][fill=myblack  ][line width=1.95]      (0, 0) circle [x radius= 3.05, y radius= 3.05]   ;
\draw [shift={(168,155)}, rotate = 0] [color=myblack  ][fill=myblack  ][line width=1.95]      (0, 0) circle [x radius= 3.05, y radius= 3.05]   ;
%Straight Lines [id:da6631004760619986] 
\draw [line width=1.95]    (227,155) ;
\draw [shift={(227,155)}, rotate = 0] [color=myblack  ][fill=myblack  ][line width=1.95]      (0, 0) circle [x radius= 3.05, y radius= 3.05]   ;
\draw [shift={(227,155)}, rotate = 0] [color=myblack  ][fill=myblack  ][line width=1.95]      (0, 0) circle [x radius= 3.05, y radius= 3.05]   ;
%Straight Lines [id:da41184634364402106] 
\draw [line width=1.95]    (287,155) ;
\draw [shift={(287,155)}, rotate = 0] [color=myblack  ][fill=myblack  ][line width=1.95]      (0, 0) circle [x radius= 3.05, y radius= 3.05]   ;
\draw [shift={(287,155)}, rotate = 0] [color=myblack  ][fill=myblack  ][line width=1.95]      (0, 0) circle [x radius= 3.05, y radius= 3.05]   ;
%Straight Lines [id:da5839684745203167] 
\draw [color=myblue  ,draw opacity=1 ][line width=1.95]    (227,155) ;
\draw [shift={(227,155)}, rotate = 0] [color=myblue  ,draw opacity=1 ][fill=myblue  ,fill opacity=1 ][line width=1.95]      (0, 0) circle [x radius= \ssr, y radius= \ssr]   ;
\draw [shift={(227,155)}, rotate = 0] [color=myblue  ,draw opacity=1 ][fill=myblue  ,fill opacity=1 ][line width=1.95]      (0, 0) circle [x radius= \ssr, y radius= \ssr]   ;
%Straight Lines [id:da705496871656988] 
\draw [color=myred  ,draw opacity=1 ][line width=1.95]    (287,155) ;
\draw [shift={(287,155)}, rotate = 0] [color=myred  ,draw opacity=1 ][fill=myred  ,fill opacity=1 ][line width=1.95]      (0, 0) circle [x radius= \ssr, y radius= \ssr]   ;
\draw [shift={(287,155)}, rotate = 0] [color=myred  ,draw opacity=1 ][fill=myred  ,fill opacity=1 ][line width=1.95]      (0, 0) circle [x radius= \ssr, y radius= \ssr]   ;
%Straight Lines [id:da4723132111038634] 
\draw[line width=0.5pt]    (109,174.3) -- (319,174.3) ;
%Straight Lines [id:da9409133003834744] 
\draw[line width=0.5pt]    (109,175.7) -- (319,175.7) ;

% Text Node
\draw (125,187) node [anchor=north west][inner sep=0.975pt]  [font=\footnotesize] [align=left] {$\displaystyle a_{1}$};
% Text Node
\draw (166,187) node [anchor=north west][inner sep=0.975pt]  [font=\footnotesize] [align=left] {$\displaystyle a_{2}$};
% Text Node
\draw (225,186) node [anchor=north west][inner sep=0.975pt]  [font=\footnotesize] [align=left] {$\displaystyle b$};
% Text Node
\draw (285,188) node [anchor=north west][inner sep=0.975pt]  [font=\footnotesize] [align=left] {$\displaystyle c$};
% Text Node
\draw (123,227) node [anchor=north west][inner sep=0.975pt]  [font=\footnotesize] [align=left] {$\displaystyle a_{11}{}$$ $};
% Text Node
\draw (163,227) node [anchor=north west][inner sep=0.975pt]  [font=\footnotesize] [align=left] {$\displaystyle a_{21}{}$$ $};
% Text Node
\draw (228,247) node [anchor=north west][inner sep=0.975pt]  [font=\footnotesize] [align=left] {$\displaystyle b$$ $};
% Text Node
\draw (288,249) node [anchor=north west][inner sep=0.975pt]  [font=\footnotesize] [align=left] {$\displaystyle c$};
% Text Node
\draw (116,283) node [anchor=north west][inner sep=0.975pt]  [font=\footnotesize] [align=left] {$\displaystyle a_{12}{}$$ $};
% Text Node
\draw (158,284) node [anchor=north west][inner sep=0.975pt]  [font=\footnotesize] [align=left] {$\displaystyle a_{22}{}$$ $};
% Text Node
\draw (125,317) node [anchor=north west][inner sep=0.975pt]  [font=\footnotesize] [align=left] {$\displaystyle a_{11}{}$$ $};
% Text Node
\draw (165,317) node [anchor=north west][inner sep=0.975pt]  [font=\footnotesize] [align=left] {$\displaystyle a_{21}{}$$ $};
% Text Node
\draw (126,374) node [anchor=north west][inner sep=0.975pt]  [font=\footnotesize] [align=left] {$\displaystyle a_{12}{}$$ $};
% Text Node
\draw (166,374) node [anchor=north west][inner sep=0.975pt]  [font=\footnotesize] [align=left] {$\displaystyle a_{22}{}$$ $};
% Text Node
\draw (232,319) node [anchor=north west][inner sep=0.975pt]  [font=\footnotesize] [align=left] {$\displaystyle b_{1}$$ $};
% Text Node
\draw (234,365) node [anchor=north west][inner sep=0.975pt]  [font=\footnotesize] [align=left] {$\displaystyle b_{2}$$ $};
% Text Node
\draw (289,339) node [anchor=north west][inner sep=0.975pt]  [font=\footnotesize] [align=left] {$\displaystyle c$$ $};
% Text Node
\draw (165,143) node [anchor=north west][inner sep=0.975pt]  [font=\footnotesize] [align=left] {$\displaystyle a$$ $};
% Text Node
\draw (224,142) node [anchor=north west][inner sep=0.975pt]  [font=\footnotesize] [align=left] {$\displaystyle b$$ $};
% Text Node
\draw (285,143) node [anchor=north west][inner sep=0.975pt]  [font=\footnotesize] [align=left] {$\displaystyle c$$ $};
% Text Node
\draw (83,150) node [anchor=north west][inner sep=0.975pt]  [font=\normalsize] [align=left] {$\displaystyle H^{( 0)}$};
% Text Node
\draw (83,190) node [anchor=north west][inner sep=0.975pt]  [font=\normalsize] [align=left] {$\displaystyle H^{( 1)}$};
% Text Node
\draw (83,252) node [anchor=north west][inner sep=0.975pt]  [font=\normalsize] [align=left] {$\displaystyle H^{( 3)}$};
% Text Node
\draw (83,333) node [anchor=north west][inner sep=0.975pt]  [font=\normalsize] [align=left] {$\displaystyle H^{( 4)}$};
\end{tikzpicture}

%% file: tikz_diagrams/fano2.tex
% Global rescaling factor -- change only this
\def\grs{2}

%\definecolor{myblue}{RGB}{0,119,255}
\definecolor{myblue}{RGB}{70,120,220}
%\definecolor{myred}{RGB}{250,50,50}
\definecolor{myred}{RGB}{240,0,0}
\definecolor{mydarkblue}{RGB}{0,107,255}

% Physical dimensions derived from the global rescaling factor
\pgfmathsetlengthmacro{\grsUnit}{0.75pt*\grs}
\pgfmathsetlengthmacro{\grsThick}{1.5pt*\grs}

\tikzset{every picture/.style={line width=\grsUnit}}

\begin{tikzpicture}[x=\grsUnit,y=\grsUnit,yscale=-1,xscale=1]
%uncomment if require: \path (0,719); %set diagram left start at 0, and has height of 719

%Straight Lines [id:da37727379361328306] 
\draw [draw=myblue]   (400,280) -- (440,280) ;
%Straight Lines [id:da9793080206604855] 
\draw [draw=myblue]   (440,280) -- (480,280) ;
%Shape: Circle [id:dp5643915601536441] 
\draw  [color=myblue  ,draw opacity=1 ] (417,257) .. controls (417,244.3) and (427.3,234) .. (440,234) .. controls (452.7,234) and (463,244.3) .. (463,257) .. controls (463,269.7) and (452.7,280) .. (440,280) .. controls (427.3,280) and (417,269.7) .. (417,257) -- cycle ;
%Straight Lines [id:da48838434819216536] 
\draw [color=myblue  ,draw opacity=1 ]   (460,245) -- (400,280) ;
%Straight Lines [id:da35952209764669496] 
\draw [color=myred  ,draw opacity=1 ]   (460,245) -- (440,210) ;
%Straight Lines [id:da7398856750002603] 
\draw [color=myred  ,draw opacity=1 ]   (480,280) -- (460,245) ;
%Straight Lines [id:da38929664809336584] 
\draw [color=myred  ,draw opacity=1 ]   (440,210) -- (440,280) ;
%Straight Lines [id:da6669272949435606] 
\draw [draw=myblue]   (420,245) -- (480,280) ;
%Straight Lines [id:da31668181370191706] 
\draw [color=myred  ,draw opacity=1 ]   (420,245) -- (400,280) ;
%Straight Lines [id:da00026957468630350867] 
\draw [color=myred  ,draw opacity=1 ]   (440,210) -- (420,245) ;

%Straight Lines [id:da08783652994698543] 
\draw [color={rgb, 255:red, 0; green, 0; blue, 0 }  ,draw opacity=1 ][line width=\grsThick]    (400,280) ;
\draw [shift={(400,280)}, rotate = 0] [color={rgb, 255:red, 0; green, 0; blue, 0 }  ,draw opacity=1 ][fill={rgb, 255:red, 0; green, 0; blue, 0 }  ,fill opacity=1 ][line width=\grsThick]      (0, 0) circle [x radius= 3.05, y radius= 3.05]   ;
\draw [shift={(400,280)}, rotate = 0] [color={rgb, 255:red, 0; green, 0; blue, 0 }  ,draw opacity=1 ][fill={rgb, 255:red, 0; green, 0; blue, 0 }  ,fill opacity=1 ][line width=\grsThick]      (0, 0) circle [x radius= 3.05, y radius= 3.05]   ;

%Straight Lines [id:da48152975740448334] 
\draw [color={rgb, 255:red, 0; green, 0; blue, 0 }  ,draw opacity=1 ][line width=\grsThick]    (420,245) ;
\draw [shift={(420,245)}, rotate = 0] [color={rgb, 255:red, 0; green, 0; blue, 0 }  ,draw opacity=1 ][fill={rgb, 255:red, 0; green, 0; blue, 0 }  ,fill opacity=1 ][line width=\grsThick]      (0, 0) circle [x radius= 3.05, y radius= 3.05]   ;
\draw [shift={(420,245)}, rotate = 0] [color={rgb, 255:red, 0; green, 0; blue, 0 }  ,draw opacity=1 ][fill={rgb, 255:red, 0; green, 0; blue, 0 }  ,fill opacity=1 ][line width=\grsThick]      (0, 0) circle [x radius= 3.05, y radius= 3.05]   ;

%Straight Lines [id:da11149692940932887] 
\draw [color={rgb, 255:red, 0; green, 0; blue, 0 }  ,draw opacity=1 ][line width=\grsThick]    (440,280) ;
\draw [shift={(440,280)}, rotate = 0] [color={rgb, 255:red, 0; green, 0; blue, 0 }  ,draw opacity=1 ][fill={rgb, 255:red, 0; green, 0; blue, 0 }  ,fill opacity=1 ][line width=\grsThick]      (0, 0) circle [x radius= 3.05, y radius= 3.05]   ;
\draw [shift={(440,280)}, rotate = 0] [color={rgb, 255:red, 0; green, 0; blue, 0 }  ,draw opacity=1 ][fill={rgb, 255:red, 0; green, 0; blue, 0 }  ,fill opacity=1 ][line width=\grsThick]      (0, 0) circle [x radius= 3.05, y radius= 3.05]   ;

%Straight Lines [id:da004679382003895616] 
\draw [color={rgb, 255:red, 0; green, 0; blue, 0 }  ,draw opacity=1 ][line width=\grsThick]    (440,257) ;
\draw [shift={(440,257)}, rotate = 0] [color={rgb, 255:red, 0; green, 0; blue, 0 }  ,draw opacity=1 ][fill={rgb, 255:red, 0; green, 0; blue, 0 }  ,fill opacity=1 ][line width=\grsThick]      (0, 0) circle [x radius= 3.05, y radius= 3.05]   ;
\draw [shift={(440,257)}, rotate = 0] [color={rgb, 255:red, 0; green, 0; blue, 0 }  ,draw opacity=1 ][fill={rgb, 255:red, 0; green, 0; blue, 0 }  ,fill opacity=1 ][line width=\grsThick]      (0, 0) circle [x radius= 3.05, y radius= 3.05]   ;

%Straight Lines [id:da8647037135356397] 
\draw [line width=\grsThick]    (460,245) ;
\draw [shift={(460,245)}, rotate = 0] [color={rgb, 255:red, 0; green, 0; blue, 0 }  ][fill={rgb, 255:red, 0; green, 0; blue, 0 }  ][line width=\grsThick]      (0, 0) circle [x radius= 3.05, y radius= 3.05]   ;
\draw [shift={(460,245)}, rotate = 0] [color={rgb, 255:red, 0; green, 0; blue, 0 }  ][fill={rgb, 255:red, 0; green, 0; blue, 0 }  ][line width=\grsThick]      (0, 0) circle [x radius= 3.05, y radius= 3.05]   ;

%Straight Lines [id:da361768479505412] 
\draw [draw=myblue][line width=\grsThick]    (460,245) ;
\draw [shift={(460,245)}, rotate = 0] [draw=myblue][fill=myblue  ,fill opacity=1 ][line width=\grsThick]      (0, 0) circle [x radius= 1.00, y radius= 1.00]   ;
\draw [shift={(460,245)}, rotate = 0] [draw=myblue][fill=myblue  ,fill opacity=1 ][line width=\grsThick]      (0, 0) circle [x radius= 1.00, y radius= 1.00]   ;

%Straight Lines [id:da002107579112746616] 
\draw [line width=\grsThick]    (480,280) ;
\draw [shift={(480,280)}, rotate = 0] [color={rgb, 255:red, 0; green, 0; blue, 0 }  ][fill={rgb, 255:red, 0; green, 0; blue, 0 }  ][line width=\grsThick]      (0, 0) circle [x radius= 3.05, y radius= 3.05]   ;
\draw [shift={(480,280)}, rotate = 0] [color={rgb, 255:red, 0; green, 0; blue, 0 }  ][fill={rgb, 255:red, 0; green, 0; blue, 0 }  ][line width=\grsThick]      (0, 0) circle [x radius= 3.05, y radius= 3.05]   ;

%Straight Lines [id:da6788320840145875] 
\draw [draw=myblue][line width=\grsThick]    (480,280) ;
\draw [shift={(480,280)}, rotate = 0] [draw=myblue][fill=myblue  ,fill opacity=1 ][line width=\grsThick]      (0, 0) circle [x radius= 1.00, y radius= 1.00]   ;
\draw [shift={(480,280)}, rotate = 0] [draw=myblue][fill=myblue  ,fill opacity=1 ][line width=\grsThick]      (0, 0) circle [x radius= 1.00, y radius= 1.00]   ;

%Straight Lines [id:da8308032347389241] 
\draw [line width=\grsThick]    (440,210) ;
\draw [shift={(440,210)}, rotate = 0] [color={rgb, 255:red, 0; green, 0; blue, 0 }  ][fill={rgb, 255:red, 0; green, 0; blue, 0 }  ][line width=\grsThick]      (0, 0) circle [x radius= 3.05, y radius= 3.05]   ;
\draw [shift={(440,210)}, rotate = 0] [color={rgb, 255:red, 0; green, 0; blue, 0 }  ][fill={rgb, 255:red, 0; green, 0; blue, 0 }  ][line width=\grsThick]      (0, 0) circle [x radius= 3.05, y radius= 3.05]   ;

%Straight Lines [id:da40231974794109193] 
\draw [color=myred  ,draw opacity=1 ][line width=\grsThick]    (440,210) ;
\draw [shift={(440,210)}, rotate = 0] [color=myred  ,draw opacity=1 ][fill=myred  ,fill opacity=1 ][line width=\grsThick]      (0, 0) circle [x radius= 1.00, y radius= 1.00]   ;
\draw [shift={(440,210)}, rotate = 0] [color=myred  ,draw opacity=1 ][fill=myred  ,fill opacity=1 ][line width=\grsThick]      (0, 0) circle [x radius= 1.00, y radius= 1.00]   ;

% Text Node
\draw (385,281) node [anchor=north west][inner sep=\grsUnit] [font=\normalsize] [align=left] {$\displaystyle a_{12}{}$$ $};
% Text Node
\draw (440,261) node [anchor=north west][inner sep=\grsUnit] [font=\normalsize] [align=left] {$\displaystyle a_{21}{}$$ $};
% Text Node
\draw (442,283) node [anchor=north west][inner sep=\grsUnit] [font=\normalsize] [align=left] {$\displaystyle a_{22}{}$$ $};
% Text Node
\draw (406,234) node [anchor=north west][inner sep=\grsUnit] [font=\normalsize] [align=left] {$\displaystyle a_{11}{}$$ $};
% Text Node
\draw (462,235) node [anchor=north west][inner sep=\grsUnit] [font=\normalsize] [align=left] {$\displaystyle b_{1}$$ $};
% Text Node
\draw (484,274) node [anchor=north west][inner sep=\grsUnit] [font=\normalsize] [align=left] {$\displaystyle b_{2}$$ $};
% Text Node
\draw (444,201) node [anchor=north west][inner sep=\grsUnit] [font=\normalsize] [align=left] {$\displaystyle c$$ $};

\end{tikzpicture}

%% file: tikz_diagrams/H2laminar.tex
\tikzset{every picture/.style={line width=0.75pt}} %set default line width to 0.75pt        

\definecolor{myblue}{RGB}{70,120,220}
\definecolor{myred}{RGB}{230,0,0}
\definecolor{myblack}{RGB}{0,0,0}

\begin{tikzpicture}[x=0.75pt,y=0.75pt,yscale=-1,xscale=1]
%uncomment if require: \path (0,719); %set diagram left start at 0, and has height of 719

%Straight Lines [id:da08783652994698543] 
\draw [color=myblack  ,draw opacity=1 ][line width=1.5]    (270,80) ;
\draw [shift={(270,80)}, rotate = 0] [color=myblack  ,draw opacity=1 ][fill=myblack  ,fill opacity=1 ][line width=1.5]      (0, 0) circle [x radius= 5.23, y radius= 5.23]   ;
\draw [shift={(270,80)}, rotate = 0] [color=myblack  ,draw opacity=1 ][fill=myblack  ,fill opacity=1 ][line width=1.5]      (0, 0) circle [x radius= 5.23, y radius= 5.23]   ;
%Straight Lines [id:da25344639433607785] 
\draw [color=myblack  ,draw opacity=1 ][line width=1.5]    (230,140) ;
\draw [shift={(230,140)}, rotate = 0] [color=myblack  ,draw opacity=1 ][fill=myblack  ,fill opacity=1 ][line width=1.5]      (0, 0) circle [x radius= 5.23, y radius= 5.23]   ;
\draw [shift={(230,140)}, rotate = 0] [color=myblack  ,draw opacity=1 ][fill=myblack  ,fill opacity=1 ][line width=1.5]      (0, 0) circle [x radius= 5.23, y radius= 5.23]   ;
%Straight Lines [id:da00891814429355342] 
\draw [color=myred  ,draw opacity=1 ][line width=2.25]    (350,200) -- (238,200) -- (190,200) ;
%Straight Lines [id:da8274955956068977] 
\draw [color=myblack  ,draw opacity=1 ]   (270,80) -- (350,200) ;
%Straight Lines [id:da5147725172810116] 
\draw [color=myblack  ,draw opacity=1 ]   (230,140) -- (270,200) ;
%Straight Lines [id:da11921487854512736] 
\draw [color=myblack  ,draw opacity=1 ]   (270,80) -- (230,140) ;
%Straight Lines [id:da049632407718946725] 
\draw [color=myblack  ,draw opacity=1 ]   (230,140) -- (190,200) ;
%Straight Lines [id:da4189770445541666] 
\draw [color=myblack  ,draw opacity=1 ]   (270,200) -- (300,261) ;
%Straight Lines [id:da4931963698777705] 
\draw [color=myblack  ,draw opacity=1 ]   (270,200) -- (240,261) ;
%Curve Lines [id:da45460062398565404] 
\draw [color=myblue  ,draw opacity=1 ][line width=2.25]    (240,261) .. controls (295.89,278.01) and (324.47,251.84) .. (338.3,228.68) .. controls (347.23,213.71) and (350,200) .. (350,200) ;
%Straight Lines [id:da1815297505536715] 
\draw [color=myblack  ,draw opacity=1 ][line width=1.5]    (190,200) ;
\draw [shift={(190,200)}, rotate = 0] [color=myblack  ,draw opacity=1 ][fill=myblack  ,fill opacity=1 ][line width=1.5]      (0, 0) circle [x radius= 5.23, y radius= 5.23]   ;
\draw [shift={(190,200)}, rotate = 0] [color=myblack  ,draw opacity=1 ][fill=myblack  ,fill opacity=1 ][line width=1.5]      (0, 0) circle [x radius= 5.23, y radius= 5.23]   ;
%Straight Lines [id:da6789230863057674] 
\draw [color=myblack  ,draw opacity=1 ][line width=1.5]    (350,200) ;
\draw [shift={(350,200)}, rotate = 0] [color=myblack  ,draw opacity=1 ][fill=myblack  ,fill opacity=1 ][line width=1.5]      (0, 0) circle [x radius= 5.23, y radius= 5.23]   ;
\draw [shift={(350,200)}, rotate = 0] [color=myblack  ,draw opacity=1 ][fill=myblack  ,fill opacity=1 ][line width=1.5]      (0, 0) circle [x radius= 5.23, y radius= 5.23]   ;
%Straight Lines [id:da04143807907959607] 
\draw [color=myblack  ,draw opacity=1 ][line width=1.5]    (300,261) ;
\draw [shift={(300,261)}, rotate = 0] [color=myblack  ,draw opacity=1 ][fill=myblack  ,fill opacity=1 ][line width=1.5]      (0, 0) circle [x radius= 5.23, y radius= 5.23]   ;
\draw [shift={(300,261)}, rotate = 0] [color=myblack  ,draw opacity=1 ][fill=myblack  ,fill opacity=1 ][line width=1.5]      (0, 0) circle [x radius= 5.23, y radius= 5.23]   ;
%Straight Lines [id:da34112719769339894] 
\draw [color=myblack  ,draw opacity=1 ][line width=1.5]    (240,261) ;
\draw [shift={(240,261)}, rotate = 0] [color=myblack  ,draw opacity=1 ][fill=myblack  ,fill opacity=1 ][line width=1.5]      (0, 0) circle [x radius= 5.23, y radius= 5.23]   ;
\draw [shift={(240,261)}, rotate = 0] [color=myblack  ,draw opacity=1 ][fill=myblack  ,fill opacity=1 ][line width=1.5]      (0, 0) circle [x radius= 5.23, y radius= 5.23]   ;
%Straight Lines [id:da9671969493607274] 
\draw [color=myblack  ,draw opacity=1 ][line width=1.5]    (270,200) ;
\draw [shift={(270,200)}, rotate = 0] [color=myblack  ,draw opacity=1 ][fill=myblack  ,fill opacity=1 ][line width=1.5]      (0, 0) circle [x radius= 5.23, y radius= 5.23]   ;
\draw [shift={(270,200)}, rotate = 0] [color=myblack  ,draw opacity=1 ][fill=myblack  ,fill opacity=1 ][line width=1.5]      (0, 0) circle [x radius= 5.23, y radius= 5.23]   ;

% Text Node
\draw (174,175) node [anchor=north west][inner sep=0.75pt]   [align=left] {$\displaystyle \{z\}$};
% Text Node
\draw (264,174.5) node [anchor=north west][inner sep=0.75pt]   [align=left] {$\displaystyle \{x,a,c\}$};
% Text Node
\draw (345,174.5) node [anchor=north west][inner sep=0.75pt]   [align=left] {$\displaystyle \{y,b\}$};
% Text Node
\draw (279,69.5) node [anchor=north west][inner sep=0.75pt]  [font=\footnotesize] [align=left] {$\displaystyle V=\{x,y,z,a,b,c\}$};
% Text Node
\draw (223,236.5) node [anchor=north west][inner sep=0.75pt]   [align=left] {$\displaystyle \{a\}$};
% Text Node
\draw (294,234.5) node [anchor=north west][inner sep=0.75pt]   [align=left] {$\displaystyle \{c\}$};
% Text Node
\draw (175,118.5) node [anchor=north west][inner sep=0.75pt]  [font=\footnotesize] [align=left] {$\displaystyle \{x,z,a,c\}$};

\end{tikzpicture}

%% file: tikz_diagrams/chainlaminar.tex
\tikzset{every picture/.style={line width=0.75pt}} %set default line width to 0.75pt        

\definecolor{myblue}{RGB}{70,120,220}
\definecolor{myred}{RGB}{240,0,0}
\definecolor{mygreen}{RGB}{50,150,40}
\definecolor{myblack}{RGB}{0,0,0}

\begin{tikzpicture}[x=0.75pt,y=0.75pt,yscale=-1,xscale=1]
%uncomment if require: \path (0,719); %set diagram left start at 0, and has height of 719

%Curve Lines [id:da13260009433393305] 
\draw [color=myblue  ,draw opacity=1 ][line width=2.25]    (310,310) .. controls (310,310) and (261,392) .. (210,370.5) ;
%Curve Lines [id:da19752793494845344] 
\draw [color=myred  ,draw opacity=1 ][line width=2.25]    (390,310) .. controls (390,310) and (341,392) .. (290,370.5) ;
%Curve Lines [id:da33608854252458054] 
\draw [color=mygreen  ,draw opacity=1 ][line width=2.25]    (440,347.5) .. controls (438,350.5) and (413,382.5) .. (370,370.5) ;
\draw [shift={(440,347.5)}, rotate = 303.69] [color=mygreen  ,draw opacity=1 ][line width=2.25]    (0,7.83) -- (0,-7.83)   ;
%Curve Lines [id:da13506006986727515] 
\draw [color=mygreen  ,draw opacity=1 ][line width=2.25]    (230,310) .. controls (230,311.5) and (213,336.5) .. (209,342.5) ;
\draw [shift={(209,342.5)}, rotate = 303.69] [color=mygreen  ,draw opacity=1 ][line width=2.25]    (0,7.83) -- (0,-7.83)   ;
%Straight Lines [id:da5806423982200006] 
\draw [color=myblack  ,draw opacity=1 ][line width=1.5]    (310,270) ;
\draw [shift={(310,270)}, rotate = 0] [color=myblack  ,draw opacity=1 ][fill=myblack  ,fill opacity=1 ][line width=1.5]      (0, 0) circle [x radius= 5.23, y radius= 5.23]   ;
\draw [shift={(310,270)}, rotate = 0] [color=myblack  ,draw opacity=1 ][fill=myblack  ,fill opacity=1 ][line width=1.5]      (0, 0) circle [x radius= 5.23, y radius= 5.23]   ;
%Straight Lines [id:da1383469865163247] 
\draw [color=myblack  ,draw opacity=1 ][line width=1.5]    (230,310) ;
\draw [shift={(230,310)}, rotate = 0] [color=myblack  ,draw opacity=1 ][fill=myblack  ,fill opacity=1 ][line width=1.5]      (0, 0) circle [x radius= 5.23, y radius= 5.23]   ;
\draw [shift={(230,310)}, rotate = 0] [color=myblack  ,draw opacity=1 ][fill=myblack  ,fill opacity=1 ][line width=1.5]      (0, 0) circle [x radius= 5.23, y radius= 5.23]   ;
%Straight Lines [id:da024543640357075125] 
\draw [color=myblack  ,draw opacity=1 ]   (310,270) -- (310,310) ;
%Straight Lines [id:da2578875368104304] 
\draw [color=myblack  ,draw opacity=1 ]   (310,270) -- (230,310) ;
%Straight Lines [id:da2710077682598493] 
\draw [color=myblack  ,draw opacity=1 ]   (230,310) -- (250,370) ;
%Straight Lines [id:da5909476060704838] 
\draw [color=myblack  ,draw opacity=1 ]   (230,310) -- (210,370.5) ;
%Straight Lines [id:da2522931279228178] 
\draw [color=myblack  ,draw opacity=1 ][line width=1.5]    (330,370) ;
\draw [shift={(330,370)}, rotate = 0] [color=myblack  ,draw opacity=1 ][fill=myblack  ,fill opacity=1 ][line width=1.5]      (0, 0) circle [x radius= 5.23, y radius= 5.23]   ;
\draw [shift={(330,370)}, rotate = 0] [color=myblack  ,draw opacity=1 ][fill=myblack  ,fill opacity=1 ][line width=1.5]      (0, 0) circle [x radius= 5.23, y radius= 5.23]   ;
%Straight Lines [id:da34914426703210355] 
\draw [color=myblack  ,draw opacity=1 ][line width=1.5]    (250,370) ;
\draw [shift={(250,370)}, rotate = 0] [color=myblack  ,draw opacity=1 ][fill=myblack  ,fill opacity=1 ][line width=1.5]      (0, 0) circle [x radius= 5.23, y radius= 5.23]   ;
\draw [shift={(250,370)}, rotate = 0] [color=myblack  ,draw opacity=1 ][fill=myblack  ,fill opacity=1 ][line width=1.5]      (0, 0) circle [x radius= 5.23, y radius= 5.23]   ;
%Straight Lines [id:da24662962124710008] 
\draw [color=myblack  ,draw opacity=1 ][line width=1.5]    (210,370.5) ;
\draw [shift={(210,370.5)}, rotate = 0] [color=myblack  ,draw opacity=1 ][fill=myblack  ,fill opacity=1 ][line width=1.5]      (0, 0) circle [x radius= 5.23, y radius= 5.23]   ;
\draw [shift={(210,370.5)}, rotate = 0] [color=myblack  ,draw opacity=1 ][fill=myblack  ,fill opacity=1 ][line width=1.5]      (0, 0) circle [x radius= 5.23, y radius= 5.23]   ;
%Straight Lines [id:da007828214668548106] 
\draw [color=myblack  ,draw opacity=1 ][line width=1.5]    (290,370.5) ;
\draw [shift={(290,370.5)}, rotate = 0] [color=myblack  ,draw opacity=1 ][fill=myblack  ,fill opacity=1 ][line width=1.5]      (0, 0) circle [x radius= 5.23, y radius= 5.23]   ;
\draw [shift={(290,370.5)}, rotate = 0] [color=myblack  ,draw opacity=1 ][fill=myblack  ,fill opacity=1 ][line width=1.5]      (0, 0) circle [x radius= 5.23, y radius= 5.23]   ;
%Straight Lines [id:da5316818530479762] 
\draw [color=myblack  ,draw opacity=1 ]   (310,310) -- (330,370) ;
%Straight Lines [id:da6371417777885994] 
\draw [color=myblack  ,draw opacity=1 ]   (310,310) -- (290,370.5) ;
%Straight Lines [id:da4369413341822853] 
\draw [color=myblack  ,draw opacity=1 ][line width=1.5]    (310,310) ;
\draw [shift={(310,310)}, rotate = 0] [color=myblack  ,draw opacity=1 ][fill=myblack  ,fill opacity=1 ][line width=1.5]      (0, 0) circle [x radius= 5.23, y radius= 5.23]   ;
\draw [shift={(310,310)}, rotate = 0] [color=myblack  ,draw opacity=1 ][fill=myblack  ,fill opacity=1 ][line width=1.5]      (0, 0) circle [x radius= 5.23, y radius= 5.23]   ;
%Straight Lines [id:da4374638784283368] 
\draw [color=myblack  ,draw opacity=1 ]   (390,310) -- (310,270) ;
%Straight Lines [id:da5541149840003711] 
\draw [color=myblack  ,draw opacity=1 ]   (390,310) -- (410,370) ;
%Straight Lines [id:da1108576514264612] 
\draw [color=myblack  ,draw opacity=1 ]   (390,310) -- (370,370.5) ;
%Straight Lines [id:da9423764595002118] 
\draw [color=myblack  ,draw opacity=1 ][line width=1.5]    (390,310) ;
\draw [shift={(390,310)}, rotate = 0] [color=myblack  ,draw opacity=1 ][fill=myblack  ,fill opacity=1 ][line width=1.5]      (0, 0) circle [x radius= 5.23, y radius= 5.23]   ;
\draw [shift={(390,310)}, rotate = 0] [color=myblack  ,draw opacity=1 ][fill=myblack  ,fill opacity=1 ][line width=1.5]      (0, 0) circle [x radius= 5.23, y radius= 5.23]   ;
%Straight Lines [id:da46356962925050593] 
\draw [color=myblack  ,draw opacity=1 ][line width=1.5]    (410,370) ;
\draw [shift={(410,370)}, rotate = 0] [color=myblack  ,draw opacity=1 ][fill=myblack  ,fill opacity=1 ][line width=1.5]      (0, 0) circle [x radius= 5.23, y radius= 5.23]   ;
\draw [shift={(410,370)}, rotate = 0] [color=myblack  ,draw opacity=1 ][fill=myblack  ,fill opacity=1 ][line width=1.5]      (0, 0) circle [x radius= 5.23, y radius= 5.23]   ;
%Straight Lines [id:da7578814443473619] 
\draw [color=myblack  ,draw opacity=1 ][line width=1.5]    (370,370.5) ;
\draw [shift={(370,370.5)}, rotate = 0] [color=myblack  ,draw opacity=1 ][fill=myblack  ,fill opacity=1 ][line width=1.5]      (0, 0) circle [x radius= 5.23, y radius= 5.23]   ;
\draw [shift={(370,370.5)}, rotate = 0] [color=myblack  ,draw opacity=1 ][fill=myblack  ,fill opacity=1 ][line width=1.5]      (0, 0) circle [x radius= 5.23, y radius= 5.23]   ;

% Text Node
\draw (318,257.5) node [anchor=north west][inner sep=0.75pt]  [font=\footnotesize] [align=left] {$\displaystyle V$};
% Text Node
\draw (194,378.5) node [anchor=north west][inner sep=0.75pt]   [align=left] {$\displaystyle \{a_{0}\}$};
% Text Node
\draw (235,378.5) node [anchor=north west][inner sep=0.75pt]   [align=left] {$\displaystyle \{b_{0}\}$};
% Text Node
\draw (276,378.5) node [anchor=north west][inner sep=0.75pt]   [align=left] {$\displaystyle \{a_{1}\}$};
% Text Node
\draw (316,378.5) node [anchor=north west][inner sep=0.75pt]   [align=left] {$\displaystyle \{b_{1}\}$};
% Text Node
\draw (356,377.5) node [anchor=north west][inner sep=0.75pt]   [align=left] {$\displaystyle \{a_{2}\}$};
% Text Node
\draw (396,377.5) node [anchor=north west][inner sep=0.75pt]   [align=left] {$\displaystyle \{b_{2}\}$};
% Text Node
\draw (182,305.5) node [anchor=north west][inner sep=0.75pt]  [font=\footnotesize] [align=left] {$\displaystyle \{a_{0} ,b_{0}\}$};
% Text Node
\draw (262,305.5) node [anchor=north west][inner sep=0.75pt]  [font=\footnotesize] [align=left] {$\displaystyle \{a_{1} ,b_{1}\}$};
% Text Node
\draw (342,305.5) node [anchor=north west][inner sep=0.75pt]  [font=\footnotesize] [align=left] {$\displaystyle \{a_{2} ,b_{2}\}$};

\end{tikzpicture}

%% file: Compilation_Files/biblio.tex
% This bibliography was generated by Logan's biblio_maker.py

% PLEASE DO NOT MODIFY this text directly until the final draft. All changes will be REVERTED when the program is refreshed.

% Please modify reference entries in references.bib. To add a citation directly, use @directbib{INSERT_TAG, biboutput={INPUT_HERE} }